\documentclass{amsart}

\usepackage[a4paper, margin=2.5cm]{geometry}
\usepackage[foot]{amsaddr}

\usepackage[T1]{fontenc}
\usepackage{graphicx}%
\usepackage{subcaption}
\usepackage{multirow}%
\usepackage{amsmath,amssymb,amsfonts}%
\usepackage{amsthm}%
\usepackage{mathrsfs}%
\usepackage{mathtools}%
\usepackage[title]{appendix}%
\usepackage{xcolor}%
\usepackage{textcomp}%
\usepackage{manyfoot}%
\usepackage{booktabs}%
\usepackage{listings}%
\usepackage{nicematrix}
\usepackage{microtype}
\usepackage[scr=boondox,  
            cal=esstix]   
           {mathalpha}

\usepackage{tikz}
\usetikzlibrary{tikzmark}
\usetikzlibrary{patterns, calc}
\usetikzlibrary{positioning,matrix,fit}
\usetikzlibrary{shapes,arrows}
\usetikzlibrary{decorations.pathreplacing,angles,quotes}
\usetikzlibrary{shapes.geometric, arrows}
 \usetikzlibrary{decorations.pathmorphing}
\usepackage{tikz-cd}

\tikzset{
  commutative diagrams/.cd, 
  arrow style=tikz, 
  diagrams={>=stealth},
  ampersand replacement=\&
}

\usepackage{hyperref}    
\usepackage[nameinlink]{cleveref}    

\def\R{\mathbb{R}}
\newcommand{\longto}{\longrightarrow}

\DeclareMathOperator{\curl}{curl}

\def\divgn{\operatorname{div}}

\newcommand{\ainnerproduct}[2]{\langle #1, #2 \rangle}
\newcommand{\aInnerproduct}[2]{\bigl\langle #1, #2 \bigr\rangle}
\DeclarePairedDelimiterX{\norm}[1]{\lVert}{\rVert}{#1}

\makeatletter
\@namedef{subjclassname@2020}{AMS Subject Classification}
\makeatother

\allowdisplaybreaks

\newtheorem{theorem}{Theorem}

\numberwithin{equation}{section}

\usepackage[backend=biber, style=alphabetic, backref, url=false, eprint=false]{biblatex}
\begin{document}

\title[Higher-Order Implicit Scheme for Maxwell's]{An Arbitrarily Higher-Order Time Implicit Scheme for Maxwell's Equations}

\author{Archana Arya and Kaushik Kalyanaraman}
\address{Department of Mathematics, Indraprastha Institute of Information Technology, Delhi, New Delhi, 110020, India}
\email{archanaa@iiitd.ac.in, kaushik@iiitd.ac.in}

\date{}

\begin{abstract}
  We propose an arbitrarily higher (even) order implicit leapfrog scheme for time discretization of a three-field formulation of Maxwell's equations. We use this in conjunction with an arbitrarily higher-order and compatible discretization using finite element spaces that form a de Rham complex. In doing so, we provide a generalization of an earlier work building from~\cite{ArKa2025} and~\cite{ArKa2026}. We prove stability, demonstrate energy conservation, and characterize the asymptotic convergence of the error for the time semidiscretization as well as for the full spatial and temporal discretization of this Maxwell's system. We also provide some numerical validation using computational examples in $\mathbb{R}^2$.
  
\end{abstract}



\maketitle

\section{Introduction} \label{sec:introduction}

\subsection{Three-field Formulation of Maxwell's Equations}

We begin by providing a three field $(p, E, H)$ formulation of the Maxwell's equations as below in which in addition to the electric and magnetic fields $E$ and $H$, a fictitious electrical pressure allows for an exact enforcement of $\divgn E$. Thus, we have: \begin{subequations}
  \begin{equation}
    \left. \begin{aligned}
      \dfrac{\partial p}{\partial t} + \nabla \cdot \varepsilon E & = 0, \\
      \nabla p + \varepsilon \dfrac{\partial E}{\partial t} - \nabla \times H &= 0, \\
      \mu \dfrac{\partial H}{\partial t} + \nabla \times E &= 0,
    \end{aligned} \right\} \text{ in } \Omega \times (0, T], \label{eqn:maxwells_eqns}
  \end{equation}
  where $\Omega \subset \R^2/\R^3$ is a domain with Lipschitz boundary $\partial \Omega$, and $T > 0$. We use the following homogeneous boundary conditions:
  \begin{equation}
    p = 0,  E \times n = 0, H \cdot n = 0 \text{ on } \partial\Omega \times (0, T], \label{eqn:BCs}
  \end{equation}
  where $n$ is the unit outward normal to $\partial \Omega$. We wish to state that the homogeneous boundary conditions are merely for convenience and all of our work can be extended to non homogeneous boundary conditions in a very standard way. Finally, we use the following initial conditions:
  \begin{equation}
    p(x, 0) = p_0(x), E(x, 0) = E_0(x), \text{ and } H(x, 0) = H_0(x) \text{~for~} x \in \Omega, \label{eqn:ICs}
  \end{equation}
\end{subequations}
and in which for the sake of consistency as to why $p$ was introduced, the initial conditions are required to also satisfy $\nabla \cdot (\varepsilon E_0) = p_0 $ while physics provides for $\nabla \cdot (\mu H_0) = 0$ in $\Omega$. Given this setup, we also wish to add that it is possible to extend our method and analysis to nonzero right hand sides in Equation~\eqref{eqn:maxwells_eqns} but it makes the already complicated analysis quite rather unwieldy and also requires more hypotheses on these forcing functions. As a tradeoff, we realized that this sort of absolute generalization does not lend itself to gaining any meaningful insights beyond what we can already derive from our work. Therefore, while it is possible to work with a specific instance of a higher order scheme with forcing functions, we leave this altogether out of our discussion.

We next cast Equations~\eqref{eqn:maxwells_eqns} to~\eqref{eqn:ICs} into the following variational formulation: for $t \in (0, T]$, find $(p(t), E(t), H(t)) \in \mathring{H}^1_{\varepsilon^{-1}}(\Omega) \times \mathring{H}_{\varepsilon}(\curl; \Omega) \times \mathring{H}_{\mu}(\divgn; \Omega)$ such that: \begin{subequations}
  \begin{alignat}{2}
    \aInnerproduct{\dfrac{\partial p}{\partial t}}{\widetilde{p}} - \aInnerproduct{ \varepsilon E}{\nabla \widetilde{p}} &=0, &&\quad \widetilde{p} \in \mathring{H}^1_{\varepsilon^{-1}}(\Omega), \label{eqn:maxwell_p_wf} \\
    \aInnerproduct{\nabla p}{\widetilde{E}} + \aInnerproduct{\varepsilon \dfrac{\partial E}{\partial t}}{\widetilde{E}} - \aInnerproduct{H}{\nabla \times \widetilde{E}} &= 0, &&\quad \widetilde{E} \in \mathring{H}_{\varepsilon}(\curl; \Omega), \label{eqn:maxwell_E_wf} \\
    \aInnerproduct{\mu \dfrac{\partial H}{\partial t}}{\widetilde{H}} + \aInnerproduct{\nabla \times E}{\widetilde{H}}, &= 0, &&\quad \widetilde{H} \in \mathring{H}_{\mu}(\divgn; \Omega), \label{eqn:maxwell_H_wf}
  \end{alignat}
\end{subequations}
For the sake of analysis, we require that the solutions $(p, E, H)$ have sufficient regularity. Given this background on our system of Maxwell's equations, we next state our higher order discretization by first providing an illustrative example for the sixth-order scheme, and then in the full generality for the $R^{\text{th}}$-order scheme.

\subsection{Warmup Example: LF$_6$ Scheme} Find $(p(t^{n + 1/2}), E(t^{n + 1/2}), H(t^n)) \in \mathring{H}^1_{\varepsilon^{-1}}(\Omega) \times \mathring{H}_{\varepsilon}(\curl; \Omega) \times \mathring{H}_{\mu}(\divgn; \Omega)$ such that:
\begin{subequations}
\begin{multline}
  \aInnerproduct{\dfrac{p^{n + \frac{1}{2}} - p^{n - \frac{1}{2}}}{\Delta t}}{\widetilde{p}} - \aInnerproduct{\dfrac{\varepsilon}{2} \left( E^{n + \frac{1}{2}} + E^{n - \frac{1}{2}} \right)}{\nabla \widetilde{p}} + \dfrac{\Delta t^2}{12} \aInnerproduct{\dfrac{\varepsilon}{2} \nabla \nabla \cdot \left( E^{n + \frac{1}{2}} + E^{n - \frac{1}{2}} \right)}{\nabla \widetilde{p}} \\ - \dfrac{\Delta t^4}{120} \aInnerproduct{\dfrac{\varepsilon}{2} \left( \nabla \nabla \cdot \right)^2 \left( E^{n + \frac{1}{2}} + E^{n - \frac{1}{2}} \right)}{\nabla \widetilde{p}}=0, \label{eqn:maxwell_p_lf6}
  \end{multline} \\
  \begin{multline}
   \aInnerproduct{\dfrac{1}{2} \nabla \left(p^{n + \frac{1}{2}} + p^{n - \frac{1}{2}} \right)}{\widetilde{E}} - \dfrac{\Delta t^2}{12} \aInnerproduct{\dfrac{1}{2} \nabla \left(p^{n + \frac{1}{2}} + p^{n - \frac{1}{2}} \right)}{\nabla \nabla \cdot \widetilde{E}} + \dfrac{\Delta t^4}{120} \aInnerproduct{\dfrac{1}{2} \nabla \left(p^{n + \frac{1}{2}} + p^{n - \frac{1}{2}} \right)}{\left( \nabla \nabla \cdot \right)^2 \widetilde{E}} \\ + \aInnerproduct{\varepsilon \dfrac{E^{n + \frac{1}{2}} - E^{n - \frac{1}{2}}}{\Delta t}}{\widetilde{E}} - \aInnerproduct{\dfrac{1}{2} \left( H^{n + 1} + H^n \right)}{\nabla \times \widetilde{E}}  - \dfrac{\Delta t^2}{12} \aInnerproduct{\dfrac{1}{2} \mu^{-1}\varepsilon^{-1} \nabla \times \nabla \times \left( H^{n + 1} + H^n \right)}{\nabla \times \widetilde{E}} \\ - \dfrac{\Delta t^4}{120} \aInnerproduct{\dfrac{1}{2} \mu^{-2}\varepsilon^{-2} \left( \nabla \times \nabla \times \right)^2 \left( H^{n + 1} + H^n \right)}{\nabla \times \widetilde{E}} = 0, \label{eqn:maxwell_E_lf6} 
     \end{multline} \\
     \begin{multline}
  \aInnerproduct{\mu \dfrac{H^{n + 1} - H^n}{\Delta t}}{\widetilde{H}} +  \aInnerproduct{\dfrac{1}{2} \nabla \times \left( E^{n + \frac{1}{2}} + E^{n - \frac{1}{2}} \right)}{\widetilde{H}} +  \dfrac{\Delta t^2}{12} \aInnerproduct{\dfrac{1}{2} \varepsilon^{-1}\mu^{-1} \nabla \times \left( E^{n + \frac{1}{2}} + E^{n - \frac{1}{2}} \right)}{\nabla \times  \nabla \times \widetilde{H}} \\ + \dfrac{\Delta t^4}{120} \aInnerproduct{\dfrac{1}{2} \varepsilon^{-2}\mu^{-2} \nabla \times \left( E^{n + \frac{1}{2}} + E^{n - \frac{1}{2}} \right)}{\left( \nabla \times  \nabla \times \right)^2 \widetilde{H}} = 0, \label{eqn:maxwell_H_lf6}
  \end{multline}
\end{subequations}
for all $\widetilde{p} \in \mathring{H}^1_{\varepsilon^{-1}}(\Omega), \; \widetilde{E} \in \mathring{H}_{\varepsilon}(\curl; \Omega)$ and $\widetilde{H} \in \mathring{H}_{\mu}(\divgn; \Omega)$. In this description, $\left( \nabla \nabla \cdot \right)^2$ stands for the operator $\left( \nabla \nabla \cdot \nabla \nabla \cdot \right)$ and likewise $\left( \nabla \times  \nabla \times \right)^2$ denotes $\left( \nabla \times  \nabla \times \nabla \times  \nabla \times \right)$. In order to bootstrap the computations, we will use the following scheme for the first half time step for $p$ and $E$, and for the first time step for $H$:
\begin{subequations}
  \begin{multline}
    \aInnerproduct{\dfrac{p^{\frac{1}{2}} - p_0}{\Delta t/2}}{\widetilde{p}} - \dfrac{1}{2} \aInnerproduct{\dfrac{\varepsilon}{2} \left( E^{\frac{1}{2}} + E_0 \right)}{\nabla \widetilde{p}} + \dfrac{1}{8} \cdot \dfrac{\Delta t^2}{12} \aInnerproduct{\dfrac{\varepsilon}{2} \nabla \nabla \cdot \left( E^{\frac{1}{2}} + E_0 \right)}{\nabla \widetilde{p}} \\ - \dfrac{1}{32} \cdot \dfrac{\Delta t^4}{120} \aInnerproduct{\dfrac{\varepsilon}{2} \left( \nabla \nabla \cdot \right)^2 \left( E^{\frac{1}{2}} + E_0 \right)}{\nabla \widetilde{p}}  = 0, \label{eqn:maxwell_p0_lf6}
    \end{multline} \\
    \begin{multline}
  \dfrac{1}{2}  \aInnerproduct{\dfrac{1}{2} \nabla \left(  p^{\frac{1}{2}} +  p_0 \right)}{\widetilde{E}} - \dfrac{1}{8} \cdot \dfrac{\Delta t^2}{12} \aInnerproduct{\dfrac{1}{2} \nabla \left(  p^{\frac{1}{2}} +  p_0 \right)}{\nabla \nabla \cdot \widetilde{E}} + \dfrac{1}{32} \cdot \dfrac{\Delta t^4}{120} \aInnerproduct{\dfrac{1}{2} \nabla \left(  p^{\frac{1}{2}} +  p_0 \right)}{\left( \nabla \nabla \cdot \right)^2 \widetilde{E}} \\ + \aInnerproduct{\varepsilon \dfrac{E^{\frac{1}{2}} - E_0}{\Delta t/2}}{\widetilde{E}}  - \aInnerproduct{\dfrac{1}{2} \left( H^1 + H_0 \right)}{\nabla \times \widetilde{E}} -\dfrac{1}{4} \cdot \dfrac{\Delta t^2}{12}  \aInnerproduct{\dfrac{1}{2} \mu^{-1} \varepsilon^{-1} \nabla \times \nabla \times \left( H^1 + H_0 \right)}{\nabla \times \widetilde{E}} \\ + \dfrac{1}{16} \cdot \dfrac{\Delta t^4}{120}  \aInnerproduct{\dfrac{1}{2} \mu^{-2} \varepsilon^{-2} \left( \nabla \times \nabla \times \right)^2 \left( H^1 + H_0 \right)}{\nabla \times \widetilde{E}} = 0, \label{eqn:maxwell_E0_lf6} 
  \end{multline} \\
  \begin{multline}
    \aInnerproduct{\mu \dfrac{H^1 - H_0}{\Delta t}}{\widetilde{H}} + \dfrac{1}{2} \aInnerproduct{\dfrac{1}{2} \nabla \times \left(E^{\frac{1}{2}} + E_0 \right)}{\widetilde{H}} + \dfrac{1}{8} \cdot \dfrac{\Delta t^2}{12} \aInnerproduct{\dfrac{1}{2} \varepsilon^{-1} \mu^{-1} \nabla \times \left(E^{\frac{1}{2}} + E_0 \right)}{ \nabla \times \nabla \times  \widetilde{H}} \\ + \dfrac{1}{32} \cdot \dfrac{\Delta t^4}{120} \aInnerproduct{\dfrac{1}{2} \varepsilon^{-2} \mu^{-2} \nabla \times \left(E^{\frac{1}{2}} + E_0 \right)}{\left( \nabla \times \nabla \times \right)^2 \widetilde{H}} = 0. \label{eqn:maxwell_H0_lf6}
    \end{multline}
\end{subequations}
We now rewrite LF$_6$ using some additional notation with summations as next. This is so as to make it easier to then state the arbitrarily higher-order scheme LF$_R$.

\subsection{LF$_6$ Scheme in a Sum Notation} Find $(p(t^{n + 1/2}), E(t^{n + 1/2}), H(t^n)) \in \mathring{H}^1_{\varepsilon^{-1}}(\Omega) \times \mathring{H}_{\varepsilon}(\curl; \Omega) \times \mathring{H}_{\mu}(\divgn; \Omega)$ such that: 
\begin{subequations} 
\begin{multline}
  \aInnerproduct{\dfrac{p^{n + \frac{1}{2}} - p^{n - \frac{1}{2}}}{\Delta t}}{\widetilde{p}} - \aInnerproduct{\dfrac{\varepsilon}{2} \left( E^{n + \frac{1}{2}} + E^{n - \frac{1}{2}} \right)}{\nabla \widetilde{p}} \\ + \sum\limits_{\mathfrak{a} = 1}^{\frac{6}{2} - 1} \sum\limits_{k_\mathfrak{a} = 1}^{\frac{6}{2} - 1 - \sum\limits_{\mathfrak{b} = 1}^{\mathfrak{a} - 1} k_\mathfrak{b}} C_{k_1} \mathfrak{f}(\mathfrak{a}) \aInnerproduct{\dfrac{\varepsilon}{2} (\nabla \nabla \cdot)^{\sum\limits_{\mathfrak{m} = 1}^\mathfrak{a} k_\mathfrak{m}} \left( E^{n + \frac{1}{2}} + E^{n - \frac{1}{2}} \right)}{\nabla \widetilde{p}} =0, \label{eqn:maxwell_p_lf6_general}
  \end{multline} \\
  \begin{multline}
   \aInnerproduct{\dfrac{1}{2} \nabla \left(p^{n + \frac{1}{2}} + p^{n - \frac{1}{2}} \right)}{\widetilde{E}} - \sum\limits_{\mathfrak{a} = 1}^{\frac{6}{2} - 1} \sum\limits_{k_\mathfrak{a} = 1}^{\frac{6}{2} - 1 - \sum\limits_{\mathfrak{b} = 1}^{\mathfrak{a} - 1} k_\mathfrak{b}} C_{k_1} \mathfrak{f}(\mathfrak{a}) \aInnerproduct{\dfrac{1}{2} \nabla \left(p^{n + \frac{1}{2}} + p^{n - \frac{1}{2}} \right)}{\left( \nabla \nabla \cdot \right)^{\sum\limits_{\mathfrak{m} = 1}^\mathfrak{a} k_\mathfrak{m}} \widetilde{E}} \\ + \aInnerproduct{\varepsilon \dfrac{E^{n + \frac{1}{2}} - E^{n - \frac{1}{2}}}{\Delta t}}{\widetilde{E}} -  \aInnerproduct{\dfrac{1}{2} \left( H^{n + 1} + H^n \right)}{\nabla \times \widetilde{E}} \\ + \sum\limits_{\mathfrak{a} = 1}^{\frac{6}{2} - 1} \sum\limits_{k_\mathfrak{a} = 1}^{\frac{6}{2} - 1 - \sum\limits_{\mathfrak{b} = 1}^{\mathfrak{a} - 1} k_\mathfrak{b}} C_{k_1} \mathfrak{f}( \mathfrak{a}) (-1)^{\sum\limits_{\mathfrak{m} = 1}^\mathfrak{a} k_\mathfrak{m}} \aInnerproduct{\dfrac{1}{2} ( \mu \varepsilon )^{-\sum\limits_{\mathfrak{m} = 1}^\mathfrak{a} k_\mathfrak{m}} \left( \nabla \times \nabla \times \right)^{\sum\limits_{\mathfrak{m} = 1}^\mathfrak{a} k_\mathfrak{m}} \left( H^{n + 1} + H^n \right)}{\nabla \times \widetilde{E}} = 0, \label{eqn:maxwell_E_lf6_general} 
     \end{multline} \\
      \begin{multline}
  \aInnerproduct{\mu \dfrac{H^{n + 1} - H^n}{\Delta t}}{\widetilde{H}} +  \aInnerproduct{\dfrac{1}{2} \nabla \times \left( E^{n + \frac{1}{2}} + E^{n - \frac{1}{2}} \right)}{\widetilde{H}} \\ - \sum\limits_{\mathfrak{a} = 1}^{\frac{6}{2} - 1} \sum\limits_{k_\mathfrak{a} = 1}^{\frac{6}{2} - 1 - \sum\limits_{\mathfrak{b} = 1}^{\mathfrak{a} - 1} k_\mathfrak{b}} C_{k_1} \mathfrak{f}(\mathfrak{a}) (-1)^{\sum\limits_{\mathfrak{m} = 1}^\mathfrak{a} k_\mathfrak{m}} \aInnerproduct{\dfrac{1}{2} (\varepsilon \mu)^{-\sum\limits_{\mathfrak{m} = 1}^\mathfrak{a} k_\mathfrak{m}} \nabla \times \left( E^{n + \frac{1}{2}} + E^{n - \frac{1}{2}} \right)}{\left( \nabla \times  \nabla \times \right)^{\sum\limits_{\mathfrak{m} = 1}^\mathfrak{a} k_\mathfrak{m}} \widetilde{H}} = 0, \label{eqn:maxwell_H_lf6_general}
       \end{multline}
\end{subequations}
for all $\widetilde{p} \in \mathring{H}^1_{\varepsilon^{-1}}(\Omega), \; \widetilde{E} \in \mathring{H}_{\varepsilon}(\curl; \Omega)$ and $\widetilde{H} \in \mathring{H}_{\mu}(\divgn; \Omega)$. The bootstrapping equations using the initial values at $t = 0$ are:
\begin{subequations}
  \begin{multline}
    \aInnerproduct{\dfrac{p^{\frac{1}{2}} - p_0}{\Delta t/2}}{\widetilde{p}} - \dfrac{1}{2} \aInnerproduct{\dfrac{\varepsilon}{2} \left( E^{\frac{1}{2}} + E_0 \right)}{\nabla \widetilde{p}} \\ + \sum\limits_{\mathfrak{a} = 1}^{\frac{6}{2} - 1} \sum\limits_{k_\mathfrak{a} = 1}^{\frac{6}{2} - 1 - \sum\limits_{\mathfrak{b} = 1}^{\mathfrak{a} - 1} k_\mathfrak{b}} \dfrac{C_{k_1}}{2^{2 \sum\limits_{\mathfrak{m} = 1}^\mathfrak{a} k_\mathfrak{m} + 1}} \mathfrak{f}( \mathfrak{a}) \aInnerproduct{\dfrac{\varepsilon}{2} (\nabla \nabla \cdot)^{\sum\limits_{\mathfrak{m} = 1}^\mathfrak{a} k_\mathfrak{m}} \left( E^{\frac{1}{2}} + E_0 \right)}{\nabla \widetilde{p}}  = 0, \label{eqn:maxwell_p0_lf6_general}
    \end{multline} \\
    \begin{multline}
  \dfrac{1}{2}  \aInnerproduct{\dfrac{1}{2} \nabla \left(  p^{\frac{1}{2}} +  p_0 \right)}{\widetilde{E}} - \sum\limits_{\mathfrak{a} = 1}^{\frac{6}{2} - 1} \sum\limits_{k_\mathfrak{a} = 1}^{\frac{6}{2} - 1 - \sum\limits_{\mathfrak{b} = 1}^{\mathfrak{a} - 1} k_\mathfrak{b}} \dfrac{C_{k_1}}{2^{2 \sum\limits_{\mathfrak{m} = 1}^\mathfrak{a} k_\mathfrak{m} + 1}} \mathfrak{f}(\mathfrak{a}) \aInnerproduct{\dfrac{1}{2} \nabla \left(p^{\frac{1}{2}} + p_0 \right)}{\left( \nabla \nabla \cdot \right)^{\sum\limits_{\mathfrak{m} = 1}^\mathfrak{a} k_\mathfrak{m}} \widetilde{E}} \\ + \aInnerproduct{\varepsilon \dfrac{E^{\frac{1}{2}} - E_0}{\Delta t/2}}{\widetilde{E}} - \aInnerproduct{\dfrac{1}{2} \left( H^1 + H_0 \right)}{\nabla \times \widetilde{E}} \\ + \sum\limits_{\mathfrak{a} = 1}^{\frac{6}{2} - 1}\sum\limits_{k_\mathfrak{a} = 1}^{\frac{6}{2} - 1 - \sum\limits_{\mathfrak{b} = 1}^{\mathfrak{a} - 1} k_\mathfrak{b}}  \dfrac{C_{k_1}}{2^{2 \sum\limits_{\mathfrak{m} = 1}^\mathfrak{a} k_\mathfrak{m}}} \mathfrak{f}( \mathfrak{a}) (-1)^{\sum\limits_{\mathfrak{m} = 1}^\mathfrak{a} k_\mathfrak{m}} \aInnerproduct{\dfrac{1}{2} ( \mu \varepsilon )^{-\sum\limits_{\mathfrak{m} = 1}^\mathfrak{a} k_\mathfrak{m}} \left( \nabla \times \nabla \times \right)^{\sum\limits_{\mathfrak{m} = 1}^\mathfrak{a} k_\mathfrak{m}} \left( H^{1} + H_0 \right)}{\nabla \times \widetilde{E}} = 0, \label{eqn:maxwell_E0_lf6_general} 
  \end{multline} \\
  \begin{multline}
    \aInnerproduct{\mu \dfrac{H^1 - H_0}{\Delta t}}{\widetilde{H}} + \dfrac{1}{2} \aInnerproduct{\dfrac{1}{2} \nabla \times \left(E^{\frac{1}{2}} + E_0 \right)}{\widetilde{H}} \\ - \sum\limits_{\mathfrak{a} = 1}^{\frac{6}{2} - 1} \sum\limits_{k_\mathfrak{a} = 1}^{\frac{6}{2} - 1 - \sum\limits_{\mathfrak{b} = 1}^{\mathfrak{a} - 1} k_\mathfrak{b}} \dfrac{C_{k_1}}{2^{2 \sum\limits_{\mathfrak{m} = 1}^\mathfrak{a} k_\mathfrak{m} + 1}} \mathfrak{f}(\mathfrak{a}) (-1)^{\sum\limits_{\mathfrak{m} = 1}^\mathfrak{a} k_\mathfrak{m}} \aInnerproduct{\dfrac{1}{2} (\varepsilon \mu)^{-\sum\limits_{\mathfrak{m} = 1}^\mathfrak{a} k_\mathfrak{m}} \nabla \times \left( E^{\frac{1}{2}} + E_0 \right)}{\left( \nabla \times  \nabla \times \right)^{\sum\limits_{\mathfrak{m} = 1}^\mathfrak{a} k_\mathfrak{m}} \widetilde{H}} = 0, \label{eqn:maxwell_H0_lf6_general}
    \end{multline}
where 
\begin{equation}
\mathfrak{f}(\mathfrak{a}) \coloneq  \left( \prod\limits_{\mathfrak{c} = 2}^\mathfrak{a} \dfrac{1}{(2 k_\mathfrak{c})!} \right) \left(\dfrac{\Delta t}{2}\right)^{2 \sum\limits_{\mathfrak{d} = 1}^\mathfrak{a} k_\mathfrak{d}}, \quad C_{k_1} \coloneq \left(\dfrac{1}{(2 k_1)!} - \dfrac{1}{(2 k_1 + 1)!}\right) . \label{eqn:fa_Ck1_value}
\end{equation}
\end{subequations}
We can now state our LF$_R$ scheme and we do so by appropriately generalizing from the sixth order in the LF$_6$ expressions as in Equations~\labelcref{eqn:maxwell_p_lf6_general,eqn:maxwell_E_lf6_general,eqn:maxwell_H_lf6_general,eqn:maxwell_p0_lf6_general,eqn:maxwell_E0_lf6_general,eqn:maxwell_H0_lf6_general} to $R^{\text{th}}$ order.

\subsection{The LF$_R$ Scheme} Find $(p(t^{n + 1/2}), E(t^{n + 1/2}), H(t^n)) \in \mathring{H}^1_{\varepsilon^{-1}}(\Omega) \times \mathring{H}_{\varepsilon}(\curl; \Omega) \times \mathring{H}_{\mu}(\divgn; \Omega)$ such that: 
\begin{subequations} 
\begin{multline}
  \aInnerproduct{\dfrac{p^{n + \frac{1}{2}} - p^{n - \frac{1}{2}}}{\Delta t}}{\widetilde{p}} - \aInnerproduct{\dfrac{\varepsilon}{2} \left( E^{n + \frac{1}{2}} + E^{n - \frac{1}{2}} \right)}{\nabla \widetilde{p}} \\ + \sum\limits_{\mathfrak{a} = 1}^{\frac{R}{2} - 1} \sum\limits_{k_\mathfrak{a} = 1}^{\frac{R}{2} - 1 - \sum\limits_{\mathfrak{b} = 1}^{\mathfrak{a} - 1} k_\mathfrak{b}} C_{k_1} \mathfrak{f}(\mathfrak{a}) \aInnerproduct{\dfrac{\varepsilon}{2} (\nabla \nabla \cdot)^{\sum\limits_{\mathfrak{m} = 1}^\mathfrak{a} k_\mathfrak{m}} \left( E^{n + \frac{1}{2}} + E^{n - \frac{1}{2}} \right)}{\nabla \widetilde{p}} =0, \label{eqn:maxwell_p_lfR}
  \end{multline} \\
  \begin{multline}
   \aInnerproduct{\dfrac{1}{2} \nabla \left(p^{n + \frac{1}{2}} + p^{n - \frac{1}{2}} \right)}{\widetilde{E}} - \sum\limits_{\mathfrak{a} = 1}^{\frac{R}{2} - 1} \sum\limits_{k_\mathfrak{a} = 1}^{\frac{R}{2} - 1 - \sum\limits_{\mathfrak{b} = 1}^{\mathfrak{a} - 1} k_\mathfrak{b}} C_{k_1} \mathfrak{f}(\mathfrak{a}) \aInnerproduct{\dfrac{1}{2} \nabla \left(p^{n + \frac{1}{2}} + p^{n - \frac{1}{2}} \right)}{\left( \nabla \nabla \cdot \right)^{\sum\limits_{\mathfrak{m} = 1}^\mathfrak{a} k_\mathfrak{m}} \widetilde{E}} \\ + \aInnerproduct{\varepsilon \dfrac{E^{n + \frac{1}{2}} - E^{n - \frac{1}{2}}}{\Delta t}}{\widetilde{E}} -  \aInnerproduct{\dfrac{1}{2} \left( H^{n + 1} + H^n \right)}{\nabla \times \widetilde{E}} \\ + \sum\limits_{\mathfrak{a} = 1}^{\frac{R}{2} - 1} \sum\limits_{k_\mathfrak{a} = 1}^{\frac{R}{2} - 1 - \sum\limits_{\mathfrak{b} = 1}^{\mathfrak{a} - 1} k_\mathfrak{b}} C_{k_1} \mathfrak{f}( \mathfrak{a}) (-1)^{\sum\limits_{\mathfrak{m} = 1}^\mathfrak{a} k_\mathfrak{m}} \aInnerproduct{\dfrac{1}{2} ( \mu \varepsilon )^{-\sum\limits_{\mathfrak{m} = 1}^\mathfrak{a} k_\mathfrak{m}} \left( \nabla \times \nabla \times \right)^{\sum\limits_{\mathfrak{m} = 1}^\mathfrak{a} k_\mathfrak{m}} \left( H^{n + 1} + H^n \right)}{\nabla \times \widetilde{E}} = 0, \label{eqn:maxwell_E_lfR} 
     \end{multline} \\
      \begin{multline}
  \aInnerproduct{\mu \dfrac{H^{n + 1} - H^n}{\Delta t}}{\widetilde{H}} +  \aInnerproduct{\dfrac{1}{2} \nabla \times \left( E^{n + \frac{1}{2}} + E^{n - \frac{1}{2}} \right)}{\widetilde{H}} \\ - \sum\limits_{\mathfrak{a} = 1}^{\frac{R}{2} - 1} \sum\limits_{k_\mathfrak{a} = 1}^{\frac{R}{2} - 1 - \sum\limits_{\mathfrak{b} = 1}^{\mathfrak{a} - 1} k_\mathfrak{b}} C_{k_1} \mathfrak{f}(\mathfrak{a}) (-1)^{\sum\limits_{\mathfrak{m} = 1}^\mathfrak{a} k_\mathfrak{m}} \aInnerproduct{\dfrac{1}{2} (\varepsilon \mu)^{-\sum\limits_{\mathfrak{m} = 1}^\mathfrak{a} k_\mathfrak{m}} \nabla \times \left( E^{n + \frac{1}{2}} + E^{n - \frac{1}{2}} \right)}{\left( \nabla \times  \nabla \times \right)^{\sum\limits_{\mathfrak{m} = 1}^\mathfrak{a} k_\mathfrak{m}} \widetilde{H}} = 0, \label{eqn:maxwell_H_lfR}
       \end{multline}
\end{subequations}
for all $\widetilde{p} \in \mathring{H}^1_{\varepsilon^{-1}}(\Omega), \; \widetilde{E} \in \mathring{H}_{\varepsilon}(\curl; \Omega)$ and $\widetilde{H} \in \mathring{H}_{\mu}(\divgn; \Omega)$ with bootstrapping using the given initial values at $t = 0$ for the first half time step for $p$ and $E$, and for the first time step for $H$:
\begin{subequations}
  \begin{multline}
    \aInnerproduct{\dfrac{p^{\frac{1}{2}} - p_0}{\Delta t/2}}{\widetilde{p}} - \dfrac{1}{2} \aInnerproduct{\dfrac{\varepsilon}{2} \left( E^{\frac{1}{2}} + E_0 \right)}{\nabla \widetilde{p}} \\ + \sum\limits_{\mathfrak{a} = 1}^{\frac{R}{2} - 1} \sum\limits_{k_\mathfrak{a} = 1}^{\frac{R}{2} - 1 - \sum\limits_{\mathfrak{b} = 1}^{\mathfrak{a} - 1} k_\mathfrak{b}} \dfrac{C_{k_1}}{2^{2 \sum\limits_{\mathfrak{m} = 1}^\mathfrak{a} k_\mathfrak{m} + 1}} \mathfrak{f}( \mathfrak{a}) \aInnerproduct{\dfrac{\varepsilon}{2} (\nabla \nabla \cdot)^{\sum\limits_{\mathfrak{m} = 1}^\mathfrak{a} k_\mathfrak{m}} \left( E^{\frac{1}{2}} + E_0 \right)}{\nabla \widetilde{p}}  = 0, \label{eqn:maxwell_p0_lfR}
    \end{multline} \\
    \begin{multline}
  \dfrac{1}{2}  \aInnerproduct{\dfrac{1}{2} \nabla \left(  p^{\frac{1}{2}} +  p_0 \right)}{\widetilde{E}} - \sum\limits_{\mathfrak{a} = 1}^{\frac{R}{2} - 1} \sum\limits_{k_\mathfrak{a} = 1}^{\frac{R}{2} - 1 - \sum\limits_{\mathfrak{b} = 1}^{\mathfrak{a} - 1} k_\mathfrak{b}} \dfrac{C_{k_1}}{2^{2 \sum\limits_{\mathfrak{m} = 1}^\mathfrak{a} k_\mathfrak{m} + 1}} \mathfrak{f}(\mathfrak{a}) \aInnerproduct{\dfrac{1}{2} \nabla \left(p^{\frac{1}{2}} + p_0 \right)}{\left( \nabla \nabla \cdot \right)^{\sum\limits_{\mathfrak{m} = 1}^\mathfrak{a} k_\mathfrak{m}} \widetilde{E}} \\ + \aInnerproduct{\varepsilon \dfrac{E^{\frac{1}{2}} - E_0}{\Delta t/2}}{\widetilde{E}} - \aInnerproduct{\dfrac{1}{2} \left( H^1 + H_0 \right)}{\nabla \times \widetilde{E}} \\ + \sum\limits_{\mathfrak{a} = 1}^{\frac{R}{2} - 1}\sum\limits_{k_\mathfrak{a} = 1}^{\frac{R}{2} - 1 - \sum\limits_{\mathfrak{b} = 1}^{\mathfrak{a} - 1} k_\mathfrak{b}}  \dfrac{C_{k_1}}{2^{2 \sum\limits_{\mathfrak{m} = 1}^\mathfrak{a} k_\mathfrak{m}}} \mathfrak{f}( \mathfrak{a}) (-1)^{\sum\limits_{\mathfrak{m} = 1}^\mathfrak{a} k_\mathfrak{m}} \aInnerproduct{\dfrac{1}{2} ( \mu \varepsilon )^{-\sum\limits_{\mathfrak{m} = 1}^\mathfrak{a} k_\mathfrak{m}} \left( \nabla \times \nabla \times \right)^{\sum\limits_{\mathfrak{m} = 1}^\mathfrak{a} k_\mathfrak{m}} \left( H^{1} + H_0 \right)}{\nabla \times \widetilde{E}} = 0, \label{eqn:maxwell_E0_lfR} 
  \end{multline} \\
  \begin{multline}
    \aInnerproduct{\mu \dfrac{H^1 - H_0}{\Delta t}}{\widetilde{H}} + \dfrac{1}{2} \aInnerproduct{\dfrac{1}{2} \nabla \times \left(E^{\frac{1}{2}} + E_0 \right)}{\widetilde{H}} \\ - \sum\limits_{\mathfrak{a} = 1}^{\frac{R}{2} - 1} \sum\limits_{k_\mathfrak{a} = 1}^{\frac{R}{2} - 1 - \sum\limits_{\mathfrak{b} = 1}^{\mathfrak{a} - 1} k_\mathfrak{b}} \dfrac{C_{k_1}}{2^{2 \sum\limits_{\mathfrak{m} = 1}^\mathfrak{a} k_\mathfrak{m} + 1}} \mathfrak{f}(\mathfrak{a}) (-1)^{\sum\limits_{\mathfrak{m} = 1}^\mathfrak{a} k_\mathfrak{m}} \aInnerproduct{\dfrac{1}{2} (\varepsilon \mu)^{-\sum\limits_{\mathfrak{m} = 1}^\mathfrak{a} k_\mathfrak{m}} \nabla \times \left( E^{\frac{1}{2}} + E_0 \right)}{\left( \nabla \times  \nabla \times \right)^{\sum\limits_{\mathfrak{m} = 1}^\mathfrak{a} k_\mathfrak{m}} \widetilde{H}} = 0, \label{eqn:maxwell_H0_lfR}
    \end{multline}
\end{subequations}
where $\mathfrak{f}(\mathfrak{a})$ and $C_{k_1}$ are defined to be same as in ~\eqref{eqn:fa_Ck1_value}

\subsection{A Note on Preliminaries} We choose to be compact and efficient in our presentation here, and refer to the earlier work~\cite[Section 2]{ArKa2025} for a substantial overview of the function spaces, finite elements, and various inequalities that we use in the course of our analysis of LF$_R$. A summary of these preliminaries can also be found in~\cite[Section 1.2]{ArKa2026}. In the context of this latter work, we wish to note that LF$_R$ is the most generalized version of the method that is called the spatial strategy scheme as in~\cite[Section 1.3]{ArKa2026}. Some background and earlier work that provide the foundation for our problem, the techniques and tools are already stated in~\cite[Section 1.1]{ArKa2025}.

\subsection{Organization} In Section~\labelcref{sec:implicit_lf$R$}, we prove the energy conservation of LF$_R$, and analyze and characterize the asymptotic convergence of the semidiscretization error for the Maxwell's system as well as the error for the full discretization of this system in conjunction with compatible finite elements as drawn from finite element exterior calculus~\cite{Arnold2018}. In Section~\labelcref{sec:numerics}, we provide some computational examples and end with a brief summary of our work.

 \section{Characterization of Implicit LF$_R$ Scheme} \label{sec:implicit_lf$R$}
 
 \subsection{Time Discretization Stability}

\begin{theorem}[Discrete Energy Estimate] \label{thm:dscrt_enrgy_estmt_lfR}
  For the semidiscretization using the LF$_R$ scheme as given in Equations~\labelcref{eqn:maxwell_p_lfR,eqn:maxwell_E_lfR,eqn:maxwell_H_lfR,eqn:maxwell_p0_lfR,eqn:maxwell_E0_lfR,eqn:maxwell_H0_lfR}, and for any fixed time step $\Delta t$ > 0 sufficiently small, we have that:
  \[
  \mathcal{E}^N \coloneq \norm{p^{N - \frac{1}{2}}}^2_{\varepsilon^{-1}} + \norm{E^{N - \frac{1}{2}}}^2_{\varepsilon} + \norm{H^N}^2_{\mu} = \norm{p_0}^2_{\varepsilon^{-1}} + \norm{E_0}^2_{\varepsilon} + \norm{H_0}^2_{\mu} \eqcolon \mathcal{E}^0.
\]
\end{theorem}

\begin{proof}
  Since Equations~\labelcref{eqn:maxwell_p_lfR,eqn:maxwell_E_lfR,eqn:maxwell_H_lfR} are true for all $\widetilde{p}\in \mathring{H}^1_{\varepsilon^{-1}}(\Omega)$, $\widetilde{E} \in \mathring{H}_{\varepsilon}(\curl; \Omega)$, $\widetilde{H} \in \mathring{H}_{\mu}(\divgn; \Omega)$, using $\widetilde{p} = 2 \Delta t \varepsilon^{-1} \left( p^{n + \frac{1}{2}} + p^{n - \frac{1}{2}} \right)$, $\widetilde{E} = 2 \Delta t \left(E^{n + \frac{1}{2}} + E^{n - \frac{1}{2}} \right)$ and $\widetilde{H} = 2 \Delta t \left(H^{n + 1} + H^n \right)$, we obtain:
  \begin{multline*}
    2 \aInnerproduct{p^{n+\frac{1}{2}} - p^{n - \frac{1}{2}}}{\varepsilon^{-1} \left( p^{n + \frac{1}{2}} + p^{n-\frac{1}{2}} \right)} - \Delta t \aInnerproduct{E^{n + \frac{1}{2}} + E^{n - \frac{1}{2}}}{\nabla \left( p^{n + \frac{1}{2}} + p^{n - \frac{1}{2}} \right)}  \\ + \Delta t \sum\limits_{\mathfrak{a} = 1}^{\frac{R}{2} - 1} \sum\limits_{k_\mathfrak{a} = 1}^{\frac{R}{2} - 1 - \sum\limits_{\mathfrak{b} = 1}^{\mathfrak{a} - 1} k_\mathfrak{b}} C_{k_1} \mathfrak{f}(\mathfrak{a}) \aInnerproduct{(\nabla \nabla \cdot)^{\sum\limits_{\mathfrak{m} = 1}^\mathfrak{a} k_\mathfrak{m}} \left( E^{n + \frac{1}{2}} + E^{n - \frac{1}{2}} \right)}{\nabla \left( p^{n + \frac{1}{2}} + p^{n - \frac{1}{2}} \right)} =0,
\end{multline*}
\vspace{-1em} \begin{multline*}
  \Delta t \aInnerproduct{\nabla \left( p^{n + \frac{1}{2}} + p^{n - \frac{1}{2}} \right)}{E^{n + \frac{1}{2}} + E^{n - \frac{1}{2}}} \\ - \Delta t \sum\limits_{\mathfrak{a} = 1}^{\frac{R}{2} - 1} \sum\limits_{k_\mathfrak{a} = 1}^{\frac{R}{2} - 1 - \sum\limits_{\mathfrak{b} = 1}^{\mathfrak{a} - 1} k_\mathfrak{b}} C_{k_1} \mathfrak{f}(\mathfrak{a})  \aInnerproduct{\nabla \left( p^{n + \frac{1}{2}} + p^{n - \frac{1}{2}} \right)}{(\nabla \nabla \cdot)^{\sum\limits_{\mathfrak{m} = 1}^\mathfrak{a} k_\mathfrak{m}} \left(E^{n + \frac{1}{2}} + E^{n - \frac{1}{2}}\right)} \\ + 2 \aInnerproduct{\varepsilon \left(E^{n + \frac{1}{2}} - E^{n - \frac{1}{2}} \right)}{E^{n + \frac{1}{2}} + E^{n - \frac{1}{2}}} \, -
  \Delta t \aInnerproduct{\left( H^{n + 1} + H^n \right)}{\nabla \times \left( E^{n + \frac{1}{2}} + E^{n - \frac{1}{2}} \right)} \\ +  \Delta t \sum\limits_{\mathfrak{a} = 1}^{\frac{R}{2} - 1} \sum\limits_{k_\mathfrak{a} = 1}^{\frac{R}{2} - 1 - \sum\limits_{\mathfrak{b} = 1}^{\mathfrak{a} - 1} k_\mathfrak{b}} C_{k_1} \mathfrak{f}(\mathfrak{a}) (-1)^{\sum\limits_{\mathfrak{m} = 1}^\mathfrak{a} k_\mathfrak{m}} \aInnerproduct{(\mu \varepsilon)^{-\sum\limits_{\mathfrak{m} = 1}^\mathfrak{a} k_\mathfrak{m}} (\nabla \times \nabla \times)^{\sum\limits_{\mathfrak{m} = 1}^\mathfrak{a} k_\mathfrak{m}} \left( H^{n + 1} + H^n \right)}{\nabla \times \left( E^{n + \frac{1}{2}} + E^{n - \frac{1}{2}} \right)} \\ = 0,
\end{multline*}
\vspace{-1em} \begin{multline*}
  2 \aInnerproduct{\mu \left(H^{n + 1} - H^{n} \right)}{H^{n + 1} + H^n} + \Delta t \aInnerproduct{\nabla \times \left( E^{n + \frac{1}{2}} + E^{n - \frac{1}{2}} \right)}{H^{n + 1} + H^n} \\ -  \Delta t \sum\limits_{\mathfrak{a} = 1}^{\frac{R}{2} - 1} \sum\limits_{k_\mathfrak{a} = 1}^{\frac{R}{2} - 1 - \sum\limits_{\mathfrak{b} = 1}^{\mathfrak{a} - 1} k_\mathfrak{b}} C_{k_1} \mathfrak{f}(\mathfrak{a}) (-1)^{\sum\limits_{\mathfrak{m} = 1}^\mathfrak{a} k_\mathfrak{m}} \aInnerproduct{(\varepsilon \mu)^{-\sum\limits_{\mathfrak{m} = 1}^\mathfrak{a} k_\mathfrak{m}} \nabla \times \left( E^{n + \frac{1}{2}} + E^{n - \frac{1}{2}} \right)}{(\nabla \times \nabla \times)^{\sum\limits_{\mathfrak{m} = 1}^\mathfrak{a} k_\mathfrak{m}} \left(H^{n + 1} + H^n \right)} \\ = 0.
\end{multline*}
Using the symmetry of the inner product and adding the above equations, we next obtain:
\begin{multline*} \begin{split}
  2 \left[ \aInnerproduct{\varepsilon^{-1} \left( p^{n + \frac{1}{2}} - p^{n - \frac{1}{2}} \right)}{p^{n + \frac{1}{2}} + p^{n - \frac{1}{2}}} + \aInnerproduct{\varepsilon \left(E^{n + \frac{1}{2}} - E^{n - \frac{1}{2}} \right)}{E^{n + \frac{1}{2}} + E^{n - \frac{1}{2}}} \right. \\ \left. + \aInnerproduct{\mu \left( H^{n + 1} - H^{n} \right)}{H^{n + 1} + H^n}\right] = 0. 
  \end{split}\\
  \implies \norm{p^{n + \frac{1}{2}}}^2_{\varepsilon^{-1}} - \norm{p^{n - \frac{1}{2}}}^2_{\varepsilon^{-1}} + \norm{E^{n + \frac{1}{2}}}^2_{\varepsilon} - \norm{E^{n - \frac{1}{2}}}^2_{\varepsilon} + \norm{H^{n+1}}^2_{\mu} -  \norm{H^n}^2_{\mu} = 0.
\end{multline*}
Summing over $n = 1$ to $N - 1$ results into:
\[
  \norm{p^{N - \frac{1}{2}}}^2_{\varepsilon^{-1}} - \norm{p^\frac{1}{2}}^2_{\varepsilon^{-1}} + \norm{E^{N - \frac{1}{2}}}^2_{\varepsilon} - \norm{E^\frac{1}{2}}^2_{\varepsilon} + \norm{H^N}^2_{\mu} - \norm{H^1}^2_{\mu} = 0.
\]
Now, $p^{\frac{1}{2}}$, $E^{\frac{1}{2}}$ and $H^1$ satisfy Equations~\labelcref{eqn:maxwell_p0_lfR,eqn:maxwell_E0_lfR,eqn:maxwell_H0_lfR}, and so using $\widetilde{p} = 2 \Delta t \varepsilon^{-1} \left( p^{\frac{1}{2}} + p^0 \right)$, $\widetilde{E} = 2 \Delta t \left(E^{\frac{1}{2}} + E^0 \right)$ and $\widetilde{H} = 4 \Delta t \left(H^{1} + H^0 \right)$ and repeating the arguments leads us to the following estimate:
\[
  \norm{p^{\frac{1}{2}}}^2_{\varepsilon^{-1}}  - \norm{p^0}^2_{\varepsilon^{-1}} + \norm{E^{\frac{1}{2}}}^2_{\varepsilon} - \norm{E^0}^2_{\varepsilon} + \norm{H^{1}}^2_{\mu} - \norm{H^0}^2_{\mu} = 0,
\]
from which we then get that:
\[
  \norm{p^{N - \frac{1}{2}}}^2_{\varepsilon^{-1}} + \norm{E^{N - \frac{1}{2}}}^2_{\varepsilon} + \norm{H^N}^2_{\mu} = \norm{p^0}^2_{\varepsilon^{-1}} + \norm{E^0}^2_{\varepsilon} + \norm{H^0}^2_{\mu}.
\]
\end{proof}

\begin{theorem}[Discrete Error Estimate]\label{thm:dscrt_error_estmt_lfR} For the semidiscretization using the LF$_R$ scheme as in Equations~\labelcref{eqn:maxwell_p_lfR,eqn:maxwell_E_lfR,eqn:maxwell_H_lfR}, and ~\labelcref{eqn:maxwell_p0_lfR,eqn:maxwell_E0_lfR,eqn:maxwell_H0_lfR}, for the solution $(p, E, H)$ of Equations~\labelcref{eqn:maxwell_p_wf,eqn:maxwell_E_wf,eqn:maxwell_H_wf} with initial conditions as in Equation~\eqref{eqn:ICs}, assuming sufficient regularity with $p \in C^{R+1}(0, T; \mathring{H}^1_{\varepsilon^{-1}}(\Omega))$, $E \in C^{R+1}(0, T; \mathring{H}_{\varepsilon}(\curl; \Omega))$, and $H \in C^{R+1}(0, T; \mathring{H}_{\mu}(\divgn; \Omega))$, and for a time step $\Delta t > 0$ sufficiently small, there exists a positive bounded constant $C$ independent of $\Delta t$ such that: \[
  \norm{e_p^{N - \frac{1}{2}}}_{\varepsilon^{-1}} + \norm{e_E^{N - \frac{1}{2}}}_{\varepsilon} + \norm{e_H^N}_{\mu} \le C \left[ \left(\Delta t\right)^R + \norm{e_p^0}_{\varepsilon^{-1}} + \norm{e_E^0}_{\varepsilon} + \norm{e_H^0}_{\mu} \right],
\]
where $e_p^{n + \frac{1}{2}} \coloneq p(t^{n + \frac{1}{2}}) - p^{n+\frac{1}{2}}$, $e_E^{n + \frac{1}{2}} \coloneq E(t^{n + \frac{1}{2}}) - E^{n + \frac{1}{2}}$ and $e_H^n \coloneq H(t^n) - H^n$ are the errors in the time semidiscretization of $p$, $E$ and $H$, respectively.
\end{theorem}

\begin{proof}
Using the Taylor remainder theorem upto $(R+1)^{\text{th}}$ order, and expressing $p(t)$ about $t = t^n$, we have that:
\begin{multline*}
  p(t) = p(t^n) + \dfrac{\partial p}{\partial t}(t^n)(t - t^n) + \dfrac{\partial^2 p}{\partial t^2}(t^n) \dfrac{(t - t^n)^2}{2!} + \dotso + \dfrac{\partial^{R-1} p}{\partial t^{R-1}}(t^n) \dfrac{(t - t^n)^{R-1}}{(R-1)!} + \dfrac{\partial^R p}{\partial t^R}(t^n) \dfrac{(t - t^n)^R}{R!} \\ + \int\limits_{t^n}^{t} \dfrac{(t - s)^R}{R!} \dfrac{\partial^{R+1} p}{\partial t^{R+1}}(s) ds,
\end{multline*}
which when evaluated at $t = t^{n + \frac{1}{2}}$ and $t = t^{n - \frac{1}{2}}$ yields:
\begin{multline*}
  p(t^{n + \frac{1}{2}}) = p(t^n) + \dfrac{\Delta t}{2} \dfrac{\partial p}{\partial t}(t^n) +\dfrac{1}{2!} \dfrac{\Delta t^2}{2^2}\dfrac{\partial^2 p}{\partial t^2}(t^n) + \dotso + \dfrac{1}{(R-1)!} \dfrac{\Delta t^{R-1}}{2^{R-1}}\dfrac{\partial^{R-1} p}{\partial t^{R-1}}(t^n) \\ + \dfrac{1}{R!} \dfrac{\Delta t^R}{2^R} \dfrac{\partial^R p}{\partial t^R}(t^n) + \int\limits_{t^n}^{\mathclap{t^{n + \frac{1}{2}}}} \dfrac{(t^{n + \frac{1}{2}} - s)^R}{R!} \dfrac{\partial^{R+1} p}{\partial t^{R+1}}(s) ds, 
 \end{multline*}
\begin{multline*}
p(t^{n - \frac{1}{2}}) = p(t^n) - \dfrac{\Delta t}{2} \dfrac{\partial p}{\partial t}(t^n) +\dfrac{1}{2!} \dfrac{\Delta t^2}{2^2}\dfrac{\partial^2 p}{\partial t^2}(t^n) + \dotso - \dfrac{1}{(R-1)!} \dfrac{\Delta t^{R-1}}{2^{R-1}}\dfrac{\partial^{R-1} p}{\partial t^{R-1}}(t^n) \\ + \dfrac{1}{R!} \dfrac{\Delta t^R}{2^R} \dfrac{\partial^R p}{\partial t^R}(t^n) + \int\limits_{t^n}^{\mathclap{t^{n - \frac{1}{2}}}} \dfrac{(t^{n - \frac{1}{2}} - s)^R}{R!} \dfrac{\partial^{R+1} p}{\partial t^{R+1}}(s) ds.
\end{multline*}
Subtracting these two equations, and using the result in the inner product term from the semidiscretization of the variational formulation leads to:
\[
  \ainnerproduct{\dfrac{p(t^{n + \frac{1}{2}}) - p(t^{n - \frac{1}{2}})}{\Delta t}}{\widetilde{p}} = \sum \limits_{k=0}^{\frac{R}{2} - 1} \dfrac{1}{(2k+1)!} \dfrac{\Delta t^{2k}}{2^{2k}} \ainnerproduct{\dfrac{\partial^{2k+1} p}{\partial t^{2k+1}}(t^n)}{\widetilde{p}} + \ainnerproduct{R^n_p}{\widetilde{p}},
\]
in which we have defined that:
\[
  R^n_p \coloneq \dfrac{1}{\Delta t} \left[\int\limits_{t^{n - \frac{1}{2}}}^{t^n} \dfrac{(t^{n - \frac{1}{2}} - s)^R}{R!} \dfrac{\partial^{R+1} p}{\partial t^{R+1}}(s) ds + \int\limits_{t^n}^{t^{n + \frac{1}{2}}} \dfrac{(t^{n + \frac{1}{2}} - s)^R}{R!} \dfrac{\partial^{R+1} p}{\partial t^{R+1}}(s) ds \right].
\]
Similarly, for $E$ and $H$, we have the following:
\begin{align*}
  \ainnerproduct{\varepsilon \dfrac{E(t^{n + \frac{1}{2}}) - E(t^{n - \frac{1}{2}})}{\Delta t}}{\widetilde{E}} &= \sum \limits_{k=0}^{\frac{R}{2} - 1} \dfrac{1}{(2k+1)!} \dfrac{\Delta t^{2k}}{2^{2k}} \ainnerproduct{\varepsilon \dfrac{\partial^{2k+1} E}{\partial t^{2k+1}}(t^n)}{\widetilde{E}} + \ainnerproduct{\varepsilon R^n_E}{\widetilde{E}}, \\
  \ainnerproduct{\mu \dfrac{H(t^{n + 1}) - H(t^n)}{\Delta t}}{\widetilde{H}} &=  \sum \limits_{k=0}^{\frac{R}{2} - 1} \dfrac{1}{(2k+1)!} \dfrac{\Delta t^{2k}}{2^{2k}} \ainnerproduct{\mu \dfrac{\partial^{2k+1} H}{\partial t^{2k+1}}(t^{n + \frac{1}{2}})}{\widetilde{H}} + \ainnerproduct{\mu R^{n + \frac{1}{2}}_H}{\widetilde{H}},
\end{align*}
and in each of which we have defined that:
\begin{align*}
  R^n_E &\coloneq \dfrac{1}{\Delta t} \left[ \int\limits_{t^{n - \frac{1}{2}}}^{t^n} \dfrac{(t^{n - \frac{1}{2}} - s)^R}{R!} \dfrac{\partial^{R+1} E}{\partial t^{R+1}}(s) ds + \int\limits_{t^n}^{t^{n + \frac{1}{2}}} \dfrac{(t^{n + \frac{1}{2}} - s)^R}{R!} \dfrac{\partial^{R+1} E}{\partial t^{R+1}}(s) ds \right], \\
R^{n + \frac{1}{2}}_H &\coloneq \dfrac{1}{\Delta t} \left[ \int\limits_{t^n}^{t^{n + \frac{1}{2}}} \dfrac{(t^n - s)^R}{R!} \dfrac{\partial^{R+1} H}{\partial t^{R+1}}(s) ds + \int\limits_{t^{n + \frac{1}{2}}}^{t^{n + 1}} \dfrac{(t^{n + 1} - s)^R}{R!} \dfrac{\partial^{R+1} H}{\partial t^{R+1}}(s) ds\right].
\end{align*}
Again using the Taylor remainder theorem upto $S^{\text{th}}$ order for $2 \leqslant S \leqslant R$, for $S$ taken to be even, we have that:

\begin{align*}
  \ainnerproduct{\dfrac{p(t^{n + \frac{1}{2}}) - p(t^{n - \frac{1}{2}})}{\Delta t}}{\widetilde{p}} & = \sum \limits_{k=0}^{\frac{S}{2} - 1} \dfrac{1}{(2k)!} \dfrac{\Delta t^{2k}}{2^{2k}} \ainnerproduct{\dfrac{\partial^{2k} p}{\partial t^{2k}}(t^n)}{\widetilde{p}} + \ainnerproduct{r^n_{S/2,p}}{\widetilde{p}}, \\
  \ainnerproduct{\varepsilon \dfrac{E(t^{n + \frac{1}{2}}) - E(t^{n - \frac{1}{2}})}{\Delta t}}{\widetilde{E}} &= \sum \limits_{k=0}^{\frac{S}{2} - 1} \dfrac{1}{(2k)!} \dfrac{\Delta t^{2k}}{2^{2k}} \ainnerproduct{\varepsilon \dfrac{\partial^{2k} E}{\partial t^{2k}}(t^n)}{\widetilde{E}} + \ainnerproduct{\varepsilon r^n_{S/2,E}}{\widetilde{E}}, \\
  \ainnerproduct{\mu \dfrac{H(t^{n + 1}) - H(t^n)}{\Delta t}}{\widetilde{H}} &=  \sum \limits_{k=0}^{\frac{S}{2} - 1} \dfrac{1}{(2k)!} \dfrac{\Delta t^{2k}}{2^{2k}} \ainnerproduct{\mu \dfrac{\partial^{2k} H}{\partial t^{2k}}(t^{n + \frac{1}{2}})}{\widetilde{H}} + \ainnerproduct{\mu r^{n + \frac{1}{2}}_{S/2,H}}{\widetilde{H}},
\end{align*}
and in each of which we have defined that:
\begin{align*}
 r^n_{S/2,p} &\coloneq \dfrac{1}{2} \left[-\int\limits_{t^{n - \frac{1}{2}}}^{t^n} \dfrac{(t^{n - \frac{1}{2}} - s)^{S-1}}{(S-1)!} \dfrac{\partial^{S} p}{\partial t^{S}}(s) ds + \int\limits_{t^n}^{t^{n + \frac{1}{2}}} \dfrac{(t^{n + \frac{1}{2}} - s)^{S-1}}{(S-1)!} \dfrac{\partial^{S} p}{\partial t^{S}}(s) ds \right], \\
  r^n_{S/2,E} &\coloneq \dfrac{1}{2} \left[-\int\limits_{t^{n - \frac{1}{2}}}^{t^n} \dfrac{(t^{n - \frac{1}{2}} - s)^{S-1}}{(S-1)!} \dfrac{\partial^{S} E}{\partial t^{S}}(s) ds + \int\limits_{t^n}^{t^{n + \frac{1}{2}}} \dfrac{(t^{n + \frac{1}{2}} - s)^{S-1}}{(S-1)!} \dfrac{\partial^{S} E}{\partial t^{S}}(s) ds \right], \\
r^{n + \frac{1}{2}}_{S/2,H} &\coloneq \dfrac{1}{2} \left[-\int\limits_{t^n}^{t^{n + \frac{1}{2}}} \dfrac{(t^n - s)^{S-1}}{(S-1)!} \dfrac{\partial^{S} H}{\partial t^{S}}(s) ds + \int\limits_{t^{n + \frac{1}{2}}}^{t^{n + 1}} \dfrac{(t^{n + 1} - s)^{S-1}}{(S-1)!} \dfrac{\partial^{S} H}{\partial t^{S}}(s) ds\right].
\end{align*}
Using these terms in the weak formulation as in Equations~\labelcref{eqn:maxwell_p_wf,eqn:maxwell_E_wf,eqn:maxwell_H_wf}  at time $t = t^n$ for $p$ and $E$ terms, and at time $t = t^{n + \frac{1}{2}}$ for $H$, we obtain:
\begin{subequations}
\begin{multline}
  \aInnerproduct{\dfrac{p(t^{n + \frac{1}{2}}) - p(t^{n - \frac{1}{2}})}{\Delta t}}{\widetilde{p}} - \aInnerproduct{\dfrac{\varepsilon}{2} \left( E(t^{n + \frac{1}{2}}) + E(t^{n - \frac{1}{2}}) \right)}{\nabla \widetilde{p}} \\ + \sum\limits_{\mathfrak{a} = 1}^{\frac{R}{2} - 1} \sum\limits_{k_\mathfrak{a} = 1}^{\frac{R}{2} - 1 - \sum\limits_{\mathfrak{b} = 1}^{\mathfrak{a} - 1} k_\mathfrak{b}} C_{k_1} \mathfrak{f}(\mathfrak{a}) \aInnerproduct{\dfrac{\varepsilon}{2} (\nabla \nabla \cdot)^{\sum\limits_{\mathfrak{m} = 1}^\mathfrak{a} k_\mathfrak{m}} \left( E(t^{n + \frac{1}{2}}) + E(t^{n - \frac{1}{2}}) \right)}{\nabla \widetilde{p}} \\ =\aInnerproduct{R_p^{n} + \varepsilon \nabla \cdot r_{R/2,E}^n - \sum\limits_{\mathfrak{a} = 1}^{\frac{R}{2} - 1} \sum\limits_{k_\mathfrak{a} = 1}^{\frac{R}{2} - 1 - \sum\limits_{\mathfrak{b} = 1}^{\mathfrak{a} - 1} k_\mathfrak{b}} C_{k_1} \mathfrak{f}(\mathfrak{a}) \varepsilon \nabla \cdot (\nabla \nabla \cdot)^{\sum\limits_{\mathfrak{m} = 1}^\mathfrak{a} k_\mathfrak{m}} r_{\left( R/2 - \sum \limits_{m=1}^a k_m \right),E}^n }{\widetilde{p}}, \label{eqn:remainder_p_lfR}
  \end{multline} \\
  \begin{multline}
   \aInnerproduct{\dfrac{1}{2} \nabla \left(p(t^{n + \frac{1}{2}}) + p(t^{n - \frac{1}{2}}) \right)}{\widetilde{E}} - \sum\limits_{\mathfrak{a} = 1}^{\frac{R}{2} - 1} \sum\limits_{k_\mathfrak{a} = 1}^{\frac{R}{2} - 1 - \sum\limits_{\mathfrak{b} = 1}^{\mathfrak{a} - 1} k_\mathfrak{b}} C_{k_1} \mathfrak{f}(\mathfrak{a}) \aInnerproduct{\dfrac{1}{2} \nabla \left(p(t^{n + \frac{1}{2}}) + p(t^{n - \frac{1}{2}}) \right)}{\left( \nabla \nabla \cdot \right)^{\sum\limits_{\mathfrak{m} = 1}^\mathfrak{a} k_\mathfrak{m}} \widetilde{E}} \\ + \aInnerproduct{\varepsilon \dfrac{E(t^{n + \frac{1}{2}}) - E(t^{n - \frac{1}{2}})}{\Delta t}}{\widetilde{E}} -  \aInnerproduct{\dfrac{1}{2} \left( H(t^{n + 1}) + H(t^n) \right)}{\nabla \times \widetilde{E}} \\ + \sum\limits_{\mathfrak{a} = 1}^{\frac{R}{2} - 1} \sum\limits_{k_\mathfrak{a} = 1}^{\frac{R}{2} - 1 - \sum\limits_{\mathfrak{b} = 1}^{\mathfrak{a} - 1} k_\mathfrak{b}} C_{k_1} \mathfrak{f}( \mathfrak{a}) (-1)^{\sum\limits_{\mathfrak{m} = 1}^\mathfrak{a} k_\mathfrak{m}} \aInnerproduct{\dfrac{1}{2} ( \mu \varepsilon )^{-\sum\limits_{\mathfrak{m} = 1}^\mathfrak{a} k_\mathfrak{m}} \left( \nabla \times \nabla \times \right)^{\sum\limits_{\mathfrak{m} = 1}^\mathfrak{a} k_\mathfrak{m}} \left( H(t^{n + 1}) + H(t^n) \right)}{\nabla \times \widetilde{E}} \\ = \aInnerproduct{\varepsilon R_E^n + \nabla r_{R/2,p}^n - \nabla \times r_{R/2,H}^{n+\frac{1}{2}} - \sum\limits_{\mathfrak{a} = 1}^{\frac{R}{2} - 1} \sum\limits_{k_\mathfrak{a} = 1}^{\frac{R}{2} - 1 - \sum\limits_{\mathfrak{b} = 1}^{\mathfrak{a} - 1} k_\mathfrak{b}} C_{k_1} \mathfrak{f}(\mathfrak{a}) \left(\varepsilon \left( \nabla \nabla \cdot \right)^{\sum\limits_{\mathfrak{m} = 1}^\mathfrak{a} k_\mathfrak{m}} \nabla r_{\left( R/2 - \sum \limits_{m=1}^a k_m \right),p}^n \right. \\ - \left. (-1)^{\sum\limits_{\mathfrak{m} = 1}^\mathfrak{a} k_\mathfrak{m}} (\mu \varepsilon)^{-\sum\limits_{\mathfrak{m} = 1}^\mathfrak{a} k_\mathfrak{m}} \nabla \times (\nabla \times \nabla \times)^{\sum\limits_{\mathfrak{m} = 1}^\mathfrak{a} k_\mathfrak{m}} r_{\left( R/2 - \sum \limits_{m=1}^a k_m \right),H}^{n+\frac{1}{2}}\right)}{\widetilde{E}}, \label{eqn:remainder_E_lfR}
     \end{multline} \\
      \begin{multline}
  \aInnerproduct{\mu \dfrac{H(t^{n + 1}) - H(t^n)}{\Delta t}}{\widetilde{H}} +  \aInnerproduct{\dfrac{1}{2} \nabla \times \left( E(t^{n + \frac{1}{2}}) + E(t^{n - \frac{1}{2}}) \right)}{\widetilde{H}} \\ - \sum\limits_{\mathfrak{a} = 1}^{\frac{R}{2} - 1} \sum\limits_{k_\mathfrak{a} = 1}^{\frac{R}{2} - 1 - \sum\limits_{\mathfrak{b} = 1}^{\mathfrak{a} - 1} k_\mathfrak{b}} C_{k_1} \mathfrak{f}(\mathfrak{a}) (-1)^{\sum\limits_{\mathfrak{m} = 1}^\mathfrak{a} k_\mathfrak{m}} \aInnerproduct{\dfrac{1}{2} (\varepsilon \mu)^{-\sum\limits_{\mathfrak{m} = 1}^\mathfrak{a} k_\mathfrak{m}} \nabla \times \left( E(t^{n + \frac{1}{2}}) + E(t^{n - \frac{1}{2}}) \right)}{\left( \nabla \times  \nabla \times \right)^{\sum\limits_{\mathfrak{m} = 1}^\mathfrak{a} k_\mathfrak{m}} \widetilde{H}} \\ = \aInnerproduct{\mu R_H^{n + \frac{1}{2}} + \nabla \times r_{R/2,E}^n \\ + \sum\limits_{\mathfrak{a} = 1}^{\frac{R}{2} - 1} \sum\limits_{k_\mathfrak{a} = 1}^{\frac{R}{2} - 1 - \sum\limits_{\mathfrak{b} = 1}^{\mathfrak{a} - 1} k_\mathfrak{b}} C_{k_1} \mathfrak{f}(\mathfrak{a}) (-1)^{\sum\limits_{\mathfrak{m} = 1}^\mathfrak{a} k_\mathfrak{m}} (\varepsilon \mu)^{-\sum\limits_{\mathfrak{m} = 1}^\mathfrak{a} k_\mathfrak{m}} (\nabla \times \nabla \times)^{\sum\limits_{\mathfrak{m} = 1}^\mathfrak{a} k_\mathfrak{m}} \nabla \times r_{\left( R/2 - \sum \limits_{m=1}^a k_m \right),E}^n}{\widetilde{H}}. \label{eqn:remainder_H_lfR}
       \end{multline}
\end{subequations}
Then subtracting the LF$_R$ scheme expressions as in Equations~\labelcref{eqn:maxwell_p_lfR,eqn:maxwell_E_lfR,eqn:maxwell_H_lfR} leads us to the following set of equations:
\begin{multline*}
  \aInnerproduct{\dfrac{e_p^{n + \frac{1}{2}} - e_p^{n - \frac{1}{2}}}{\Delta t}}{\widetilde{p}} - \aInnerproduct{\dfrac{\varepsilon}{2} \left( e_E^{n + \frac{1}{2}} + e_E^{n - \frac{1}{2}} \right)}{\nabla \widetilde{p}} \\ + \sum\limits_{\mathfrak{a} = 1}^{\frac{R}{2} - 1} \sum\limits_{k_\mathfrak{a} = 1}^{\frac{R}{2} - 1 - \sum\limits_{\mathfrak{b} = 1}^{\mathfrak{a} - 1} k_\mathfrak{b}} C_{k_1} \mathfrak{f}(\mathfrak{a}) \aInnerproduct{\dfrac{\varepsilon}{2} (\nabla \nabla \cdot)^{\sum\limits_{\mathfrak{m} = 1}^\mathfrak{a} k_\mathfrak{m}} \left( e_E^{n + \frac{1}{2}} + e_E^{n - \frac{1}{2}} \right)}{\nabla \widetilde{p}} \\ = \aInnerproduct{R_p^{n} + \varepsilon \nabla \cdot r_{R/2,E}^n - \sum\limits_{\mathfrak{a} = 1}^{\frac{R}{2} - 1} \sum\limits_{k_\mathfrak{a} = 1}^{\frac{R}{2} - 1 - \sum\limits_{\mathfrak{b} = 1}^{\mathfrak{a} - 1} k_\mathfrak{b}} C_{k_1} \mathfrak{f}(\mathfrak{a}) \varepsilon \nabla \cdot (\nabla \nabla \cdot)^{\sum\limits_{\mathfrak{m} = 1}^\mathfrak{a} k_\mathfrak{m}} r_{\left( R/2 - \sum \limits_{m=1}^a k_m \right),E}^n }{\widetilde{p}},
  \end{multline*}
  \begin{multline*}
   \aInnerproduct{\dfrac{1}{2} \nabla \left(e_p^{n + \frac{1}{2}} + e_p^{n - \frac{1}{2}} \right)}{\widetilde{E}} - \sum\limits_{\mathfrak{a} = 1}^{\frac{R}{2} - 1} \sum\limits_{k_\mathfrak{a} = 1}^{\frac{R}{2} - 1 - \sum\limits_{\mathfrak{b} = 1}^{\mathfrak{a} - 1} k_\mathfrak{b}} C_{k_1} \mathfrak{f}(\mathfrak{a}) \aInnerproduct{\dfrac{1}{2} \nabla \left(e_p^{n + \frac{1}{2}} + e_p^{n - \frac{1}{2}} \right)}{\left( \nabla \nabla \cdot \right)^{\sum\limits_{\mathfrak{m} = 1}^\mathfrak{a} k_\mathfrak{m}} \widetilde{E}} \\ + \aInnerproduct{\varepsilon \dfrac{e_E^{n + \frac{1}{2}} - e_E^{n - \frac{1}{2}}}{\Delta t}}{\widetilde{E}} -  \aInnerproduct{\dfrac{1}{2} \left( e_H^{n + 1} + e_H^n \right)}{\nabla \times \widetilde{E}} \\ + \sum\limits_{\mathfrak{a} = 1}^{\frac{R}{2} - 1} \sum\limits_{k_\mathfrak{a} = 1}^{\frac{R}{2} - 1 - \sum\limits_{\mathfrak{b} = 1}^{\mathfrak{a} - 1} k_\mathfrak{b}} C_{k_1} \mathfrak{f}( \mathfrak{a}) (-1)^{\sum\limits_{\mathfrak{m} = 1}^\mathfrak{a} k_\mathfrak{m}} \aInnerproduct{\dfrac{1}{2} ( \mu \varepsilon )^{-\sum\limits_{\mathfrak{m} = 1}^\mathfrak{a} k_\mathfrak{m}} \left( \nabla \times \nabla \times \right)^{\sum\limits_{\mathfrak{m} = 1}^\mathfrak{a} k_\mathfrak{m}} \left( e_H^{n + 1} + e_H^n \right)}{\nabla \times \widetilde{E}} \\ =  \aInnerproduct{\varepsilon R_E^n + \nabla r_{R/2,p}^n - \nabla \times r_{R/2,H}^{n+\frac{1}{2}} - \sum\limits_{\mathfrak{a} = 1}^{\frac{R}{2} - 1} \sum\limits_{k_\mathfrak{a} = 1}^{\frac{R}{2} - 1 - \sum\limits_{\mathfrak{b} = 1}^{\mathfrak{a} - 1} k_\mathfrak{b}} C_{k_1} \mathfrak{f}(\mathfrak{a}) \left(\varepsilon \left( \nabla \nabla \cdot \right)^{\sum\limits_{\mathfrak{m} = 1}^\mathfrak{a} k_\mathfrak{m}} \nabla r_{\left( R/2 - \sum \limits_{m=1}^a k_m \right),p}^n \right. \\ - \left. (-1)^{\sum\limits_{\mathfrak{m} = 1}^\mathfrak{a} k_\mathfrak{m}} (\mu \varepsilon)^{-\sum\limits_{\mathfrak{m} = 1}^\mathfrak{a} k_\mathfrak{m}} \nabla \times (\nabla \times \nabla \times)^{\sum\limits_{\mathfrak{m} = 1}^\mathfrak{a} k_\mathfrak{m}} r_{\left( R/2 - \sum \limits_{m=1}^a k_m \right),H}^{n+\frac{1}{2}}\right)}{\widetilde{E}},
     \end{multline*}
      \begin{multline*}
  \aInnerproduct{\mu \dfrac{e_H^{n + 1} - e_H^n}{\Delta t}}{\widetilde{H}} +  \aInnerproduct{\dfrac{1}{2} \nabla \times \left( e_E^{n + \frac{1}{2}} + e_E^{n - \frac{1}{2}} \right)}{\widetilde{H}} \\ - \sum\limits_{\mathfrak{a} = 1}^{\frac{R}{2} - 1} \sum\limits_{k_\mathfrak{a} = 1}^{\frac{R}{2} - 1 - \sum\limits_{\mathfrak{b} = 1}^{\mathfrak{a} - 1} k_\mathfrak{b}} C_{k_1} \mathfrak{f}(\mathfrak{a}) (-1)^{\sum\limits_{\mathfrak{m} = 1}^\mathfrak{a} k_\mathfrak{m}} \aInnerproduct{\dfrac{1}{2} (\varepsilon \mu)^{-\sum\limits_{\mathfrak{m} = 1}^\mathfrak{a} k_\mathfrak{m}} \nabla \times \left( e_E^{n + \frac{1}{2}} + e_E^{n - \frac{1}{2}} \right)}{\left( \nabla \times  \nabla \times \right)^{\sum\limits_{\mathfrak{m} = 1}^\mathfrak{a} k_\mathfrak{m}} \widetilde{H}} \\ = \aInnerproduct{\mu R_H^{n + \frac{1}{2}} + \nabla \times r_{R/2,E}^n \\ + \sum\limits_{\mathfrak{a} = 1}^{\frac{R}{2} - 1} \sum\limits_{k_\mathfrak{a} = 1}^{\frac{R}{2} - 1 - \sum\limits_{\mathfrak{b} = 1}^{\mathfrak{a} - 1} k_\mathfrak{b}} C_{k_1} \mathfrak{f}(\mathfrak{a}) (-1)^{\sum\limits_{\mathfrak{m} = 1}^\mathfrak{a} k_\mathfrak{m}} (\varepsilon \mu)^{-\sum\limits_{\mathfrak{m} = 1}^\mathfrak{a} k_\mathfrak{m}} (\nabla \times \nabla \times)^{\sum\limits_{\mathfrak{m} = 1}^\mathfrak{a} k_\mathfrak{m}} \nabla \times r_{\left( R/2 - \sum \limits_{m=1}^a k_m \right),E}^n}{\widetilde{H}}.
       \end{multline*}
Likewise, for the semidiscrete approximation of the initial system as in Equations~\labelcref{eqn:maxwell_p0_lfR,eqn:maxwell_E0_lfR,eqn:maxwell_H0_lfR}, we obtain for their errors the following system of equations:
  \begin{multline*}
    \aInnerproduct{\dfrac{e_p^{\frac{1}{2}} - e_{p_0}}{\Delta t/2}}{\widetilde{p}} - \dfrac{1}{2} \aInnerproduct{\dfrac{\varepsilon}{2} \left( e_E^{\frac{1}{2}} + e_{E_0} \right)}{\nabla \widetilde{p}} \\ + \sum\limits_{\mathfrak{a} = 1}^{\frac{R}{2} - 1} \sum\limits_{k_\mathfrak{a} = 1}^{\frac{R}{2} - 1 - \sum\limits_{\mathfrak{b} = 1}^{\mathfrak{a} - 1} k_\mathfrak{b}} \dfrac{C_{k_1}}{2^{2 \sum\limits_{\mathfrak{m} = 1}^\mathfrak{a} k_\mathfrak{m} + 1}} \mathfrak{f}( \mathfrak{a}) \aInnerproduct{\dfrac{\varepsilon}{2} (\nabla \nabla \cdot)^{\sum\limits_{\mathfrak{m} = 1}^\mathfrak{a} k_\mathfrak{m}} \left( e_E^{\frac{1}{2}} + e_{E_0} \right)}{\nabla \widetilde{p}} \\  = \aInnerproduct{\dfrac{R_p^{0}}{2} + \dfrac{\varepsilon}{2} \nabla \cdot r_{R/2,E}^0 - \sum\limits_{\mathfrak{a} = 1}^{\frac{R}{2} - 1} \sum\limits_{k_\mathfrak{a} = 1}^{\frac{R}{2} - 1 - \sum\limits_{\mathfrak{b} = 1}^{\mathfrak{a} - 1} k_\mathfrak{b}} \dfrac{C_{k_1}}{2^{2 \sum\limits_{\mathfrak{m} = 1}^\mathfrak{a} k_\mathfrak{m} + 1}} \mathfrak{f}(\mathfrak{a}) \varepsilon \nabla \cdot (\nabla \nabla \cdot)^{\sum\limits_{\mathfrak{m} = 1}^\mathfrak{a} k_\mathfrak{m}} r_{\left( R/2 - \sum \limits_{m=1}^a k_m \right),E}^0 }{\widetilde{p}},
    \end{multline*}
    \begin{multline*}
  \dfrac{1}{2}  \aInnerproduct{\dfrac{1}{2} \nabla \left(  e_p^{\frac{1}{2}} +  e_{p_0} \right)}{\widetilde{E}} - \sum\limits_{\mathfrak{a} = 1}^{\frac{R}{2} - 1} \sum\limits_{k_\mathfrak{a} = 1}^{\frac{R}{2} - 1 - \sum\limits_{\mathfrak{b} = 1}^{\mathfrak{a} - 1} k_\mathfrak{b}} \dfrac{C_{k_1}}{2^{2 \sum\limits_{\mathfrak{m} = 1}^\mathfrak{a} k_\mathfrak{m} + 1}} \mathfrak{f}(\mathfrak{a}) \aInnerproduct{\dfrac{1}{2} \nabla \left(e_p^{\frac{1}{2}} + e_{p_0} \right)}{\left( \nabla \nabla \cdot \right)^{\sum\limits_{\mathfrak{m} = 1}^\mathfrak{a} k_\mathfrak{m}} \widetilde{E}} \\ + \aInnerproduct{\varepsilon \dfrac{e_E^{\frac{1}{2}} - e_{E_0}}{\Delta t/2}}{\widetilde{E}} - \aInnerproduct{\dfrac{1}{2} \left( e_H^1 + e_{H_0} \right)}{\nabla \times \widetilde{E}} \\ + \sum\limits_{\mathfrak{a} = 1}^{\frac{R}{2} - 1}\sum\limits_{k_\mathfrak{a} = 1}^{\frac{R}{2} - 1 - \sum\limits_{\mathfrak{b} = 1}^{\mathfrak{a} - 1} k_\mathfrak{b}}  \dfrac{C_{k_1}}{2^{2 \sum\limits_{\mathfrak{m} = 1}^\mathfrak{a} k_\mathfrak{m}}} \mathfrak{f}( \mathfrak{a}) (-1)^{\sum\limits_{\mathfrak{m} = 1}^\mathfrak{a} k_\mathfrak{m}} \aInnerproduct{\dfrac{1}{2} ( \mu \varepsilon )^{-\sum\limits_{\mathfrak{m} = 1}^\mathfrak{a} k_\mathfrak{m}} \left( \nabla \times \nabla \times \right)^{\sum\limits_{\mathfrak{m} = 1}^\mathfrak{a} k_\mathfrak{m}} \left( e_H^{1} + e_{H_0} \right)}{\nabla \times \widetilde{E}} \\ = \aInnerproduct{\dfrac{\varepsilon}{2} R_E^0 + \dfrac{1}{2} \nabla r_{R/2,p}^0 - \nabla \times r_{R/2,H}^{\frac{1}{2}} - \sum\limits_{\mathfrak{a} = 1}^{\frac{R}{2} - 1} \sum\limits_{k_\mathfrak{a} = 1}^{\frac{R}{2} - 1 - \sum\limits_{\mathfrak{b} = 1}^{\mathfrak{a} - 1} k_\mathfrak{b}} \dfrac{C_{k_1}}{2^{2 \sum\limits_{\mathfrak{m} = 1}^\mathfrak{a} k_\mathfrak{m} + 1}} \mathfrak{f}(\mathfrak{a}) \left(\varepsilon \left( \nabla \nabla \cdot \right)^{\sum\limits_{\mathfrak{m} = 1}^\mathfrak{a} k_\mathfrak{m}} \nabla r_{\left( R/2 - \sum \limits_{m=1}^a k_m \right),p}^0 \right. \\ - \left. (-1)^{\sum\limits_{\mathfrak{m} = 1}^\mathfrak{a} k_\mathfrak{m}} (\mu \varepsilon)^{-\sum\limits_{\mathfrak{m} = 1}^\mathfrak{a} k_\mathfrak{m}} \nabla \times (\nabla \times \nabla \times)^{\sum\limits_{\mathfrak{m} = 1}^\mathfrak{a} k_\mathfrak{m}} r_{\left( R/2 - \sum \limits_{m=1}^a k_m \right),H}^{\frac{1}{2}}\right)}{\widetilde{E}},
  \end{multline*}
  \begin{multline*}
    \aInnerproduct{\mu \dfrac{e_H^1 - e_{H_0}}{\Delta t}}{\widetilde{H}} + \dfrac{1}{2} \aInnerproduct{\dfrac{1}{2} \nabla \times \left(e_E^{\frac{1}{2}} + e_{E_0} \right)}{\widetilde{H}} \\ - \sum\limits_{\mathfrak{a} = 1}^{\frac{R}{2} - 1} \sum\limits_{k_\mathfrak{a} = 1}^{\frac{R}{2} - 1 - \sum\limits_{\mathfrak{b} = 1}^{\mathfrak{a} - 1} k_\mathfrak{b}} \dfrac{C_{k_1}}{2^{2 \sum\limits_{\mathfrak{m} = 1}^\mathfrak{a} k_\mathfrak{m} + 1}} \mathfrak{f}(\mathfrak{a}) (-1)^{\sum\limits_{\mathfrak{m} = 1}^\mathfrak{a} k_\mathfrak{m}} \aInnerproduct{\dfrac{1}{2} (\varepsilon \mu)^{-\sum\limits_{\mathfrak{m} = 1}^\mathfrak{a} k_\mathfrak{m}} \nabla \times \left( e_E^{\frac{1}{2}} + e_{E_0} \right)}{\left( \nabla \times  \nabla \times \right)^{\sum\limits_{\mathfrak{m} = 1}^\mathfrak{a} k_\mathfrak{m}} \widetilde{H}} \\ = \aInnerproduct{\mu R_H^{\frac{1}{2}} + \dfrac{1}{2} \nabla \times r_{R/2,E}^0 \\ + \sum\limits_{\mathfrak{a} = 1}^{\frac{R}{2} - 1} \sum\limits_{k_\mathfrak{a} = 1}^{\frac{R}{2} - 1 - \sum\limits_{\mathfrak{b} = 1}^{\mathfrak{a} - 1} k_\mathfrak{b}} \dfrac{C_{k_1}}{2^{2 \sum\limits_{\mathfrak{m} = 1}^\mathfrak{a} k_\mathfrak{m} + 1}} \mathfrak{f}(\mathfrak{a}) (-1)^{\sum\limits_{\mathfrak{m} = 1}^\mathfrak{a} k_\mathfrak{m}} (\varepsilon \mu)^{-\sum\limits_{\mathfrak{m} = 1}^\mathfrak{a} k_\mathfrak{m}} (\nabla \times \nabla \times)^{\sum\limits_{\mathfrak{m} = 1}^\mathfrak{a} k_\mathfrak{m}} \nabla \times r_{\left( R/2 - \sum \limits_{m=1}^a k_m \right),E}^0}{\widetilde{H}},
    \end{multline*}
and for which we define the initial remainder terms as follows:
\begin{align*}
R^0_p & \coloneq \dfrac{1}{\Delta t} \left[ \int\limits_{t^0}^{t^\frac{1}{4}} \dfrac{(t^0 - s)^R}{R!} \dfrac{\partial^{R+1} p}{\partial t^{R+1}}(s) ds + \int\limits_{t^\frac{1}{4}}^{t^{\frac{1}{2}}} \dfrac{(t^{\frac{1}{2}} - s)^R}{R!} \dfrac{\partial^{R+1} p}{\partial t^{R+1}}(s) ds \right], \\
R^0_E & \coloneq \dfrac{1}{\Delta t} \left[ \int\limits_{t^0}^{t^\frac{1}{4}} \dfrac{(t^0 - s)^R}{R!} \dfrac{\partial^{R+1} E}{\partial t^{R+1}}(s) ds + \int\limits_{t^\frac{1}{4}}^{t^{\frac{1}{2}}} \dfrac{(t^{\frac{1}{2}} - s)^R}{R!} \dfrac{\partial^{R+1} E}{\partial t^{R+1}}(s) ds \right], \\
R^{ \frac{1}{2}}_H & \coloneq \dfrac{1}{\Delta t} \left[ \int\limits_{t^0}^{t^{ \frac{1}{2}}} \dfrac{(t^0 - s)^R}{R!} \dfrac{\partial^{R+1} H}{\partial t^{R+1}}(s) ds + \int\limits_{t^{\frac{1}{2}}}^{t^{1}} \dfrac{(t^{1} - s)^R}{R!} \dfrac{\partial^{R+1} H}{\partial t^{R+1}}(s) ds\right], \\
 r^0_{S/2,p} &\coloneq \dfrac{1}{2} \left[-\int\limits_{t^0}^{t^\frac{1}{4}} \dfrac{(t^0 - s)^{S-1}}{(S-1)!} \dfrac{\partial^{S} p}{\partial t^{S}}(s) ds + \int\limits_{t^\frac{1}{4}}^{t^{\frac{1}{2}}} \dfrac{(t^{\frac{1}{2}} - s)^{S-1}}{(S-1)!} \dfrac{\partial^{S} p}{\partial t^{S}}(s) ds \right], \\
  r^0_{S/2,E} &\coloneq \dfrac{1}{2} \left[-\int\limits_{t^0}^{t^\frac{1}{4}} \dfrac{(t^0 - s)^{S-1}}{(S-1)!} \dfrac{\partial^{S} E}{\partial t^{S}}(s) ds + \int\limits_{t^\frac{1}{4}}^{t^{\frac{1}{2}}} \dfrac{(t^{\frac{1}{2}} - s)^{S-1}}{(S-1)!} \dfrac{\partial^{S} E}{\partial t^{S}}(s) ds \right], \\
r^{\frac{1}{2}}_{S/2,H} &\coloneq \dfrac{1}{2} \left[-\int\limits_{t^0}^{t^{\frac{1}{2}}} \dfrac{(t^0 - s)^{S-1}}{(S-1)!} \dfrac{\partial^{S} H}{\partial t^{S}}(s) ds + \int\limits_{t^{\frac{1}{2}}}^{t^{1}} \dfrac{(t^{1} - s)^{S-1}}{(S-1)!} \dfrac{\partial^{S} H}{\partial t^{S}}(s) ds\right],
\end{align*}
for $2 \leqslant S \leqslant R$ with $S$ even. Now, in this set of weak formulation equations for the errors, we choose the test functions to be $\widetilde{p} = 2 \Delta t \varepsilon^{-1} \left( e_p^{n + \frac{1}{2}} + e_p^{n - \frac{1}{2}} \right)$, $\widetilde{E} = 2 \Delta t \left( e_E^{n + \frac{1}{2}} + e_E^{n - \frac{1}{2}} \right)$ and $\widetilde{H} = 2 \Delta t \left( e_H^n + e_H^{n - 1} \right)$. Next, by following essentially the same sequence of steps as in Theorem~\ref{thm:dscrt_enrgy_estmt_lfR}, we obtain the estimate for these error terms to be:
\begin{multline*}
  \norm{e_p^{n + \frac{1}{2}}}^2_{\varepsilon^{-1}} - \norm{e_p^{n - \frac{1}{2}}}^2_{\varepsilon^{-1}} + \norm{e_E^{n + \frac{1}{2}}}^2_{\varepsilon} - \norm{e_E^{n - \frac{1}{2}}}^2_{\varepsilon} + \norm{e_H^{n+1}}^2_{\mu} -  \norm{e_H^n}^2_{\mu} \le \Delta t \left[ \norm{e_p^{n + \frac{1}{2}}}^2_{\varepsilon^{-1}} + \norm{e_p^{n - \frac{1}{2}}}^2_{\varepsilon^{-1}} \right. \\ \left. + \norm{e_E^{n + \frac{1}{2}}}^2_{\varepsilon} + \norm{e_E^{n - \frac{1}{2}}}^2_{\varepsilon} + \norm{e_H^{n+1}}^2_{\mu} +  \norm{e_H^n}^2_{\mu}\right] +
 \Delta t \left[ \norm{R_p^n}^2_{\varepsilon^{-1}} + \norm{R_E^n}^2_{\varepsilon} + \norm{R_H^{n + \frac{1}{2}}}^2_{\mu} + \norm{\nabla r_{R/2,p}^n} ^2_{\varepsilon^{-1}} \right. \\ \left. + \norm{\nabla \cdot r_{R/2,E}^n}^2_{\varepsilon} + \varepsilon^{-1} \mu^{-1} \norm{\nabla \times r_{R/2,E}^n}^2_{\varepsilon} +  \varepsilon^{-1} \mu^{-1} \norm{\nabla \times r_{R/2,H}^{n + \frac{1}{2}}}^2_{\mu} \right. \\ 
\left. +\sum\limits_{\mathfrak{a} = 1}^{\frac{R}{2} - 1} \sum\limits_{k_\mathfrak{a} = 1}^{\frac{R}{2} - 1 - \sum\limits_{\mathfrak{b} = 1}^{\mathfrak{a} - 1} k_\mathfrak{b}} C_{k_1}^2 \mathfrak{f}(\mathfrak{a})^2 \left(\norm{\left( \nabla \nabla \cdot \right)^{\sum\limits_{\mathfrak{m} = 1}^\mathfrak{a} k_\mathfrak{m}} \nabla r_{\left( R/2 - \sum \limits_{m=1}^a k_m \right),p}^n}^2_{\varepsilon^{-1}} + \norm{\nabla \cdot (\nabla \nabla \cdot)^{\sum\limits_{\mathfrak{m} = 1}^\mathfrak{a} k_\mathfrak{m}} r_{\left( R/2 - \sum \limits_{m=1}^a k_m \right),E}^n }^2_\varepsilon  \right. \right. \\ \left.\left. + (\varepsilon \mu)^{-(2 \sum\limits_{\mathfrak{m} = 1}^\mathfrak{a} k_\mathfrak{m} + 1)} \norm{(\nabla \times \nabla \times)^{\sum\limits_{\mathfrak{m} = 1}^\mathfrak{a} k_\mathfrak{m}} \nabla \times r_{\left( R/2 - \sum \limits_{m=1}^a k_m \right),E}^n}^2_\varepsilon \right. \right. \\ \left.\left. + (\mu \varepsilon)^{-(2 \sum\limits_{\mathfrak{m} = 1}^\mathfrak{a} k_\mathfrak{m} + 1)} \norm{ \nabla \times (\nabla \times \nabla \times)^{\sum\limits_{\mathfrak{m} = 1}^\mathfrak{a} k_\mathfrak{m}} r_{\left( R/2 - \sum \limits_{m=1}^a k_m \right),H}^{n+\frac{1}{2}}}^2_\mu \right) \right].
\end{multline*}
Now, summing over $n = 1$ to $N-1$, we get that:
\begin{multline*}
  \norm{e_p^{N - \frac{1}{2}}}^2_{\varepsilon^{-1}} - \norm{e_p^{ \frac{1}{2}}}^2_{\varepsilon^{-1}} + \norm{e_E^{N - \frac{1}{2}}}^2_{\varepsilon} - \norm{e_E^{\frac{1}{2}}}^2_{\varepsilon} + \norm{e_H^N}^2_{\mu} -  \norm{e_H^1}^2_{\mu} \le \Delta t \left[ \norm{e_p^{N - \frac{1}{2}}}^2_{\varepsilon^{-1}} + \norm{e_p^{ \frac{1}{2}}}^2_{\varepsilon^{-1}} + \norm{e_E^{N - \frac{1}{2}}}^2_{\varepsilon} \right. \\ \left. + \norm{e_E^{\frac{1}{2}}}^2_{\varepsilon} + \norm{e_H^N}^2_{\mu} + \norm{e_H^1}^2_{\mu}\right] + 2 \Delta t \sum\limits_{n = 1}^{N - 2} \left[ \norm{e_p^{n + \frac{1}{2}}}^2_{\varepsilon^{-1}} + \norm{e_E^{n + \frac{1}{2}}}^2_{\varepsilon} + \norm{e_H^{n+1}}^2_{\mu}\right] +
 \Delta t \sum\limits_{n = 1}^{N - 1} \left[ \norm{R_p^n}^2_{\varepsilon^{-1}} + \norm{R_E^n}^2_{\varepsilon} \right. \\ \left. + \norm{R_H^{n + \frac{1}{2}}}^2_{\mu} + \norm{\nabla r_{R/2,p}^n} ^2_{\varepsilon^{-1}} + \norm{\nabla \cdot r_{R/2,E}^n}^2_{\varepsilon} + \varepsilon^{-1} \mu^{-1} \norm{\nabla \times r_{R/2,E}^n}^2_{\varepsilon} +  \varepsilon^{-1} \mu^{-1} \norm{\nabla \times r_{R/2,H}^{n + \frac{1}{2}}}^2_{\mu} \right. \\ \left. +\sum\limits_{\mathfrak{a} = 1}^{\frac{R}{2} - 1} \sum\limits_{k_\mathfrak{a} = 1}^{\frac{R}{2} - 1 - \sum\limits_{\mathfrak{b} = 1}^{\mathfrak{a} - 1} k_\mathfrak{b}} C_{k_1}^2 \mathfrak{f}(\mathfrak{a})^2 \left(\norm{\left( \nabla \nabla \cdot \right)^{\sum\limits_{\mathfrak{m} = 1}^\mathfrak{a} k_\mathfrak{m}} \nabla r_{\left( R/2 - \sum \limits_{m=1}^a k_m \right),p}^n}^2_{\varepsilon^{-1}} + \norm{\nabla \cdot (\nabla \nabla \cdot)^{\sum\limits_{\mathfrak{m} = 1}^\mathfrak{a} k_\mathfrak{m}} r_{\left( R/2 - \sum \limits_{m=1}^a k_m \right),E}^n }^2_\varepsilon  \right. \right. \\ \left.\left. + (\varepsilon \mu)^{-(2 \sum\limits_{\mathfrak{m} = 1}^\mathfrak{a} k_\mathfrak{m} + 1)} \norm{(\nabla \times \nabla \times)^{\sum\limits_{\mathfrak{m} = 1}^\mathfrak{a} k_\mathfrak{m}} \nabla \times r_{\left( R/2 - \sum \limits_{m=1}^a k_m \right),E}^n}^2_\varepsilon \right. \right. \\ \left.\left. + (\mu \varepsilon)^{-(2 \sum\limits_{\mathfrak{m} = 1}^\mathfrak{a} k_\mathfrak{m} + 1)} \norm{ \nabla \times (\nabla \times \nabla \times)^{\sum\limits_{\mathfrak{m} = 1}^\mathfrak{a} k_\mathfrak{m}} r_{\left( R/2 - \sum \limits_{m=1}^a k_m \right),H}^{n+\frac{1}{2}}}^2_\mu \right) \right].
\end{multline*}
Similarly, for the case of the initial equations, choosing test functions to be $\widetilde{p} = 2 \Delta t \varepsilon^{-1} \left( e_p^\frac{1}{2} + e_p^0 \right)$, $\widetilde{E} = 2 \Delta t \left( e_E^\frac{1}{2}  + e_E^0 \right)$ and $\widetilde{H} = 4 \Delta t \left( e_H^1 + e_H^0 \right)$, we get that:
\begin{multline*}
  \norm{e_p^{\frac{1}{2}}}^2_{\varepsilon^{-1}} - \norm{e_p^0}^2_{\varepsilon^{-1}} + \norm{e_E^{\frac{1}{2}}}^2_{\varepsilon} - \norm{e_E^0}^2_{\varepsilon} + \norm{e_H^{1}}^2_{\mu} -  \norm{e_H^0}^2_{\mu} \\ \le \Delta t \left[   \norm{e_p^{\frac{1}{2}}}^2_{\varepsilon^{-1}} + \norm{e_p^0}^2_{\varepsilon^{-1}} + \norm{e_E^{\frac{1}{2}}}^2_{\varepsilon} + \norm{e_E^0}^2_{\varepsilon} + \norm{e_H^{1}}^2_{\mu} +  \norm{e_H^0}^2_{\mu} \right] + \Delta t \left[ \norm{R_p^0}^2_{\varepsilon^{-1}} + \norm{R_E^0}^2_{\varepsilon} + \norm{R_H^{\frac{1}{2}}}^2_{\mu} \right. \\ \left. + \norm{\nabla r_{R/2,p}^0} ^2_{\varepsilon^{-1}} + \norm{\nabla \cdot r_{R/2,E}^0}^2_{\varepsilon} + \varepsilon^{-1} \mu^{-1} \norm{\nabla \times r_{R/2,E}^0}^2_{\varepsilon} +  \varepsilon^{-1} \mu^{-1} \norm{\nabla \times r_{R/2,H}^{\frac{1}{2}}}^2_{\mu} \right. \\ 
\left. +\sum\limits_{\mathfrak{a} = 1}^{\frac{R}{2} - 1} \sum\limits_{k_\mathfrak{a} = 1}^{\frac{R}{2} - 1 - \sum\limits_{\mathfrak{b} = 1}^{\mathfrak{a} - 1} k_\mathfrak{b}} \dfrac{C_{k_1}^2}{4^{2 \sum\limits_{\mathfrak{m} = 1}^\mathfrak{a} k_\mathfrak{m} + 1}} \mathfrak{f}(\mathfrak{a})^2 \left(\norm{\left( \nabla \nabla \cdot \right)^{\sum\limits_{\mathfrak{m} = 1}^\mathfrak{a} k_\mathfrak{m}} \nabla r_{\left( R/2 - \sum \limits_{m=1}^a k_m \right),p}^0}^2_{\varepsilon^{-1}}  \right. \right. \\ \left.\left. + \norm{\nabla \cdot (\nabla \nabla \cdot)^{\sum\limits_{\mathfrak{m} = 1}^\mathfrak{a} k_\mathfrak{m}} r_{\left( R/2 - \sum \limits_{m=1}^a k_m \right),E}^0}^2_\varepsilon + (\varepsilon \mu)^{-(2 \sum\limits_{\mathfrak{m} = 1}^\mathfrak{a} k_\mathfrak{m} + 1)} \norm{(\nabla \times \nabla \times)^{\sum\limits_{\mathfrak{m} = 1}^\mathfrak{a} k_\mathfrak{m}} \nabla \times r_{\left( R/2 - \sum \limits_{m=1}^a k_m \right),E}^0}^2_\varepsilon \right. \right. \\ \left.\left. + (\mu \varepsilon)^{-(2 \sum\limits_{\mathfrak{m} = 1}^\mathfrak{a} k_\mathfrak{m} + 1)} \norm{ \nabla \times (\nabla \times \nabla \times)^{\sum\limits_{\mathfrak{m} = 1}^\mathfrak{a} k_\mathfrak{m}} r_{\left( R/2 - \sum \limits_{m=1}^a k_m \right),H}^{\frac{1}{2}}}^2_\mu \right) \right].
\end{multline*}
Adding the previous two equations and using the initial conditions as in Equation~\eqref{eqn:ICs} and positivity of all the right-hand side terms, we have that:
\begin{multline*}
  \norm{e_p^{N - \frac{1}{2}}}^2_{\varepsilon^{-1}} + \norm{e_E^{N - \frac{1}{2}}}^2_{\varepsilon} + \norm{e_H^N}^2_{\mu} \le \dfrac{1 + \Delta t}{1 - \Delta t} \left[ \norm{e_p^0}^2_{\varepsilon^{-1}} + \norm{e_E^0}^2_{\varepsilon} + \norm{e_H^0}^2_{\mu} \right] + \\
  \dfrac{\Delta t}{1 - \Delta t} \sum\limits_{n = 0}^{N - 1} \left[ 2 \left( \norm{e_p^{n + \frac{1}{2}}}^2_{\varepsilon^{-1}} + \norm{e_E^{n + \frac{1}{2}}}^2_{\varepsilon} + \norm{e_H^{n + 1}}^2_{\mu} \right) + \left( \norm{R_p^n}^2_{\varepsilon^{-1}} + \norm{R_E^n}^2_{\varepsilon} + \norm{R_H^{n + \frac{1}{2}}}^2_{\mu} \right. \right. \\ + \left.\left. \norm{\nabla r_{R/2,p}^n} ^2_{\varepsilon^{-1}} + \norm{\nabla \cdot r_{R/2,E}^n}^2_{\varepsilon} + \varepsilon^{-1} \mu^{-1} \norm{\nabla \times r_{R/2,E}^n}^2_{\varepsilon} +  \varepsilon^{-1} \mu^{-1} \norm{\nabla \times r_{R/2,H}^{n + \frac{1}{2}}}^2_{\mu} \right. \right. \\ 
\left. \left. +\sum\limits_{\mathfrak{a} = 1}^{\frac{R}{2} - 1} \sum\limits_{k_\mathfrak{a} = 1}^{\frac{R}{2} - 1 - \sum\limits_{\mathfrak{b} = 1}^{\mathfrak{a} - 1} k_\mathfrak{b}} C_{k_1}^2 \mathfrak{f}(\mathfrak{a})^2 \left(\norm{\left( \nabla \nabla \cdot \right)^{\sum\limits_{\mathfrak{m} = 1}^\mathfrak{a} k_\mathfrak{m}} \nabla r_{\left( R/2 - \sum \limits_{m=1}^a k_m \right),p}^n}^2_{\varepsilon^{-1}} + \norm{\nabla \cdot (\nabla \nabla \cdot)^{\sum\limits_{\mathfrak{m} = 1}^\mathfrak{a} k_\mathfrak{m}} r_{\left( R/2 - \sum \limits_{m=1}^a k_m \right),E}^n }^2_\varepsilon  \right. \right. \right. \\ \left. \left. \left. + (\varepsilon \mu)^{-(2 \sum\limits_{\mathfrak{m} = 1}^\mathfrak{a} k_\mathfrak{m} + 1)} \norm{(\nabla \times \nabla \times)^{\sum\limits_{\mathfrak{m} = 1}^\mathfrak{a} k_\mathfrak{m}} \nabla \times r_{\left( R/2 - \sum \limits_{m=1}^a k_m \right),E}^n}^2_\varepsilon \right. \right. \right. \\ \left.\left.\left.+ (\mu \varepsilon)^{-(2 \sum\limits_{\mathfrak{m} = 1}^\mathfrak{a} k_\mathfrak{m} + 1)} \norm{ \nabla \times (\nabla \times \nabla \times)^{\sum\limits_{\mathfrak{m} = 1}^\mathfrak{a} k_\mathfrak{m}} r_{\left( R/2 - \sum \limits_{m=1}^a k_m \right),H}^{n+\frac{1}{2}}}^2_\mu \right) \right) \right].
\end{multline*}
Applying the discrete Gronwall inequality similar to Theorem~\ref{thm:dscrt_enrgy_estmt_lfR}, we obtain the estimate:
\begin{multline*}
  \norm{e_p^{N - \frac{1}{2}}}^2_{\varepsilon^{-1}} + \norm{e_E^{N - \frac{1}{2}}}^2_{\varepsilon} +\norm{e_H^N}^2_{\mu} \le \left[ \dfrac{6 \Delta t}{5} \sum\limits_{n = 0}^{N - 1} \left( \norm{R_p^n}^2_{\varepsilon^{-1}} + \norm{R_E^n}^2_{\varepsilon} + \norm{R_H^{n + \frac{1}{2}}}^2_{\mu} \right. \right. \\ + \left.\left. \norm{\nabla r_{R/2,p}^n} ^2_{\varepsilon^{-1}} + \norm{\nabla \cdot r_{R/2,E}^n}^2_{\varepsilon} + \varepsilon^{-1} \mu^{-1} \norm{\nabla \times r_{R/2,E}^n}^2_{\varepsilon} +  \varepsilon^{-1} \mu^{-1} \norm{\nabla \times r_{R/2,H}^{n + \frac{1}{2}}}^2_{\mu} \right. \right. \\ 
\left. \left. +\sum\limits_{\mathfrak{a} = 1}^{\frac{R}{2} - 1} \sum\limits_{k_\mathfrak{a} = 1}^{\frac{R}{2} - 1 - \sum\limits_{\mathfrak{b} = 1}^{\mathfrak{a} - 1} k_\mathfrak{b}} C_{k_1}^2 \mathfrak{f}(\mathfrak{a})^2 \left(\norm{\left( \nabla \nabla \cdot \right)^{\sum\limits_{\mathfrak{m} = 1}^\mathfrak{a} k_\mathfrak{m}} \nabla r_{\left( R/2 - \sum \limits_{m=1}^a k_m \right),p}^n}^2_{\varepsilon^{-1}} + \norm{\nabla \cdot (\nabla \nabla \cdot)^{\sum\limits_{\mathfrak{m} = 1}^\mathfrak{a} k_\mathfrak{m}} r_{\left( R/2 - \sum \limits_{m=1}^a k_m \right),E}^n }^2_\varepsilon  \right. \right. \right. \\ \left. \left. \left. + (\varepsilon \mu)^{-(2 \sum\limits_{\mathfrak{m} = 1}^\mathfrak{a} k_\mathfrak{m} + 1)} \norm{(\nabla \times \nabla \times)^{\sum\limits_{\mathfrak{m} = 1}^\mathfrak{a} k_\mathfrak{m}} \nabla \times r_{\left( R/2 - \sum \limits_{m=1}^a k_m \right),E}^n}^2_\varepsilon  \right. \right. \right. \\ \left. \left. \left. + (\mu \varepsilon)^{-(2 \sum\limits_{\mathfrak{m} = 1}^\mathfrak{a} k_\mathfrak{m} + 1)} \norm{ \nabla \times (\nabla \times \nabla \times)^{\sum\limits_{\mathfrak{m} = 1}^\mathfrak{a} k_\mathfrak{m}} r_{\left( R/2 - \sum \limits_{m=1}^a k_m \right),H}^{n+\frac{1}{2}}}^2_\mu \right) \right) + \dfrac{7}{5} \left(\norm{e_p^0}^2_{\varepsilon^{-1}} + \norm{e_E^0}^2_{\varepsilon}  + \norm{e_H^0}^2_\mu \right)\right] \exp\left( 4 T \right).
\end{multline*}
Now, we need to obtain bounding estimates for each of the Taylor remainder terms and to do so, we first consider the first remainder term corresponding to $p$ and argue as follows:
\begin{align*}
\norm{R^n_p}^2_{\varepsilon^{-1}} &= \dfrac{1}{\left(R! \Delta t \right)^2} \norm[\bigg]{\int\limits_{\mathclap{t^{n - 1}}}^{\mathclap{t^{n - \frac{1}{2}}}} (t^{n - 1} - s)^R \dfrac{\partial^{R+1} p}{\partial t^{R+1}}(s) ds + \int\limits_{\mathclap{t^{n - \frac{1}{2}}}}^{t^n} (t^n - s)^R \dfrac{\partial^{R+1} p}{\partial t^{R+1}}(s) ds}^2_{\varepsilon^{-1}}, \\
&\le \dfrac{1}{\left(R! \Delta t \right)^2} \norm[\bigg]{\int\limits_{t^{n - 1}}^{t^n} (t^n - s)^R \dfrac{\partial^{R+1} p}{\partial t^{R+1}}(s) ds}^2_{\varepsilon^{-1}}, \quad \text{(using $t^{n - 1} < t^n$)} \\
&\le \dfrac{1}{\left(R! \Delta t\right)^2} \int\limits_{t^{n - 1}}^{t^n} (s - t^n)^{2R} ds \int\limits_{\mathclap{t^{n - 1}}}^{t^n} \norm[\bigg]{\dfrac{\partial^{R+1} p}{\partial t^{R+1}}(s)}^2_{\varepsilon^{-1}} ds, \quad \text{(by Cauchy-Schwarz)} \\
&= \dfrac{\left(\Delta t \right)^{2R-1}}{(R!)^2 (2R+1)} \int\limits_{\mathclap{t^{n - 1}}}^{t^n} \norm[\bigg]{\dfrac{\partial^{R+1} p}{\partial t^{R+1}}(s)}^2_{\varepsilon^{-1}} ds,
\end{align*}
and now summing both sides over $n = 0$ to $N$, we have that:
\begin{equation*}
  \sum\limits_{n = 0}^N \norm{R^n_p}^2_{\varepsilon^{-1}} \le \dfrac{\left(\Delta t \right)^{2R-1}}{(R!)^2 (2R+1)} \int\limits_0^T \norm[\bigg]{\dfrac{\partial^{R+1} p}{\partial t^{R+1}}(s)}^2_{\varepsilon^{-1}} ds = \dfrac{\left(\Delta t \right)^{2R-1}}{(R!)^2 (2R+1)} \norm[\bigg]{\dfrac{\partial^{R+1} p}{\partial t^{R+1}}}^2_{L^2(0, T; L^2_{\varepsilon^{-1}}(\Omega))}.
\end{equation*}
Similarly, for the remaining Taylor remainder terms, we have that:
\begin{align*}
\norm{R^n_E}^2_{\varepsilon} & \le \dfrac{\left(\Delta t \right)^{2R-1}}{(R!)^2 (2R+1)} \norm[\bigg]{\dfrac{\partial^{R+1} E}{\partial t^{R+1}}}^2_{L^2(0, T; L^2_\varepsilon(\Omega))}, \\
  \sum\limits_{n = 0}^{N - 1} \norm{R^{n + \frac{1}{2}}_H}^2_\mu &\le \dfrac{\left(\Delta t \right)^{2R-1}}{(R!)^2 (2R+1)}\norm[\bigg]{\dfrac{\partial^{R+1} H}{\partial t^{R+1}}}^2_{L^2(0, T; L^2_\mu(\Omega))}, \\
  \sum\limits_{n = 0}^{N - 1} \norm{\Phi (r^n_{S/2,p})}^2_{\varepsilon^{-1}} &\le  \dfrac{\left( \Delta t \right)^{2S-1}}{2^{2S-1} (2S-1)!} \norm[\bigg]{\dfrac{\partial^S \left(\Phi (p) \right)}{\partial t^S}}^2_{L^2(0, T; L^2_{\varepsilon^{-1}}(\Omega))}, \\
  \sum\limits_{n = 0}^{N - 1} \norm{\Phi (r^n_{S/2,E})}^2_{\varepsilon} &\le \dfrac{\left( \Delta t \right)^{2S-1}}{2^{2S-1} (2S-1)!} \norm[\bigg]{\dfrac{\partial^S \left(\Phi (E) \right)}{\partial t^S}}^2_{L^2(0, T; L^2_\varepsilon(\Omega))},  \\
  \sum\limits_{n = 0}^{N - 1} \norm{\Phi (r^{n + \frac{1}{2}}_{S/2,H})}^2_\mu & \le \dfrac{\left( \Delta t \right)^{2S-1}}{2^{2S-1} (2S-1)!} \norm[\bigg]{\dfrac{\partial^S (\Phi (H))}{\partial t^S}}^2_{L^2(0, T; L^2_\mu(\Omega))},
\end{align*}
for $2 \leqslant S \leqslant R$, $S$ even, and in which $\Phi$ is taken to be an operator (for example, $\nabla \times \nabla \times$ or $\nabla \nabla \cdot$) acting on objects taken to be from a correspondingly appropriate function space on $\Omega$. Finally, using the regularity assumptions for $p$, $E$ and $H$, and $1$- and $2$-norm equivalence, we obtain the required result:
\[
  \norm{e_p^{N - \frac{1}{2}}}_{\varepsilon^{-1}} + \norm{e_E^{N - \frac{1}{2}}}_{\varepsilon} + \norm{e_H^N}_{\mu} \le C \left[ (\Delta t)^R + \norm{e_p^0}_{\varepsilon^{-1}} + \norm{e_E^0}_{\varepsilon} + \norm{e_H^0}_{\mu} \right]. \qedhere
\]
\end{proof}

\subsection{Error Estimate for Full Discretization}

We now present the error analysis for the full discretization of the three-field formulation of the Maxwell's equations using a compatible sequence of arbitrary order simplicial de Rham finite elements in conjunction with the LF$_R$ scheme. To do so, first we let $\Pi_h^0$, $\Pi_h^1$ and $\Pi_h^2$ denote the respective smoothed $L^2$ projection operators, that is, let $\Pi_h^0: \mathring{H}^1_{\varepsilon^{-1}}(\Omega) \longto U_h$, $\Pi_h^1: \mathring{H}_{\varepsilon}(\curl; \Omega) \longto V_h$ and $\Pi_h^2: \mathring{H}_{\mu}(\divgn; \Omega) \longto W_h$. The details about these operators has become somewhat standardized, and expositions be found in many articles such as \cite{Schoberl2008,Christiansen2007,ArFaWi2006,Arnold2018,Ern2021}. We also refer to the earlier works in~\cite{ArKa2025,ArKa2026} for more specific details as it pertains to the $(p, E, H$) Maxwell's system.

Given this brief background about these smoothed $L^2$ projection operators, we now define the errors for $p$, $E$ and $H$ at time $(n + 1/2) \Delta t$ or $n \Delta t$ under the full discretization to be the following and this is essentially a generalization of the analysis for the full error as in \cite[Section 5.2]{ArKa2025} for LF$_2$ and \cite[Section 2.2]{ArKa2026} for LF$_4$. So, we have that: \begin{alignat}{2}
  e_{p_h}^{n + \frac{1}{2}} &\coloneq p(t^{n + \frac{1}{2}}) - p_h^{n + \frac{1}{2}} &&= \eta^{n+\frac{1}{2}} - \eta_h^{n+\frac{1}{2}}, \label{eqn:p_fullerror_lfR} \\
  e_{E_h}^{n + \frac{1}{2}} &\coloneq E(t^{n + \frac{1}{2}}) - E_h^{n + \frac{1}{2}} &&= \zeta^{n + \frac{1}{2}} - \zeta_h^{n + \frac{1}{2}}, \label{eqn:E_fullerror_lfR} \\
  e_{H_h}^n &\coloneq H(t^n) - H_h^n &&= \xi^n - \xi_h^n, \label{eqn:H_fullerror_lfR}
\end{alignat}
and in which we now have the following definitions for the newly introduced terms:
\begin{alignat}{3}
  \eta^{n + \frac{1}{2}} &\coloneq p(t^{n + \frac{1}{2}}) - \Pi_h^0 p(t^{n + \frac{1}{2}}), &&\qquad \eta_h^{n + \frac{1}{2}} &&\coloneq p_h^{n + \frac{1}{2}}  - \Pi_h^0 p(t^{n + \frac{1}{2}}), \label{eqn:p_fullerror_sub_lfR} \\
  \zeta^{n + \frac{1}{2}} &\coloneq E(t^{n + \frac{1}{2}}) - \Pi_h^1 E(t^{n + \frac{1}{2}}), &&\qquad \zeta_h^{n + \frac{1}{2}} &&\coloneq E_h^{n + \frac{1}{2}} - \Pi_h^1 E(t^{n + \frac{1}{2}}), \label{eqn:E_fullerror_sub_lfR} \\
  \xi^n &\coloneq H(t^n) - \Pi_h^2 H(t^n), &&\qquad \xi_h^n &&\coloneq H_h^n - \Pi_h^2 H(t^n). \label{eqn:H_fullerror_sub_lfR}
\end{alignat}
For the LF$_R$ scheme as in Equations~\labelcref{eqn:maxwell_p_lfR,eqn:maxwell_E_lfR,eqn:maxwell_H_lfR}, using a de Rham sequence of finite dimensional subspaces of the corresponding function spaces for the spatial discretization of $(p^{n + \frac{1}{2}}, E^{n + \frac{1}{2}}, H^{n + 1})$, we obtain the following discrete problem: find $(p_h^{n + \frac{1}{2}}, E_h^{n + \frac{1}{2}}, H_h^{n + 1}) \in U_h \times V_h \times W_h \subseteq \mathring{H}_{\varepsilon^{-1}}^1 \times \mathring{H}_{\varepsilon}(\curl; \Omega) \times \mathring{H}_{\mu}(\divgn; \Omega)$ such that:
\begin{subequations}
\begin{multline}
  \aInnerproduct{\dfrac{p_h^{n + \frac{1}{2}} - p_h^{n - \frac{1}{2}}}{\Delta t}}{\widetilde{p}} - \aInnerproduct{\dfrac{\varepsilon}{2} \left( E_h^{n + \frac{1}{2}} + E_h^{n - \frac{1}{2}} \right)}{\nabla \widetilde{p}} \\ + \sum\limits_{\mathfrak{a} = 1}^{\frac{R}{2} - 1} \sum\limits_{k_\mathfrak{a} = 1}^{\frac{R}{2} - 1 - \sum\limits_{\mathfrak{b} = 1}^{\mathfrak{a} - 1} k_\mathfrak{b}} C_{k_1} \mathfrak{f}(\mathfrak{a}) \aInnerproduct{\dfrac{\varepsilon}{2} (\nabla \nabla \cdot)^{\sum\limits_{\mathfrak{m} = 1}^\mathfrak{a} k_\mathfrak{m}} \left( E_h^{n + \frac{1}{2}} + E_h^{n - \frac{1}{2}} \right)}{\nabla \widetilde{p}} =0, \label{eqn:maxwell_p_lfR_full}
  \end{multline} \\
  \begin{multline}
   \aInnerproduct{\dfrac{1}{2} \nabla \left(p_h^{n + \frac{1}{2}} + p_h^{n - \frac{1}{2}} \right)}{\widetilde{E}} - \sum\limits_{\mathfrak{a} = 1}^{\frac{R}{2} - 1} \sum\limits_{k_\mathfrak{a} = 1}^{\frac{R}{2} - 1 - \sum\limits_{\mathfrak{b} = 1}^{\mathfrak{a} - 1} k_\mathfrak{b}} C_{k_1} \mathfrak{f}(\mathfrak{a}) \aInnerproduct{\dfrac{1}{2} \nabla \left(p_h^{n + \frac{1}{2}} + p_h^{n - \frac{1}{2}} \right)}{\left( \nabla \nabla \cdot \right)^{\sum\limits_{\mathfrak{m} = 1}^\mathfrak{a} k_\mathfrak{m}} \widetilde{E}} \\ + \aInnerproduct{\varepsilon \dfrac{E_h^{n + \frac{1}{2}} - E_h^{n - \frac{1}{2}}}{\Delta t}}{\widetilde{E}} -  \aInnerproduct{\dfrac{1}{2} \left( H_h^{n + 1} + H_h^n \right)}{\nabla \times \widetilde{E}} \\  + \sum\limits_{\mathfrak{a} = 1}^{\frac{R}{2} - 1} \sum\limits_{k_\mathfrak{a} = 1}^{\frac{R}{2} - 1 - \sum\limits_{\mathfrak{b} = 1}^{\mathfrak{a} - 1} k_\mathfrak{b}} C_{k_1} \mathfrak{f}( \mathfrak{a}) (-1)^{\sum\limits_{\mathfrak{m} = 1}^\mathfrak{a} k_\mathfrak{m}} \aInnerproduct{\dfrac{1}{2} ( \mu \varepsilon )^{-\sum\limits_{\mathfrak{m} = 1}^\mathfrak{a} k_\mathfrak{m}} \left( \nabla \times \nabla \times \right)^{\sum\limits_{\mathfrak{m} = 1}^\mathfrak{a} k_\mathfrak{m}} \left( H_h^{n + 1} + H_h^n \right)}{\nabla \times \widetilde{E}} = 0, \label{eqn:maxwell_E_lfR_full} 
     \end{multline} \\
      \begin{multline}
  \aInnerproduct{\mu \dfrac{H_h^{n + 1} - H_h^n}{\Delta t}}{\widetilde{H}} +  \aInnerproduct{\dfrac{1}{2} \nabla \times \left( E_h^{n + \frac{1}{2}} + E_h^{n - \frac{1}{2}} \right)}{\widetilde{H}} \\ - \sum\limits_{\mathfrak{a} = 1}^{\frac{R}{2} - 1} \sum\limits_{k_\mathfrak{a} = 1}^{\frac{R}{2} - 1 - \sum\limits_{\mathfrak{b} = 1}^{\mathfrak{a} - 1} k_\mathfrak{b}} C_{k_1} \mathfrak{f}(\mathfrak{a}) (-1)^{\sum\limits_{\mathfrak{m} = 1}^\mathfrak{a} k_\mathfrak{m}} \aInnerproduct{\dfrac{1}{2} (\varepsilon \mu)^{-\sum\limits_{\mathfrak{m} = 1}^\mathfrak{a} k_\mathfrak{m}} \nabla \times \left( E_h^{n + \frac{1}{2}} + E_h^{n - \frac{1}{2}} \right)}{\left( \nabla \times  \nabla \times \right)^{\sum\limits_{\mathfrak{m} = 1}^\mathfrak{a} k_\mathfrak{m}} \widetilde{H}} = 0, \label{eqn:maxwell_H_lfR_full}
       \end{multline}
\end{subequations}
for all $(\widetilde{p}, \widetilde{E}, \widetilde{H}) \in U_h \times V_h \times W_h$, and for $n \in \{1, \dotso, N - 1\}$. Using $n = 0$ for bootstrapping as in Equations~\labelcref{eqn:maxwell_p0_lfR,eqn:maxwell_E0_lfR,eqn:maxwell_H0_lfR} leads to the discrete problem: find $(p_h^{\frac{1}{2}}, E_h^{\frac{1}{2}}, H_h^1) \in U_h \times V_h \times W_h \subseteq \mathring{H}_{\varepsilon^{-1}}^1 \times \mathring{H}_{\varepsilon}(\curl; \Omega) \times \mathring{H}_{\mu}(\divgn; \Omega)$ such that:
\begin{subequations}
  \begin{multline}
    \aInnerproduct{\dfrac{p_h^{\frac{1}{2}} - p_h^0}{\Delta t/2}}{\widetilde{p}} - \dfrac{1}{2} \aInnerproduct{\dfrac{\varepsilon}{2} \left( E_h^{\frac{1}{2}} + E_h^0 \right)}{\nabla \widetilde{p}} \\ + \sum\limits_{\mathfrak{a} = 1}^{\frac{R}{2} - 1} \sum\limits_{k_\mathfrak{a} = 1}^{\frac{R}{2} - 1 - \sum\limits_{\mathfrak{b} = 1}^{\mathfrak{a} - 1} k_\mathfrak{b}} \dfrac{C_{k_1}}{2^{2 \sum\limits_{\mathfrak{m} = 1}^\mathfrak{a} k_\mathfrak{m} + 1}} \mathfrak{f}( \mathfrak{a}) \aInnerproduct{\dfrac{\varepsilon}{2} (\nabla \nabla \cdot)^{\sum\limits_{\mathfrak{m} = 1}^\mathfrak{a} k_\mathfrak{m}} \left( E_h^{\frac{1}{2}} + E_h^0 \right)}{\nabla \widetilde{p}}  = 0, \label{eqn:maxwell_p0_lfR_full}
    \end{multline} \\
    \begin{multline}
  \dfrac{1}{2}  \aInnerproduct{\dfrac{1}{2} \nabla \left(  p_h^{\frac{1}{2}} +  p_h^0 \right)}{\widetilde{E}} - \sum\limits_{\mathfrak{a} = 1}^{\frac{R}{2} - 1} \sum\limits_{k_\mathfrak{a} = 1}^{\frac{R}{2} - 1 - \sum\limits_{\mathfrak{b} = 1}^{\mathfrak{a} - 1} k_\mathfrak{b}} \dfrac{C_{k_1}}{2^{2 \sum\limits_{\mathfrak{m} = 1}^\mathfrak{a} k_\mathfrak{m} + 1}} \mathfrak{f}(\mathfrak{a}) \aInnerproduct{\dfrac{1}{2} \nabla \left(p_h^{\frac{1}{2}} + p_h^0 \right)}{\left( \nabla \nabla \cdot \right)^{\sum\limits_{\mathfrak{m} = 1}^\mathfrak{a} k_\mathfrak{m}} \widetilde{E}} \\ + \aInnerproduct{\varepsilon \dfrac{E_h^{\frac{1}{2}} - E_h^0}{\Delta t/2}}{\widetilde{E}}  - \aInnerproduct{\dfrac{1}{2} \left( H_h^1 + H_h^0 \right)}{\nabla \times \widetilde{E}} \\ + \sum\limits_{\mathfrak{a} = 1}^{\frac{R}{2} - 1}\sum\limits_{k_\mathfrak{a} = 1}^{\frac{R}{2} - 1 - \sum\limits_{\mathfrak{b} = 1}^{\mathfrak{a} - 1} k_\mathfrak{b}}  \dfrac{C_{k_1}}{2^{2 \sum\limits_{\mathfrak{m} = 1}^\mathfrak{a} k_\mathfrak{m}}} \mathfrak{f}( \mathfrak{a}) (-1)^{\sum\limits_{\mathfrak{m} = 1}^\mathfrak{a} k_\mathfrak{m}} \aInnerproduct{\dfrac{1}{2} ( \mu \varepsilon )^{-\sum\limits_{\mathfrak{m} = 1}^\mathfrak{a} k_\mathfrak{m}} \left( \nabla \times \nabla \times \right)^{\sum\limits_{\mathfrak{m} = 1}^\mathfrak{a} k_\mathfrak{m}} \left( H_h^{1} + H_h^0 \right)}{\nabla \times \widetilde{E}} = 0, \label{eqn:maxwell_E0_lfR_full} 
  \end{multline} \\
  \begin{multline}
    \aInnerproduct{\mu \dfrac{H_h^1 - H_h^0}{\Delta t}}{\widetilde{H}} + \dfrac{1}{2} \aInnerproduct{\dfrac{1}{2} \nabla \times \left(E_h^{\frac{1}{2}} + E_h^0 \right)}{\widetilde{H}} \\ - \sum\limits_{\mathfrak{a} = 1}^{\frac{R}{2} - 1} \sum\limits_{k_\mathfrak{a} = 1}^{\frac{R}{2} - 1 - \sum\limits_{\mathfrak{b} = 1}^{\mathfrak{a} - 1} k_\mathfrak{b}} \dfrac{C_{k_1}}{2^{2 \sum\limits_{\mathfrak{m} = 1}^\mathfrak{a} k_\mathfrak{m} + 1}} \mathfrak{f}(\mathfrak{a}) (-1)^{\sum\limits_{\mathfrak{m} = 1}^\mathfrak{a} k_\mathfrak{m}} \aInnerproduct{\dfrac{1}{2} (\varepsilon \mu)^{-\sum\limits_{\mathfrak{m} = 1}^\mathfrak{a} k_\mathfrak{m}} \nabla \times \left( E_h^{\frac{1}{2}} + E_h^0 \right)}{\left( \nabla \times  \nabla \times \right)^{\sum\limits_{\mathfrak{m} = 1}^\mathfrak{a} k_\mathfrak{m}} \widetilde{H}} = 0, \label{eqn:maxwell_H0_lfR_full}
    \end{multline}
for all $(\widetilde{p}, \widetilde{E}, \widetilde{H}) \in U_h \times V_h \times W_h$ given $(p_h^0, E_h^0, H_h^0) \in U_h \times V_h \times W_h$. 
Let $\Pi_h^0: \mathring{H}^1_{\varepsilon^{-1}}(\Omega) \longto U_h$, $\Pi_h^1: \mathring{H}_{\varepsilon}(\curl; \Omega) \longto V_h$ and $\Pi_h^2: \mathring{H}_{\mu}(\divgn; \Omega) \longto W_h$ denote the respective smoothed $L^2$ projection operators as in the sense of \cite{Schoberl2008,Christiansen2007} and a detailed discussion is available in ~\cite[Lemma 4.3.8]{Brenner2008} ~\cite[Theorem 5.3]{ArFaWi2006}~\cite[Lemma 11.9, Corollary 11.11, Theorem 16.10, Theorem 17.5]{Ern2021}. For the initial conditions, we set:

  \begin{equation}
p_h^0 \coloneq \Pi_h^0 p_0, \quad E_h^0 \coloneq \Pi_h^1 E_0, \quad \text{~and~} \quad H_h^0 \coloneq \Pi_h^2 H_0. \label{eqn:initial_lfR}
  \end{equation}
\end{subequations}
With this, we can now state our theorem for convergence of error in the full discretization of the system of Maxwell's equations using LF$_R$ and $r^{\text{th}}$ order polynomial Whitney basis finite elements for the various function spaces.

\begin{theorem}[Full Error Estimate]\label{thm:full_error_estmt_lfR}
Let $p \in C^{R+1}(0, T; \mathring{H}^1_{\varepsilon^{-1}}(\Omega))$, $E \in C^{R+1}(0, T; \mathring{H}_{\varepsilon}(\curl; \Omega))$, and $H \in C^{R+1}(0, T; \mathring{H}_{\mu}(\divgn; \Omega))$ be the solution to the variational formulation of the Maxwell's equations as in Equations~\labelcref{eqn:maxwell_p_wf,eqn:maxwell_E_wf,eqn:maxwell_H_wf}, and let $(p_h^{n + \frac{1}{2}}, E_h^{n + \frac{1}{2}}, H_h^{n + 1})$ be the solution of the fully discretized Maxwell's equations using the LF$_R$ scheme as in Equations~\labelcref{eqn:maxwell_p_lfR_full,eqn:maxwell_E_lfR_full,eqn:maxwell_H_lfR_full,eqn:maxwell_p0_lfR_full,eqn:maxwell_E0_lfR_full,eqn:maxwell_H0_lfR_full,eqn:initial_lfR}. If the time step $\Delta t > 0$ and the mesh parameter $h > 0$ are both sufficiently small, then there exists a positive bounded constant $C$ independent of both $\Delta t$ and $h$ such that the following error estimate holds:
\[
  \norm{e_{p_h}^{N - \frac{1}{2}}}_{\varepsilon^{-1}} + \norm{e_{E_h}^{N - \frac{1}{2}}}_{\varepsilon} + \norm{e_{H_h}^N}_{\mu} \le C \left[ (\Delta t)^R + h^r + h^r (\Delta t)^R \right],
\]
where the finite element subspaces $U_h$, $V_h$ and $W_h$ are each spanned by their respective Whitney form basis of polynomial order $r \ge 1$.
\end{theorem}

 \begin{proof}
 First, we shall subtract the set of equations for the full discretization as in Equations~\labelcref{eqn:maxwell_p_lfR_full,eqn:maxwell_E_lfR_full,eqn:maxwell_H_lfR_full} from Equations~\labelcref{eqn:remainder_p_lfR,eqn:remainder_E_lfR,eqn:remainder_H_lfR},  and then use the error terms in Equations~\labelcref{eqn:p_fullerror_lfR,eqn:E_fullerror_lfR,eqn:H_fullerror_lfR} and thereby obtain:
\begin{multline*}
  \aInnerproduct{\dfrac{e_{p_h}^{n + \frac{1}{2}} - e_{p_h}^{n - \frac{1}{2}}}{\Delta t}}{\widetilde{p}} - \aInnerproduct{\dfrac{\varepsilon}{2} \left( e_{E_h}^{n + \frac{1}{2}} + e_{E_h}^{n - \frac{1}{2}} \right)}{\nabla \widetilde{p}} \\ + \sum\limits_{\mathfrak{a} = 1}^{\frac{R}{2} - 1} \sum\limits_{k_\mathfrak{a} = 1}^{\frac{R}{2} - 1 - \sum\limits_{\mathfrak{b} = 1}^{\mathfrak{a} - 1} k_\mathfrak{b}} C_{k_1} \mathfrak{f}(\mathfrak{a}) \aInnerproduct{\dfrac{\varepsilon}{2} (\nabla \nabla \cdot)^{\sum\limits_{\mathfrak{m} = 1}^\mathfrak{a} k_\mathfrak{m}} \left( e_{E_h}^{n + \frac{1}{2}} + e_{E_h}^{n - \frac{1}{2}} \right)}{\nabla \widetilde{p}} \\ = \aInnerproduct{R_p^{n} + \varepsilon \nabla \cdot r_{R/2,E}^n - \sum\limits_{\mathfrak{a} = 1}^{\frac{R}{2} - 1} \sum\limits_{k_\mathfrak{a} = 1}^{\frac{R}{2} - 1 - \sum\limits_{\mathfrak{b} = 1}^{\mathfrak{a} - 1} k_\mathfrak{b}} C_{k_1} \mathfrak{f}(\mathfrak{a}) \varepsilon \nabla \cdot (\nabla \nabla \cdot)^{\sum\limits_{\mathfrak{m} = 1}^\mathfrak{a} k_\mathfrak{m}} r_{\left( R/2 - \sum \limits_{m=1}^a k_m \right),E}^n }{\widetilde{p}},
  \end{multline*}
  \begin{multline*}
   \aInnerproduct{\dfrac{1}{2} \nabla \left(e_{p_h}^{n + \frac{1}{2}} + e_{p_h}^{n - \frac{1}{2}} \right)}{\widetilde{E}} - \sum\limits_{\mathfrak{a} = 1}^{\frac{R}{2} - 1} \sum\limits_{k_\mathfrak{a} = 1}^{\frac{R}{2} - 1 - \sum\limits_{\mathfrak{b} = 1}^{\mathfrak{a} - 1} k_\mathfrak{b}} C_{k_1} \mathfrak{f}(\mathfrak{a}) \aInnerproduct{\dfrac{1}{2} \nabla \left(e_{p_h}^{n + \frac{1}{2}} + e_{p_h}^{n - \frac{1}{2}} \right)}{\left( \nabla \nabla \cdot \right)^{\sum\limits_{\mathfrak{m} = 1}^\mathfrak{a} k_\mathfrak{m}} \widetilde{E}} \\ + \aInnerproduct{\varepsilon \dfrac{e_{E_h}^{n + \frac{1}{2}} - e_{E_h}^{n - \frac{1}{2}}}{\Delta t}}{\widetilde{E}} -  \aInnerproduct{\dfrac{1}{2} \left( e_{H_h}^{n + 1} + e_{H_h}^n \right)}{\nabla \times \widetilde{E}} \\ + \sum\limits_{\mathfrak{a} = 1}^{\frac{R}{2} - 1} \sum\limits_{k_\mathfrak{a} = 1}^{\frac{R}{2} - 1 - \sum\limits_{\mathfrak{b} = 1}^{\mathfrak{a} - 1} k_\mathfrak{b}} C_{k_1} \mathfrak{f}( \mathfrak{a}) (-1)^{\sum\limits_{\mathfrak{m} = 1}^\mathfrak{a} k_\mathfrak{m}} \aInnerproduct{\dfrac{1}{2} ( \mu \varepsilon )^{-\sum\limits_{\mathfrak{m} = 1}^\mathfrak{a} k_\mathfrak{m}} \left( \nabla \times \nabla \times \right)^{\sum\limits_{\mathfrak{m} = 1}^\mathfrak{a} k_\mathfrak{m}} \left( e_{H_h}^{n + 1} + e_{H_h}^n \right)}{\nabla \times \widetilde{E}} \\ =  \aInnerproduct{\varepsilon R_E^n + \nabla r_{R/2,p}^n - \nabla \times r_{R/2,H}^{n+\frac{1}{2}} - \sum\limits_{\mathfrak{a} = 1}^{\frac{R}{2} - 1} \sum\limits_{k_\mathfrak{a} = 1}^{\frac{R}{2} - 1 - \sum\limits_{\mathfrak{b} = 1}^{\mathfrak{a} - 1} k_\mathfrak{b}} C_{k_1} \mathfrak{f}(\mathfrak{a}) \left(\varepsilon \left( \nabla \nabla \cdot \right)^{\sum\limits_{\mathfrak{m} = 1}^\mathfrak{a} k_\mathfrak{m}} \nabla r_{\left( R/2 - \sum \limits_{m=1}^a k_m \right),p}^n \right. \\ - \left. (-1)^{\sum\limits_{\mathfrak{m} = 1}^\mathfrak{a} k_\mathfrak{m}} (\mu \varepsilon)^{-\sum\limits_{\mathfrak{m} = 1}^\mathfrak{a} k_\mathfrak{m}} \nabla \times (\nabla \times \nabla \times)^{\sum\limits_{\mathfrak{m} = 1}^\mathfrak{a} k_\mathfrak{m}} r_{\left( R/2 - \sum \limits_{m=1}^a k_m \right),H}^{n+\frac{1}{2}}\right)}{\widetilde{E}},
     \end{multline*}
      \begin{multline*}
  \aInnerproduct{\mu \dfrac{e_{H_h}^{n + 1} - e_{H_h}^n}{\Delta t}}{\widetilde{H}} +  \aInnerproduct{\dfrac{1}{2} \nabla \times \left( e_{E_h}^{n + \frac{1}{2}} + e_{E_h}^{n - \frac{1}{2}} \right)}{\widetilde{H}} \\ - \sum\limits_{\mathfrak{a} = 1}^{\frac{R}{2} - 1} \sum\limits_{k_\mathfrak{a} = 1}^{\frac{R}{2} - 1 - \sum\limits_{\mathfrak{b} = 1}^{\mathfrak{a} - 1} k_\mathfrak{b}} C_{k_1} \mathfrak{f}(\mathfrak{a}) (-1)^{\sum\limits_{\mathfrak{m} = 1}^\mathfrak{a} k_\mathfrak{m}} \aInnerproduct{\dfrac{1}{2} (\varepsilon \mu)^{-\sum\limits_{\mathfrak{m} = 1}^\mathfrak{a} k_\mathfrak{m}} \nabla \times \left( e_{E_h}^{n + \frac{1}{2}} + e_{E_h}^{n - \frac{1}{2}} \right)}{\left( \nabla \times  \nabla \times \right)^{\sum\limits_{\mathfrak{m} = 1}^\mathfrak{a} k_\mathfrak{m}} \widetilde{H}} \\ = \aInnerproduct{\mu R_H^{n + \frac{1}{2}} + \nabla \times r_{R/2,E}^n \\ + \sum\limits_{\mathfrak{a} = 1}^{\frac{R}{2} - 1} \sum\limits_{k_\mathfrak{a} = 1}^{\frac{R}{2} - 1 - \sum\limits_{\mathfrak{b} = 1}^{\mathfrak{a} - 1} k_\mathfrak{b}} C_{k_1} \mathfrak{f}(\mathfrak{a}) (-1)^{\sum\limits_{\mathfrak{m} = 1}^\mathfrak{a} k_\mathfrak{m}} (\varepsilon \mu)^{-\sum\limits_{\mathfrak{m} = 1}^\mathfrak{a} k_\mathfrak{m}} (\nabla \times \nabla \times)^{\sum\limits_{\mathfrak{m} = 1}^\mathfrak{a} k_\mathfrak{m}} \nabla \times r_{\left( R/2 - \sum \limits_{m=1}^a k_m \right),E}^n}{\widetilde{H}}.
       \end{multline*}
 Next, using the values of the error terms $e_{p_h}^n$, $e_{E_h}^n$ and $e_{H_h}^n$ as in Equations~\labelcref{eqn:p_fullerror_lfR,eqn:E_fullerror_lfR,eqn:H_fullerror_lfR} in the above equations, we get:
 \begin{multline*}
  \aInnerproduct{\dfrac{\left(\eta^{n + \frac{1}{2}} - \eta^{n - \frac{1}{2}}\right) - \left( \eta^{n + \frac{1}{2}}_h - \eta^{n - \frac{1}{2}}_h \right)}{\Delta t}}{\widetilde{p}} - \aInnerproduct{\dfrac{\varepsilon}{2} \left( \left( \zeta^{n + \frac{1}{2}} + \zeta^{n - \frac{1}{2}} \right) - \left(\zeta_h^{n + \frac{1}{2}} + \zeta^{n - \frac{1}{2}}_h \right) \right)}{\nabla \widetilde{p}} \\ + \sum\limits_{\mathfrak{a} = 1}^{\frac{R}{2} - 1} \sum\limits_{k_\mathfrak{a} = 1}^{\frac{R}{2} - 1 - \sum\limits_{\mathfrak{b} = 1}^{\mathfrak{a} - 1} k_\mathfrak{b}} C_{k_1} \mathfrak{f}(\mathfrak{a}) \aInnerproduct{\dfrac{\varepsilon}{2} (\nabla \nabla \cdot)^{\sum\limits_{\mathfrak{m} = 1}^\mathfrak{a} k_\mathfrak{m}} \left( \left( \zeta^{n + \frac{1}{2}} + \zeta^{n - \frac{1}{2}} \right) - \left(\zeta_h^{n + \frac{1}{2}} + \zeta^{n - \frac{1}{2}}_h \right) \right)}{\nabla \widetilde{p}} \\ = \aInnerproduct{R_p^{n} + \varepsilon \nabla \cdot r_{R/2,E}^n - \sum\limits_{\mathfrak{a} = 1}^{\frac{R}{2} - 1} \sum\limits_{k_\mathfrak{a} = 1}^{\frac{R}{2} - 1 - \sum\limits_{\mathfrak{b} = 1}^{\mathfrak{a} - 1} k_\mathfrak{b}} C_{k_1} \mathfrak{f}(\mathfrak{a}) \varepsilon \nabla \cdot (\nabla \nabla \cdot)^{\sum\limits_{\mathfrak{m} = 1}^\mathfrak{a} k_\mathfrak{m}} r_{\left( R/2 - \sum \limits_{m=1}^a k_m \right),E}^n }{\widetilde{p}},
  \end{multline*}
  \begin{multline*}
   \aInnerproduct{\dfrac{1}{2} \nabla \left( \left( \eta^{n + \frac{1}{2}} + \eta^{n - \frac{1}{2}} \right) - \left( \eta^{n + \frac{1}{2}}_h + \eta^{n - \frac{1}{2}}_h \right) \right)}{\widetilde{E}} \\ - \sum\limits_{\mathfrak{a} = 1}^{\frac{R}{2} - 1} \sum\limits_{k_\mathfrak{a} = 1}^{\frac{R}{2} - 1 - \sum\limits_{\mathfrak{b} = 1}^{\mathfrak{a} - 1} k_\mathfrak{b}} C_{k_1} \mathfrak{f}(\mathfrak{a}) \aInnerproduct{\dfrac{1}{2} \nabla \left( \left( \eta^{n + \frac{1}{2}} + \eta^{n - \frac{1}{2}} \right) - \left( \eta^{n + \frac{1}{2}}_h + \eta^{n - \frac{1}{2}}_h \right) \right)}{\left( \nabla \nabla \cdot \right)^{\sum\limits_{\mathfrak{m} = 1}^\mathfrak{a} k_\mathfrak{m}} \widetilde{E}} \\ + \aInnerproduct{\varepsilon \dfrac{\left( \zeta^{n + \frac{1}{2}} - \zeta^{n - \frac{1}{2}} \right) - \left(\zeta^{n + \frac{1}{2}}_h - \zeta^{n - \frac{1}{2}}_h \right)}{\Delta t}}{\widetilde{E}} -  \aInnerproduct{\dfrac{1}{2} \left( \left( \xi^{n + 1} + \xi^{n} \right) - \left(\xi^{n + 1}_h - \xi^{n}_h \right) \right)}{\nabla \times \widetilde{E}} \\ + \sum\limits_{\mathfrak{a} = 1}^{\frac{R}{2} - 1} \sum\limits_{k_\mathfrak{a} = 1}^{\frac{R}{2} - 1 - \sum\limits_{\mathfrak{b} = 1}^{\mathfrak{a} - 1} k_\mathfrak{b}} C_{k_1} \mathfrak{f}( \mathfrak{a}) (-1)^{\sum\limits_{\mathfrak{m} = 1}^\mathfrak{a} k_\mathfrak{m}} \aInnerproduct{\dfrac{1}{2} ( \mu \varepsilon )^{-\sum\limits_{\mathfrak{m} = 1}^\mathfrak{a} k_\mathfrak{m}} \left( \nabla \times \nabla \times \right)^{\sum\limits_{\mathfrak{m} = 1}^\mathfrak{a} k_\mathfrak{m}} \left( \left( \xi^{n + 1} + \xi^{n} \right) - \left(\xi^{n + 1}_h - \xi^{n}_h \right) \right)}{\nabla \times \widetilde{E}} \\ =  \aInnerproduct{\varepsilon R_E^n + \nabla r_{R/2,p}^n - \nabla \times r_{R/2,H}^{n+\frac{1}{2}} - \sum\limits_{\mathfrak{a} = 1}^{\frac{R}{2} - 1} \sum\limits_{k_\mathfrak{a} = 1}^{\frac{R}{2} - 1 - \sum\limits_{\mathfrak{b} = 1}^{\mathfrak{a} - 1} k_\mathfrak{b}} C_{k_1} \mathfrak{f}(\mathfrak{a}) \left(\varepsilon \left( \nabla \nabla \cdot \right)^{\sum\limits_{\mathfrak{m} = 1}^\mathfrak{a} k_\mathfrak{m}} \nabla r_{\left( R/2 - \sum \limits_{m=1}^a k_m \right),p}^n \right. \\ - \left. (-1)^{\sum\limits_{\mathfrak{m} = 1}^\mathfrak{a} k_\mathfrak{m}} (\mu \varepsilon)^{-\sum\limits_{\mathfrak{m} = 1}^\mathfrak{a} k_\mathfrak{m}} \nabla \times (\nabla \times \nabla \times)^{\sum\limits_{\mathfrak{m} = 1}^\mathfrak{a} k_\mathfrak{m}} r_{\left( R/2 - \sum \limits_{m=1}^a k_m \right),H}^{n+\frac{1}{2}}\right)}{\widetilde{E}},
  \end{multline*}
  \begin{multline*}
  \aInnerproduct{\mu \dfrac{\left(\xi^{n + 1} - \xi^{n} \right) - \left( \xi^{n + 1}_h - \xi^{n}_h \right)}{\Delta t}}{\widetilde{H}} +  \aInnerproduct{\dfrac{1}{2} \nabla \times \left( \left( \zeta^{n + \frac{1}{2}} + \zeta^{n - \frac{1}{2}} \right) - \left( \zeta^{n + \frac{1}{2}}_h + \zeta^{n - \frac{1}{2}}_h \right) \right)}{\widetilde{H}} \\ - \sum\limits_{\mathfrak{a} = 1}^{\frac{R}{2} - 1} \sum\limits_{k_\mathfrak{a} = 1}^{\frac{R}{2} - 1 - \sum\limits_{\mathfrak{b} = 1}^{\mathfrak{a} - 1} k_\mathfrak{b}} C_{k_1} \mathfrak{f}(\mathfrak{a}) (-1)^{\sum\limits_{\mathfrak{m} = 1}^\mathfrak{a} k_\mathfrak{m}} \\ \left(\aInnerproduct{\dfrac{1}{2} (\varepsilon \mu)^{-\sum\limits_{\mathfrak{m} = 1}^\mathfrak{a} k_\mathfrak{m}} \nabla \times \left( \left( \zeta^{n + \frac{1}{2}} + \zeta^{n - \frac{1}{2}} \right) - \left( \zeta^{n + \frac{1}{2}}_h + \zeta^{n - \frac{1}{2}}_h \right) \right)}{\left( \nabla \times  \nabla \times \right)^{\sum\limits_{\mathfrak{m} = 1}^\mathfrak{a} k_\mathfrak{m}} \widetilde{H}}\right) = \aInnerproduct{\mu R_H^{n + \frac{1}{2}} + \nabla \times r_{R/2,E}^n \\ +  \sum\limits_{\mathfrak{a} = 1}^{\frac{R}{2} - 1} \sum\limits_{k_\mathfrak{a} = 1}^{\frac{R}{2} - 1 - \sum\limits_{\mathfrak{b} = 1}^{\mathfrak{a} - 1} k_\mathfrak{b}} C_{k_1} \mathfrak{f}(\mathfrak{a}) (-1)^{\sum\limits_{\mathfrak{m} = 1}^\mathfrak{a} k_\mathfrak{m}} (\varepsilon \mu)^{-\sum\limits_{\mathfrak{m} = 1}^\mathfrak{a} k_\mathfrak{m}} (\nabla \times \nabla \times)^{\sum\limits_{\mathfrak{m} = 1}^\mathfrak{a} k_\mathfrak{m}} \nabla \times r_{\left( R/2 - \sum \limits_{m=1}^a k_m \right),E}^n}{\widetilde{H}}.
\end{multline*}
Since these equations are true for all $(\widetilde{p}, \widetilde{E}, \widetilde{H}) \in U_h \times V_h \times W_h$, we choose $\widetilde{p} = -2 \Delta t \varepsilon^{-1} \left( \eta_h^{n + \frac{1}{2}} + \eta_h^{n - \frac{1}{2}}\right)$, $\widetilde{E} = -2 \Delta t \left( \zeta_h^{n + \frac{1}{2}} + \zeta_h^{n - \frac{1}{2}} \right)$ and $\widetilde{H} = -2 \Delta t \left( \xi_h^n + \xi_h^{n - 1} \right)$ and using the fact that $\nabla U_h \subseteq V_h$ and $\nabla \times V_h \subseteq W_h$ and so on, we obtain that:
\begin{multline}
2 \ainnerproduct{\varepsilon^{-1} \left( \eta^{n + \frac{1}{2}}_h - \eta^{n - \frac{1}{2}}_h \right)}{\eta^{n + \frac{1}{2}}_h + \eta^{n - \frac{1}{2}}_h} + 2 \ainnerproduct{\varepsilon \left( \zeta^{n + \frac{1}{2}}_h - \zeta^{n - \frac{1}{2}}_h \right)}{\zeta^{n + \frac{1}{2}}_h + \zeta^{n - \frac{1}{2}}_h} \, + 2 \ainnerproduct{\mu \left( \xi^n_h - \xi^{n - 1}_h \right)}{\xi^n_h + \xi^{n - 1}_h} \\ 
= 2 \ainnerproduct{\varepsilon^{-1} \left( \eta^{n + \frac{1}{2}} - \eta^{n -  \frac{1}{2}} \right)}{\eta^{n + \frac{1}{2}}_h + \eta^{n - \frac{1}{2}}_h} \, + 2 \ainnerproduct{\varepsilon \left( \zeta^{n + \frac{1}{2}} - \zeta^{n - \frac{1}{2}} \right)}{\zeta^{n + \frac{1}{2}}_h + \zeta^{n - \frac{1}{2}}_h} + 2 \ainnerproduct{\mu \left( \xi^n - \xi^{n - 1} \right)}{\xi^n_h + \xi^{n - 1}_h} \\ 
+ 2 \Delta t \ainnerproduct{- R_p^n - \varepsilon \nabla \cdot r_{R/2,E}^n + \sum\limits_{\mathfrak{a} = 1}^{\frac{R}{2} - 1} \sum\limits_{k_\mathfrak{a} = 1}^{\frac{R}{2} - 1 - \sum\limits_{\mathfrak{b} = 1}^{\mathfrak{a} - 1} k_\mathfrak{b}} C_{k_1} \mathfrak{f}(\mathfrak{a}) \varepsilon \nabla \cdot (\nabla \nabla \cdot)^{\sum\limits_{\mathfrak{m} = 1}^\mathfrak{a} k_\mathfrak{m}} r_{\left( R/2 - \sum \limits_{m=1}^a k_m \right),E}^n}{\varepsilon^{-1} \left( \eta^{n + \frac{1}{2}}_h + \eta^{n - \frac{1}{2}}_h \right)} \\
 + 2 \Delta t \ainnerproduct{-\varepsilon R_E^n - \nabla r_{R/2,p}^n + \nabla \times r_{R/2,H}^{n+\frac{1}{2}} + \sum\limits_{\mathfrak{a} = 1}^{\frac{R}{2} - 1} \sum\limits_{k_\mathfrak{a} = 1}^{\frac{R}{2} - 1 - \sum\limits_{\mathfrak{b} = 1}^{\mathfrak{a} - 1} k_\mathfrak{b}} C_{k_1} \mathfrak{f}(\mathfrak{a}) \left(\varepsilon \left( \nabla \nabla \cdot \right)^{\sum\limits_{\mathfrak{m} = 1}^\mathfrak{a} k_\mathfrak{m}} \nabla r_{\left( R/2 - \sum \limits_{m=1}^a k_m \right),p}^n \right. \\ - \left. (-1)^{\sum\limits_{\mathfrak{m} = 1}^\mathfrak{a} k_\mathfrak{m}} (\mu \varepsilon)^{-\sum\limits_{\mathfrak{m} = 1}^\mathfrak{a} k_\mathfrak{m}} \nabla \times (\nabla \times \nabla \times)^{\sum\limits_{\mathfrak{m} = 1}^\mathfrak{a} k_\mathfrak{m}} r_{\left( R/2 - \sum \limits_{m=1}^a k_m \right),H}^{n+\frac{1}{2}}\right)}{\zeta^{n + \frac{1}{2}}_h + \zeta^{n - \frac{1}{2}}_h}   \\ 
 + 2 \Delta t \ainnerproduct{-\mu R_H^{n + \frac{1}{2}} - \nabla \times r_{R/2,E}^n \\ - \sum\limits_{\mathfrak{a} = 1}^{\frac{R}{2} - 1} \sum\limits_{k_\mathfrak{a} = 1}^{\frac{R}{2} - 1 - \sum\limits_{\mathfrak{b} = 1}^{\mathfrak{a} - 1} k_\mathfrak{b}} C_{k_1} \mathfrak{f}(\mathfrak{a}) (-1)^{\sum\limits_{\mathfrak{m} = 1}^\mathfrak{a} k_\mathfrak{m}} (\varepsilon \mu)^{-\sum\limits_{\mathfrak{m} = 1}^\mathfrak{a} k_\mathfrak{m}} (\nabla \times \nabla \times)^{\sum\limits_{\mathfrak{m} = 1}^\mathfrak{a} k_\mathfrak{m}} \nabla \times r_{\left( R/2 - \sum \limits_{m=1}^a k_m \right),E}^n}{\xi^n_h + \xi^{n - 1}_h}. \label{eqn:suberror_p+E+H_lf4}
\end{multline}
Consider that  $\varepsilon^{-1} \left(\eta^{n + \frac{1}{2}} - \eta^{n - \frac{1}{2}}\right) = \varepsilon^{-1} \left(I - \Pi_h^0\right) \left( p(t^{n + \frac{1}{2}}) - p(t^{n - \frac{1}{2}})\right)$ by Equation~\eqref{eqn:p_fullerror_sub_lfR}. Using the Taylor theorem with remainder as in Theorem~\ref{thm:dscrt_error_estmt_lfR}, applying the Cauchy-Schwarz, AM-GM, and triangle inequalities, we have the following inequality:
\begin{align*}
2 \ainnerproduct{\varepsilon^{-1} \left( \eta^{n + \frac{1}{2}} - \eta^{n - \frac{1}{2}} \right)}{\eta^{n + \frac{1}{2}}_h + \eta^{n - \frac{1}{2}}_h} & = 2 \Delta t \ainnerproduct{\varepsilon^{-1} \left( I - \Pi_h^0 \right) \sum \limits_{k=0}^{\frac{R}{2} - 1} \dfrac{1}{(2k+1)!} \dfrac{\Delta t^{2k}}{2^{2k}} \dfrac{\partial^{2k+1} p}{\partial t^{2k+1}}(t^n)}{\eta^{n + \frac{1}{2}}_h + \eta^{n - \frac{1}{2}}_h} \\ & + 2 \Delta t \ainnerproduct{\varepsilon^{-1} \left( I - \Pi_h^0 \right) R_p^n}{\eta^{n + \frac{1}{2}}_h + \eta^{n - \frac{1}{2}}_h} \\ 
 & \le \Delta t \bigg[\sum \limits_{k=0}^{\frac{R}{2} - 1} \dfrac{1}{(2k+1)!} \dfrac{\Delta t^{2k}}{2^{2k}} \norm[\bigg]{(I - \Pi_h^0) \dfrac{\partial^{2k+1} p}{\partial t^{2k+1}}(t^n)}^2_{\varepsilon^{-1}} \!\! \\ & +  \norm[\bigg]{(I - \Pi_h^0) R_p^n}^2_{\varepsilon^{-1}} \bigg] \! + (R+2) \Delta t \bigg[ \norm{\eta_h^{n + \frac{1}{2}}}^2_{\varepsilon^{-1}} + \norm{\eta_h^{n - \frac{1}{2}}}^2_{\varepsilon^{-1}} \bigg].
\end{align*}
Similarly, using Equations~\labelcref{eqn:E_fullerror_sub_lfR,eqn:H_fullerror_sub_lfR} for the error terms for $E$ and $H$, we obtain that:
\begin{gather*}
\begin{split}
2  \ainnerproduct{\varepsilon \left( \zeta^{n + \frac{1}{2}} - \zeta^{n - \frac{1}{2}} \right)}{\zeta^{n + \frac{1}{2}}_h + \zeta^{n - \frac{1}{2}}_h} \le \Delta t \bigg[\sum \limits_{k=0}^{\frac{R}{2} - 1} \dfrac{1}{(2k+1)!} \dfrac{\Delta t^{2k}}{2^{2k}} \norm[\bigg]{(I - \Pi_h^1) \dfrac{\partial^{2k+1} E}{\partial t^{2k+1}}(t^n)}^2_{\varepsilon} \!\! \\  +  \norm[\bigg]{(I - \Pi_h^1) R_E^n}^2_{\varepsilon} \bigg] \! + (R+2) \Delta t \bigg[ \norm{\zeta_h^{n + \frac{1}{2}}}^2_{\varepsilon} + \norm{\zeta_h^{n - \frac{1}{2}}}^2_{\varepsilon} \bigg],
\end{split} \\ 
\begin{split}
2 \ainnerproduct{\mu \left( \xi^{n + 1} - \xi^n \right)}{\xi^{n + 1}_h + \xi^n_h} \le \Delta t \bigg[\sum \limits_{k=0}^{\frac{R}{2} - 1} \dfrac{1}{(2k+1)!} \dfrac{\Delta t^{2k}}{2^{2k}} \norm[\bigg]{(I - \Pi_h^2) \dfrac{\partial^{2k+1} H}{\partial t^{2k+1}}(t^{n + \frac{1}{2}})}^2_{\mu} \!\! \\  +  \norm[\bigg]{(I - \Pi_h^2) R_H^{n + \frac{1}{2}}}^2_{\mu} \bigg] \! + (R+2) \Delta t \bigg[ \norm{\xi_h^{n + 1}}^2_{\mu} + \norm{\xi_h^n}^2_{\mu} \bigg].
\end{split}
\end{gather*}
Using these inequalities for $\eta$, $\zeta$ and $\xi$ in Equation~\eqref{eqn:suberror_p+E+H_lf4}, and summing from $n = 1$ to $N-1$, we get that:
\begin{multline*}
  \norm{\eta_h^{N - \frac{1}{2}}}^2_{\varepsilon^{-1}} + \norm{\zeta_h^{N - \frac{1}{2}}}^2_{\varepsilon} + \norm{\xi_h^N}^2_{\mu} \le (R+4) \Delta t \sum\limits_{n = 0}^{N-1} \bigg[ \norm{\eta_h^{n + \frac{1}{2}}}^2_{\varepsilon^{-1}} +\norm{\zeta_h^{n + \frac{1}{2}}}^2_{\varepsilon} + \norm{\xi_h^n}^2_{\mu} \bigg] \\+ \Delta t \sum\limits_{n = 0}^{N} \bigg[\sum \limits_{k=0}^{\frac{R}{2} - 1} \left(\dfrac{1}{(2k+1)!}\right)^2 \dfrac{\Delta t^{4k}}{2^{4k}} \left(\norm[\bigg]{(I - \Pi_h^0) \dfrac{\partial^{2k+1} p}{\partial t^{2k+1}}(t^n)}^2_{\varepsilon^{-1}} + \norm[\bigg]{(I - \Pi_h^1) \dfrac{\partial^{2k+1} E}{\partial t^{2k+1}}(t^{n - \frac{1}{2}})}^2_{\varepsilon} \right. \\ \left. + \norm[\bigg]{(I - \Pi_h^2) \dfrac{\partial^{2k+1} H}{\partial t^{2k+1}}(t^{n + \frac{1}{2}})}^2_\mu\right) + 
  \norm[\bigg]{(I - \Pi_h^0) R_p^n}^2_{\varepsilon^{-1}} + \norm[\bigg]{(I - \Pi_h^1) R_E^n}^2_{\varepsilon} + \norm[\bigg]{(I - \Pi_h^2) R_H^n}^2_\mu  \\ + \norm{R_p^n}^2_{\varepsilon^{-1}} + \norm{R_E^n}^2_{\varepsilon} + \norm{R_H^{n + \frac{1}{2}}}^2_{\mu} \norm{\nabla r_{R/2,p}^n} ^2_{\varepsilon^{-1}} + \norm{\nabla \cdot r_{R/2,E}^n}^2_{\varepsilon} + \varepsilon^{-1} \mu^{-1} \norm{\nabla \times r_{R/2,E}^n}^2_{\varepsilon} +  \varepsilon^{-1} \mu^{-1} \norm{\nabla \times r_{R/2,H}^{n + \frac{1}{2}}}^2_{\mu}  \\ 
 +\sum\limits_{\mathfrak{a} = 1}^{\frac{R}{2} - 1} \sum\limits_{k_\mathfrak{a} = 1}^{\frac{R}{2} - 1 - \sum\limits_{\mathfrak{b} = 1}^{\mathfrak{a} - 1} k_\mathfrak{b}} C_{k_1}^2 \mathfrak{f}(\mathfrak{a})^2 \left(\norm{\left( \nabla \nabla \cdot \right)^{\sum\limits_{\mathfrak{m} = 1}^\mathfrak{a} k_\mathfrak{m}} \nabla r_{\left( R/2 - \sum \limits_{m=1}^a k_m \right),p}^n}^2_{\varepsilon^{-1}} + \norm{\nabla \cdot (\nabla \nabla \cdot)^{\sum\limits_{\mathfrak{m} = 1}^\mathfrak{a} k_\mathfrak{m}} r_{\left( R/2 - \sum \limits_{m=1}^a k_m \right),E}^n }^2_\varepsilon  \right. \\ \left. + (\varepsilon \mu)^{-(2 \sum\limits_{\mathfrak{m} = 1}^\mathfrak{a} k_\mathfrak{m} + 1)} \norm{(\nabla \times \nabla \times)^{\sum\limits_{\mathfrak{m} = 1}^\mathfrak{a} k_\mathfrak{m}} \nabla \times r_{\left( R/2 - \sum \limits_{m=1}^a k_m \right),E}^n}^2_\varepsilon \right. \\ \left. + (\mu \varepsilon)^{-(2 \sum\limits_{\mathfrak{m} = 1}^\mathfrak{a} k_\mathfrak{m} + 1)} \norm{ \nabla \times (\nabla \times \nabla \times)^{\sum\limits_{\mathfrak{m} = 1}^\mathfrak{a} k_\mathfrak{m}} r_{\left( R/2 - \sum \limits_{m=1}^a k_m \right),H}^{n+\frac{1}{2}}}^2_\mu \right) \bigg] + \left[ \norm{\eta_h^\frac{1}{2}}^2_{\varepsilon^{-1}} + \norm{\zeta_h^\frac{1}{2}}^2_{\varepsilon} + \norm{\xi_h^1}^2_{\mu} \right].
\end{multline*}
Likewise, for the semidiscrete approximation of the initial system as in Equations~\labelcref{eqn:maxwell_p0_lfR,eqn:maxwell_E0_lfR,eqn:maxwell_H0_lfR}, following the same steps as for general one, we obtain for their errors the following inequality:
 \begin{multline*}
   \norm{\eta_h^{\frac{1}{2}}}^2_{\varepsilon^{-1}} + \norm{\zeta_h^{\frac{1}{2}}}^2_{\varepsilon} + \norm{\xi_h^1}^2_{\mu} \le \dfrac{(R+4)}{2} \Delta t \bigg[ \norm{\eta_h^{\frac{1}{2}}}^2_{\varepsilon^{-1}} + \norm{\eta_h^0}^2_{\varepsilon^{-1}} +\norm{\zeta_h^{ \frac{1}{2}}}^2_{\varepsilon} + \norm{\zeta_h^0}^2_{\varepsilon} + \norm{\xi_h^1}^2_{\mu} + \norm{\xi_h^0}^2_{\mu}\bigg] \\+ \Delta t \left[ \norm{R_p^0}^2_{\varepsilon^{-1}} + \norm{R_E^0}^2_{\varepsilon} + \norm{R_H^{\frac{1}{2}}}^2_{\mu} + \norm{\nabla r_{R/2,p}^0} ^2_{\varepsilon^{-1}} + \norm{\nabla \cdot r_{R/2,E}^0}^2_{\varepsilon} + \varepsilon^{-1} \mu^{-1} \norm{\nabla \times r_{R/2,E}^0}^2_{\varepsilon}  \right. \\ \left. +  \varepsilon^{-1} \mu^{-1} \norm{\nabla \times r_{R/2,H}^{\frac{1}{2}}}^2_{\mu} +\sum\limits_{\mathfrak{a} = 1}^{\frac{R}{2} - 1} \sum\limits_{k_\mathfrak{a} = 1}^{\frac{R}{2} - 1 - \sum\limits_{\mathfrak{b} = 1}^{\mathfrak{a} - 1} k_\mathfrak{b}} \dfrac{C_{k_1}^2}{4^{2 \sum\limits_{\mathfrak{m} = 1}^\mathfrak{a} k_\mathfrak{m} + 1}} \mathfrak{f}(\mathfrak{a})^2 \left(\norm{\left( \nabla \nabla \cdot \right)^{\sum\limits_{\mathfrak{m} = 1}^\mathfrak{a} k_\mathfrak{m}} \nabla r_{\left( R/2 - \sum \limits_{m=1}^a k_m \right),p}^0}^2_{\varepsilon^{-1}} \right. \right. \\ \left. \left. + \norm{\nabla \cdot (\nabla \nabla \cdot)^{\sum\limits_{\mathfrak{m} = 1}^\mathfrak{a} k_\mathfrak{m}} r_{\left( R/2 - \sum \limits_{m=1}^a k_m \right),E}^0}^2_\varepsilon + (\varepsilon \mu)^{-(2 \sum\limits_{\mathfrak{m} = 1}^\mathfrak{a} k_\mathfrak{m} + 1)} \norm{(\nabla \times \nabla \times)^{\sum\limits_{\mathfrak{m} = 1}^\mathfrak{a} k_\mathfrak{m}} \nabla \times r_{\left( R/2 - \sum \limits_{m=1}^a k_m \right),E}^0}^2_\varepsilon \right. \right. \\ \left. \left.+ (\mu \varepsilon)^{-(2 \sum\limits_{\mathfrak{m} = 1}^\mathfrak{a} k_\mathfrak{m} + 1)} \norm{ \nabla \times (\nabla \times \nabla \times)^{\sum\limits_{\mathfrak{m} = 1}^\mathfrak{a} k_\mathfrak{m}} r_{\left( R/2 - \sum \limits_{m=1}^a k_m \right),H}^{\frac{1}{2}}}^2_\mu \right) \right] + \left[\norm{\eta_h^0}^2_{\varepsilon^{-1}} + \norm{\zeta_h^0}^2_{\varepsilon} + \norm{\xi_h^0}^2_{\mu}  \right].
  \end{multline*} 
Adding the last two inequalities, we arrive at the following resulting inequality:
\begin{multline*}
  \norm{\eta_h^{N - \frac{1}{2}}}^2_{\varepsilon^{-1}} + \norm{\zeta_h^{N - \frac{1}{2}}}^2_{\varepsilon} + \norm{\xi_h^N}^2_{\mu} \le \\
  \left(\dfrac{1 + \dfrac{R+4}{2} \Delta t}{1 - \dfrac{R+4}{2} \Delta t} \right) \Bigg[ \norm{\eta_h^0}^2_{\varepsilon^{-1}} + \norm{\zeta_h^0}^2_{\varepsilon} + \norm{\xi_h^0}^2_{\mu} \Bigg] + \left(\dfrac{(R+4) \Delta t}{1 - \dfrac{R+4}{2} \Delta t} \right) \sum\limits_{n = 0}^{N - 1} \Bigg[ \norm{\eta_h^{n + \frac{1}{2}}}^2_{\varepsilon^{-1}} + \norm{\zeta_h^{n + \frac{1}{2}}}^2_{\varepsilon} +  \norm{\xi_h^{n + 1}}^2_{\mu} \Bigg] + \\
  \left( \dfrac{ \Delta t}{1 - \dfrac{R+4}{2} \Delta t} \right) \sum\limits_{n = 0}^{N - 1} \Bigg[\sum \limits_{k=0}^{\frac{R}{2} - 1} \left(\dfrac{1}{(2k+1)!}\right)^2 \dfrac{\Delta t^{4k}}{2^{4k}} \left(\norm[\bigg]{(I - \Pi_h^0) \dfrac{\partial^{2k+1} p}{\partial t^{2k+1}}(t^n)}^2_{\varepsilon^{-1}} + \norm[\bigg]{(I - \Pi_h^1) \dfrac{\partial^{2k+1} E}{\partial t^{2k+1}}(t^n)}^2_{\varepsilon} \right. \\ + \left. \norm[\bigg]{(I - \Pi_h^2) \dfrac{\partial^{2k+1} H}{\partial t^{2k+1}}(t^{n + \frac{1}{2}})}^2_\mu\right) + 
  \norm[\bigg]{(I - \Pi_h^0) R_p^n}^2_{\varepsilon^{-1}} + \norm[\bigg]{(I - \Pi_h^1) R_E^n}^2_{\varepsilon} + \norm[\bigg]{(I - \Pi_h^2) R_H^{n + \frac{1}{2}}}^2_\mu  + \norm{R_p^n}^2_{\varepsilon^{-1}} + \norm{R_E^n}^2_{\varepsilon} \\ + \norm{R_H^{n + \frac{1}{2}}}^2_{\mu} + \norm{\nabla r_{R/2,p}^n} ^2_{\varepsilon^{-1}} + \norm{\nabla \cdot r_{R/2,E}^n}^2_{\varepsilon} + \varepsilon^{-1} \mu^{-1} \norm{\nabla \times r_{R/2,E}^n}^2_{\varepsilon} +  \varepsilon^{-1} \mu^{-1} \norm{\nabla \times r_{R/2,H}^{n + \frac{1}{2}}}^2_{\mu}  \\ 
 +\sum\limits_{\mathfrak{a} = 1}^{\frac{R}{2} - 1} \sum\limits_{k_\mathfrak{a} = 1}^{\frac{R}{2} - 1 - \sum\limits_{\mathfrak{b} = 1}^{\mathfrak{a} - 1} k_\mathfrak{b}} C_{k_1}^2 \mathfrak{f}(\mathfrak{a})^2 \left(\norm{\left( \nabla \nabla \cdot \right)^{\sum\limits_{\mathfrak{m} = 1}^\mathfrak{a} k_\mathfrak{m}} \nabla r_{\left( R/2 - \sum \limits_{m=1}^a k_m \right),p}^n}^2_{\varepsilon^{-1}} + \norm{\nabla \cdot (\nabla \nabla \cdot)^{\sum\limits_{\mathfrak{m} = 1}^\mathfrak{a} k_\mathfrak{m}} r_{\left( R/2 - \sum \limits_{m=1}^a k_m \right),E}^n }^2_\varepsilon  \right. \\ \left. + (\varepsilon \mu)^{-(2 \sum\limits_{\mathfrak{m} = 1}^\mathfrak{a} k_\mathfrak{m} + 1)} \norm{(\nabla \times \nabla \times)^{\sum\limits_{\mathfrak{m} = 1}^\mathfrak{a} k_\mathfrak{m}} \nabla \times r_{\left( R/2 - \sum \limits_{m=1}^a k_m \right),E}^n}^2_\varepsilon \right. \\ \left.+ (\mu \varepsilon)^{-(2 \sum\limits_{\mathfrak{m} = 1}^\mathfrak{a} k_\mathfrak{m} + 1)} \norm{ \nabla \times (\nabla \times \nabla \times)^{\sum\limits_{\mathfrak{m} = 1}^\mathfrak{a} k_\mathfrak{m}} r_{\left( R/2 - \sum \limits_{m=1}^a k_m \right),H}^{n+\frac{1}{2}}}^2_\mu \right) \Bigg].
\end{multline*}
We now apply the discrete Gronwall inequality with $\Delta t < \dfrac{1}{3(R+4)}$ to get that:
\begin{multline*}
  \norm{\eta_h^{N - \frac{1}{2}}}^2_{\varepsilon^{-1}} + \norm{\zeta_h^{N - \frac{1}{2}}}^2_{\varepsilon} + \norm{\xi_h^N}^2_{\mu} \le \Bigg[ \dfrac{6 \Delta t}{5} \sum\limits_{n = 0}^{N - 1} \Bigg(\sum \limits_{k=0}^{\frac{R}{2} - 1} \left(\dfrac{1}{(2k+1)!}\right)^2 \dfrac{\Delta t^{4k}}{2^{4k}} \left(\norm[\bigg]{(I - \Pi_h^0) \dfrac{\partial^{2k+1} p}{\partial t^{2k+1}}(t^n)}^2_{\varepsilon^{-1}} \right. \\ \left. + \norm[\bigg]{(I - \Pi_h^1) \dfrac{\partial^{2k+1} E}{\partial t^{2k+1}}(t^n)}^2_{\varepsilon} + \norm[\bigg]{(I - \Pi_h^2) \dfrac{\partial^{2k+1} H}{\partial t^{2k+1}}(t^{n + \frac{1}{2}})}^2_\mu\right) + 
  \norm[\bigg]{(I - \Pi_h^0) R_p^n}^2_{\varepsilon^{-1}} + \norm[\bigg]{(I - \Pi_h^1) R_E^n}^2_{\varepsilon} \\ + \norm[\bigg]{(I - \Pi_h^2) R_H^{n + \frac{1}{2}}}^2_\mu  + \norm{R_p^n}^2_{\varepsilon^{-1}} + \norm{R_E^n}^2_{\varepsilon}  + \norm{R_H^{n + \frac{1}{2}}}^2_{\mu} + \norm{\nabla r_{R/2,p}^n} ^2_{\varepsilon^{-1}} + \norm{\nabla \cdot r_{R/2,E}^n}^2_{\varepsilon} + \varepsilon^{-1} \mu^{-1} \norm{\nabla \times r_{R/2,E}^n}^2_{\varepsilon} \\ + \varepsilon^{-1} \mu^{-1} \norm{\nabla \times r_{R/2,H}^{n + \frac{1}{2}}}^2_{\mu} +\sum\limits_{\mathfrak{a} = 1}^{\frac{R}{2} - 1} \sum\limits_{k_\mathfrak{a} = 1}^{\frac{R}{2} - 1 - \sum\limits_{\mathfrak{b} = 1}^{\mathfrak{a} - 1} k_\mathfrak{b}} C_{k_1}^2 \mathfrak{f}(\mathfrak{a})^2 \left(\norm{\left( \nabla \nabla \cdot \right)^{\sum\limits_{\mathfrak{m} = 1}^\mathfrak{a} k_\mathfrak{m}} \nabla r_{\left( R/2 - \sum \limits_{m=1}^a k_m \right),p}^n}^2_{\varepsilon^{-1}} \right. \\ \left. + \norm{\nabla \cdot (\nabla \nabla \cdot)^{\sum\limits_{\mathfrak{m} = 1}^\mathfrak{a} k_\mathfrak{m}} r_{\left( R/2 - \sum \limits_{m=1}^a k_m \right),E}^n }^2_\varepsilon + (\varepsilon \mu)^{-(2 \sum\limits_{\mathfrak{m} = 1}^\mathfrak{a} k_\mathfrak{m} + 1)} \norm{(\nabla \times \nabla \times)^{\sum\limits_{\mathfrak{m} = 1}^\mathfrak{a} k_\mathfrak{m}} \nabla \times r_{\left( R/2 - \sum \limits_{m=1}^a k_m \right),E}^n}^2_\varepsilon \right. \\ \left. + (\mu \varepsilon)^{-(2 \sum\limits_{\mathfrak{m} = 1}^\mathfrak{a} k_\mathfrak{m} + 1)} \norm{ \nabla \times (\nabla \times \nabla \times)^{\sum\limits_{\mathfrak{m} = 1}^\mathfrak{a} k_\mathfrak{m}} r_{\left( R/2 - \sum \limits_{m=1}^a k_m \right),H}^{n+\frac{1}{2}}}^2_\mu \right) \Bigg) \\ + \dfrac{7}{5} \Big( \norm{\eta_h^0}^2_{\varepsilon^{-1}} + \norm{\zeta_h^0}^2_{\varepsilon} + \norm{\xi_h^0}^2_{\mu} \Big) \Bigg] \exp \left(2(R+4) T \right).
\end{multline*}
Using our estimates for the Taylor remainders from Theorem~\labelcref{thm:dscrt_error_estmt_lfR} for these terms on the right-hand side of the above inequality, we further get that:
\begin{multline*}
  \Delta t \sum\limits_{n = 0}^N \Big[ \norm{R^n_p}^2_{\varepsilon^{-1}} + \norm{R^n_E}^2_{\varepsilon} + \norm{R^{n + \frac{1}{2}}_H}^2_\mu \Big] \le \left(\Delta t \right)^{2R} \left[\norm[\bigg]{\dfrac{\partial^{R+1} p}{\partial t^{R+1}}}^2_{L^2(0, T; L^2_{\varepsilon^{-1}}(\Omega))} \!\! + \, \norm[\bigg]{\dfrac{\partial^{R+1} E}{\partial t^{R+1}}}^2_{L^2(0, T; L^2_\varepsilon(\Omega))} \!\! \right. \\ \left. + \, \norm[\bigg]{\dfrac{\partial^{R+1} H}{\partial t^{R+1}}}^2_{L^2(0, T; L^2_\mu(\Omega))} \right].
\end{multline*}
Now, for $2 \leqslant S \leqslant R$, $S$ even, $\ell > 0$, and $\Phi$ to be an appropriate operator acting on functions from correspondingly appropriate function spaces on $\Omega$ (as in the proof of Theorem~\labelcref{thm:dscrt_error_estmt_lfR}), we have that:
\begin{multline*}
  \Delta t \sum\limits_{n = 0}^{N - 1} \left(\Delta t\right)^{2(R-S)} \left[\norm{\Phi (r^n_{S/2,p})}^2_{\varepsilon^{-1}} + \left(\mu \varepsilon \right)^{-\ell} \norm{\Phi (r^n_{S/2,E})}^2_{\varepsilon} + \left(\mu \varepsilon \right)^{-\ell}  \norm{\Phi (r^{n + \frac{1}{2}}_{S/2,H})}^2_\mu \right] \\
  \le  \left( \Delta t \right)^{2R} \left[ \norm[\bigg]{\dfrac{\partial^S \left(\Phi (p) \right)}{\partial t^S}}^2_{L^2(0, T; L^2_{\varepsilon^{-1}}(\Omega))} \!\! + \, \left(\mu \varepsilon \right)^{-\ell} \norm[\bigg]{\dfrac{\partial^S \left(\Phi (E) \right)}{\partial t^S}}^2_{L^2(0, T; L^2_\varepsilon(\Omega))} \!\! + \, \left(\mu \varepsilon \right)^{-\ell} \norm[\bigg]{\dfrac{\partial^S (\Phi (H))}{\partial t^S}}^2_{L^2(0, T; L^2_\mu(\Omega))} \right].
\end{multline*}
Next, for $p \in \mathring{H}_{\varepsilon^{-1}}^1(\Omega)$, $E \in \mathring{H}_\varepsilon(\curl; \Omega)$ and $H \in \mathring{H}_\mu(\divgn; \Omega)$, there exists positive bounded constants $C_{0, p}$, $C_{0, E}$, $C_{0, H}$, $C_{m, p}$, $C_{m, E}$, $C_{m, H}$, for $0 \leqslant m \leqslant \dfrac{R}{2} -1$ such that we have the following error bounds for the $L^2$ projections:
\begin{alignat*}{3}
 \norm[\bigg]{(I - \Pi_h^0) R_p^n}_{\varepsilon^{-1}} & \le C_{0, p} h^r \norm[\bigg]{R_p^n}_{\varepsilon^{-1}}, && \qquad   \norm[\bigg]{(I - \Pi_h^0) \dfrac{\partial^{2m + 1} p}{\partial t^{2m + 1}}(t^n)}_{\varepsilon^{-1}} && \le C_{m, p} h^r \norm[\bigg]{\dfrac{\partial^{2m + 1} p}{\partial t^{2m + 1}}(t^n)}_{\varepsilon^{-1}},  \\
 \norm[\bigg]{(I - \Pi_h^1) R_E^n}_{\varepsilon} & \le C_{0, E} h^r \norm[\bigg]{R_E^n}_{\varepsilon}, && \qquad  \norm[\bigg]{(I - \Pi_h^1) \dfrac{\partial^{2m + 1} E}{\partial t^{2m + 1}}(t^n)}_{\varepsilon} && \le C_{m, E} h^r \norm[\bigg]{\dfrac{\partial^{2m + 1} E}{\partial t^{2m + 1}}(t^n)}_{\varepsilon}, \\
\norm[\bigg]{(I - \Pi_h^2) R_H^{n + \frac{1}{2}}}_{\mu} & \le C_{0, H} h^r \norm[\bigg]{R_H^{n + \frac{1}{2}}}_{\mu}, && \qquad   \norm[\bigg]{(I - \Pi_h^2) \dfrac{\partial^{2m + 1} H}{\partial t^{2m + 1}}(t^{n + \frac{1}{2}})}_{\mu} && \le C_{m, H} h^r \norm[\bigg]{\dfrac{\partial^{2m + 1} H}{\partial t^{2m + 1}}(t^{n + \frac{1}{2}})}_{\mu}.
\end{alignat*}
Set $C_0 \coloneq \max \{C_{0, p}, C_{0, E}, C_{0, H}, \max \limits_{0 \leqslant m \leqslant R/2} \{ C_{m, p}, C_{m, E}, C_{m, H} \} \}$. We therefore have that:
 \begin{multline*}
 \Delta t \sum\limits_{n = 0}^{N} \left[\sum \limits_{k=0}^{\frac{R}{2} - 1} \left(\dfrac{1}{(2k+1)!}\right)^2 \dfrac{\Delta t^{4k}}{2^{4k}} \left(\norm[\bigg]{(I - \Pi_h^0) \dfrac{\partial^{2k+1} p}{\partial t^{2k+1}}(t^n)}^2_{\varepsilon^{-1}} + \norm[\bigg]{(I - \Pi_h^1) \dfrac{\partial^{2k+1} E}{\partial t^{2k+1}}(t^n)}^2_{\varepsilon} \right.\right. \\ \left.\left. + \norm[\bigg]{(I - \Pi_h^2) \dfrac{\partial^{2k+1} H}{\partial t^{2k+1}}(t^{n + \frac{1}{2}})}^2_\mu\right)\right] \\
\le C_0 h^{2 r} \sum\limits_{n = 0}^{N-1} \Delta t \left[\sum \limits_{k=0}^{\frac{R}{2} - 1} \left(\dfrac{1}{(2k+1)!}\right)^2 \dfrac{\Delta t^{4k}}{2^{4k}} \left( \norm[\bigg]{\dfrac{\partial^{2k+1} p}{\partial t^{2k+1}}(t^n)}^2_{\varepsilon^{-1}} \!\! + \norm[\bigg]{\dfrac{\partial^{2k+1} E}{\partial t^{2k+1}}(t^n)}^2_{\varepsilon} \!\! + \norm[\bigg]{\dfrac{\partial^{2k+1} H}{\partial t^{2k+1}}(t^{n + \frac{1}{2}})}^2_{\mu}\right) \right] \\
\le C_0 h^{2 r} \sum \limits_{k=0}^{\frac{R}{2} - 1} \int\limits_0^T \left[ \norm[\bigg]{\dfrac{\partial^{2k+1} p}{\partial t^{2k+1}}(t^n)}^2_{\varepsilon^{-1}} \!\! + \norm[\bigg]{\dfrac{\partial^{2k+1} E}{\partial t^{2k+1}}(t^n)}^2_{\varepsilon} \!\! + \norm[\bigg]{\dfrac{\partial^{2k+1} H}{\partial t^{2k+1}}(t^{n + \frac{1}{2}})}^2_{\mu} \right] dt \\
= C_0 h^{2 r} \sum \limits_{k=0}^{\frac{R}{2} - 1} \left[ \norm[\bigg]{\dfrac{\partial^{2k+1} p}{\partial t^{2k+1}}}^2_{L^2(0, T; L^2_{\varepsilon^{-1}}(\Omega))} \!\! + \norm[\bigg]{\dfrac{\partial^{2k+1} E}{\partial t^{2k+1}}}^2_{L^2(0, T; L^2_{\varepsilon}(\Omega))} \!\! + \norm[\bigg]{\dfrac{\partial^{2k+1} H}{\partial t^{2k+1}}}^2_{L^2(0, T; L^2_{\mu}(\Omega))} \right],
\end{multline*}
and likewise for the $L^2$ projection of the remainder terms:
\begin{multline*}
  \Delta t \sum\limits_{n = 0}^N \left[ \norm[\bigg]{(I - \Pi_h^0) R_p^n}_{\varepsilon^{-1}} \!\! + \norm[\bigg]{(I - \Pi_h^1) R_E^n}_{\varepsilon} \!\! + \norm[\bigg]{(I - \Pi_h^2) R_H^{n + \frac{1}{2}}}_{\mu} \right], \\
  \le C_0 h^{2 r} \sum\limits_{n = 0}^N  \Delta t \left[ \norm{R_p^n}_{L^2_{\varepsilon^{-1}}(\Omega)} + \norm{R_E^n}_{L^2_{\varepsilon}(\Omega)} + \norm{R_H^{n + \frac{1}{2}}}_{L^2_{\mu}(\Omega)} \right], \\
  \le C_0 h^{2 r} (\Delta t)^{2R} \left[ \norm[\bigg]{\dfrac{\partial^{R+1} p}{\partial t^{R+1}}}^2_{L^2(0, T; L^2_{\varepsilon^{-1}}(\Omega))} \!\! + \norm[\bigg]{\dfrac{\partial^{R+1} E}{\partial t^{R+1}}}^2_{L^2(0, T; L^2_{\varepsilon}(\Omega))} \!\! + \norm[\bigg]{\dfrac{\partial^{R+1} H}{\partial t^{R+1}}}^2_{L^2(0, T; L^2_{\mu}(\Omega))} \right].
\end{multline*}
We next set $p_h^0 \coloneq \Pi_h^0 p_0$, $E_h^0 \coloneq \Pi_h^1 E_0$, and $H_h^0 \coloneq \Pi_h^2 H_0$, and note that for $0 \leqslant m \leqslant \dfrac{R}{2}$, $2 \leqslant S \leqslant R$, $S$ even, $\ell > 0$, and for $\Phi$ to be an appropriate operator (as before), there exists positive bounded constants $M_m$ and $M_S$ such that:
\begin{align*}
  \norm[\bigg]{\dfrac{\partial^{2m+1} p}{\partial t^{2m+1}}}^2_{L^2(0, T; L^2_{\varepsilon^{-1}}(\Omega))} \!\! + \norm[\bigg]{\dfrac{\partial^{2m+1} E}{\partial t^{2m+1}}}^2_{L^2(0, T; L^2_{\varepsilon}(\Omega))} \!\! + \norm[\bigg]{\dfrac{\partial^{2m+1} H^{2m+1}}{\partial t}}^2_{L^2(0, T; L^2_{\mu}(\Omega))} &\le M_m, \\
\norm[\bigg]{\dfrac{\partial^S \left(\Phi (p) \right)}{\partial t^S}}^2_{L^2(0, T; L^2_{\varepsilon^{-1}}(\Omega))} \!\! + \, \left(\mu \varepsilon \right)^{-\ell} \norm[\bigg]{\dfrac{\partial^S \left(\Phi (E) \right)}{\partial t^S}}^2_{L^2(0, T; L^2_\varepsilon(\Omega))} \!\! + \, \left(\mu \varepsilon \right)^{-\ell} \norm[\bigg]{\dfrac{\partial^S (\Phi (H))}{\partial t^S}}^2_{L^2(0, T; L^2_\mu(\Omega))} &\le M_S,
\end{align*}
Consequently, we have the following estimate:
\[
  \norm{\eta_h^{N-\frac{1}{2}}}^2_{\varepsilon^{-1}} + \norm{\zeta_h^{N-\frac{1}{2}}}^2_{\varepsilon} + \norm{\xi_h^N}^2_{\mu} \le \widetilde{C} \left[h^{2 r} + h^{2 r} (\Delta t)^{2R} + (\Delta t)^{2R} \right],
\]
where $\widetilde{C} = 2 \max\{C_0 M_m, M_m, M_S \} \exp(2(R+4) T)$ for $0 \leqslant m \leqslant \dfrac{R}{2}$, $2 \leqslant S \leqslant R$, $S$ even and then with the equivalence between $1$- and $2$-norms we get that:
\[
  \norm{\eta_h^{N-\frac{1}{2}}}_{\varepsilon^{-1}} + \norm{\zeta_h^{N-\frac{1}{2}}}_{\varepsilon} + \norm{\xi_h^N}_{\mu} \le C_1 \left[h^{r} + h^r (\Delta t)^R + (\Delta t)^R \right],
\]
where $C_1 = 2 \sqrt{\widetilde{C}}$. Also, there are positive bounded constants $\widetilde{C_2}$, $\widetilde{C_3}$ and $\widetilde{C_4}$ such that:
\begin{alignat*}{4}
  \norm{\eta^{N-\frac{1}{2}}}_{\varepsilon^{-1}} &= \norm{(I - \Pi_h^0) p(t^{N-\frac{1}{2}})}_{\varepsilon^{-1}} &&\le \widetilde{C_2} h^r \norm{p(t^{N-\frac{1}{2}})}_{L^2_{\varepsilon^{-1}}(\Omega)} &&\implies \norm{\eta^{N-\frac{1}{2}}}_{\varepsilon^{-1}} &&\le C_2 h^r, \\
  \norm{\zeta^{N-\frac{1}{2}}}_{\varepsilon} &= \norm{(I - \Pi_h^1) E(t^{N-\frac{1}{2}})}_{\varepsilon} &&\le \widetilde{C_3} h^r \norm{E(t^{N-\frac{1}{2}})}_{L^2_{\varepsilon}(\Omega)} &&\implies \norm{\zeta^{N-\frac{1}{2}}}_{\varepsilon} &&\le C_3 h^r, \\
  \norm{\xi^N}_\mu &= \norm{(I - \Pi_h^2) H(t^N)}_\mu &&\le \widetilde{C_4} h^r \norm{H(t^N)}_{L^2_{\mu}(\Omega)} &&\implies \norm{\xi^N}_\mu &&\le C_4 h^r,
\end{alignat*}
in which $C_2 = \widetilde{C_2} \norm{p(t^{N-\frac{1}{2}})}_{L^2_{\varepsilon^{-1}}(\Omega)}$, $C_3 = \widetilde{C_3} \norm{E(t^{N-\frac{1}{2}})}_{L^2_{\varepsilon}(\Omega)}$ and $C_4 = \widetilde{C_4} \norm{H(t^N)}_{L^2_{\mu}(\Omega)}$ are all bounded positive constants due to Theorem~\ref{thm:dscrt_enrgy_estmt_lfR}. Finally, this provides us with our desired result by choosing $C = C_1 + C_2 + C_3 + C_4$:
\begin{align*}
  \norm{e_{p_h}^{N - \frac{1}{2}}}_{\varepsilon^{-1}} + \norm{e_{E_h}^{N - \frac{1}{2}}}_{\varepsilon} + \norm{e_{H_h}^{N}}_{\mu} &= \norm{\eta^{N - \frac{1}{2}} - \eta^{N - \frac{1}{2}}_h}_{\varepsilon^{-1}} + \norm{\zeta^{N - \frac{1}{2}} - \zeta^{N - \frac{1}{2}}_h}_{\varepsilon} + \norm{\xi^N - \xi^N_h}_\mu \\
  &\le C \left[ (\Delta t )^R + h^r + h^r (\Delta t)^R \right]. \qedhere
\end{align*}
\end{proof}

\section{Numerical Results and Summary}\label{sec:numerics}

We present two examples in $\R^2$ as stated below, and demonstrate their numerical solution using LF$_6$ in conjunction with linear and quadratic Whitney simplicial finite elements on a triangulation of $\Omega$. Our examples are both specified on the domain $\Omega$ which is a unit square in $\R^2$ and we compute the solution for time $t \in [0, 1]$.

\medskip \noindent \textbf{Example 1}: We consider the problem on $\Omega$ with analytical solutions, and material parameters as below.
\[
  p = 0, \quad E = %
  \begin{bNiceMatrix}
    \sin \pi y \cos \pi t \\
    \sin \pi x \cos \pi t
  \end{bNiceMatrix}, %
  \quad H = (\cos \pi y - \cos \pi x) \sin \pi t,
\]
\[\epsilon = 1, \quad \mu = 1.\]
The results of solving this problem with LF$_6$, and with linear and quadratic Whitney basis are shown in Figures~\labelcref{fig:example1_lf6_linear,fig:example1_lf6_quadratic}, respectively.

\medskip \noindent \textbf{Example 2}: This problem is also on $\Omega$ but with $p \ne 0$. The boundary conditions for $p$ and $E$ are nonhomogeneous, and they are homogeneous for $H$. The analytical solutions and the material paramaters are as below.
\[
  p = \left(\cos \pi x + \cos \pi y\right) \sin \pi t,
\]
\[
 E = %
  \begin{bNiceMatrix}
    \sin \pi (\sqrt{2} t - x - y) - \sin \pi x \cos \pi t \\
    -\sin \pi (\sqrt{2} t - x - y) - \sin \pi y \cos \pi t
  \end{bNiceMatrix}, %
  \quad H = -\sqrt{2} \sin \pi (\sqrt{2} t - x - y),
\]
\[\epsilon = 1, \quad \mu = 1.\]
The results of solving this problem with LF$_6$, and with linear and quadratic Whitney basis are shown in Figures~\labelcref{fig:example2_lf6_linear,fig:example2_lf6_quadratic}, respectively.

We wish to conclude by stating that we have delineated an arbitrarily higher (even) order time discretization scheme LF$_R$ for the $(p, E, H)$ Maxwell's systems in $\R^2$ and $\R^3$. The LF$_R$ scheme is implicit and uses two staggered time grids, and therefore is in the category of leapfrog schemes. We have analysed LF$_R$ in the full generality in conjunction with compatible and arbitrary-order de Rham finite elements spanned by Whitney basis. In doing so, we have also demonstrated its energy conservation and the rate of convergence of its error in terms of the time step $\Delta t$, mesh parameter $h$, time order $R$, and spatial polynomial order $r$. Our work is thus comprehensive in both defining a generalized scheme and performing its error analysis. We hope that as future work our method can be extended to other instances of this Maxwell's system as it applies to, for example, scattering and reflection problems, and other similar interesting real-world problems in computational electromagnetics. We also hope that our work would be beneficial for the paradigm of scientific machine learning especially Physics-informed neural networks for Maxwell's equations.



\printbibliography

\begin{figure}[htb]
  \centering
  \includegraphics[scale=1]{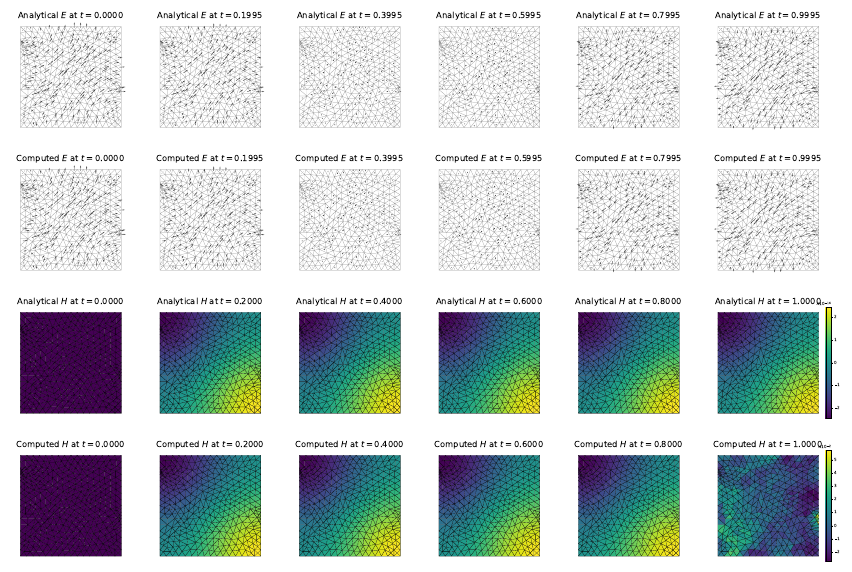}
  \caption{\textbf{Linear finite elements with LF$_6$}: Plots of solutions $(E_h^{n - 1/2}, H_h^n)$ at different times for \textbf{Example 1} of Section~\ref{sec:numerics} using LF$_6$ and linear Whitney forms for the FEEC spaces. The solutions for $p$ are not shown due to them being identically $0$. We also show projections of the analytical solutions to the linear Whitney forms spaces at these time steps for a visual comparison.}
  \label{fig:example1_lf6_linear}
\end{figure}

\begin{figure}[htb]
  \centering
  \includegraphics[scale=1]{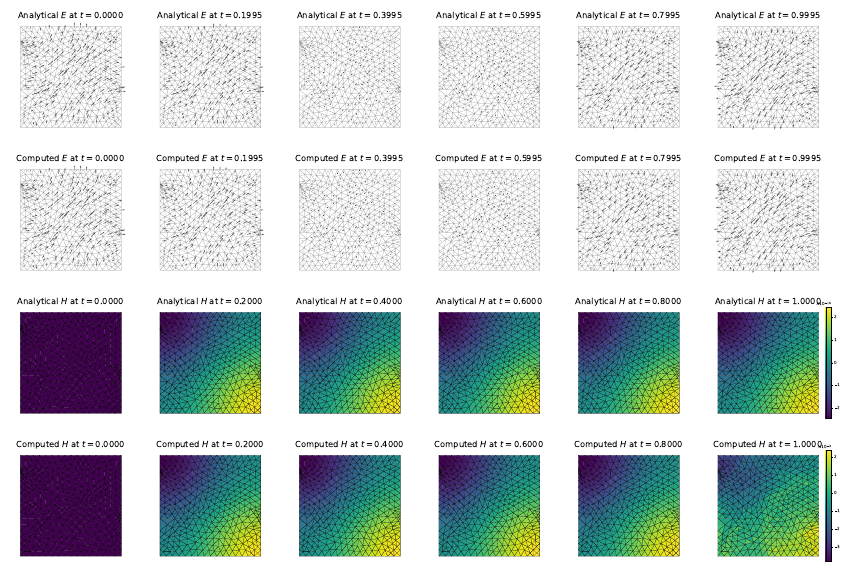}
  \caption{\textbf{Quadratic finite elements with LF$_6$}: Plots of solutions $(E_h^{n - 1/2}, H_h^n)$ at different times for \textbf{Example 1} of Section~\ref{sec:numerics} using LF$_6$ and quadratic Whitney forms in FEEC. The solutions for $p$ are again not shown due to them being identically $0$. We also show projections of the analytical solutions to the quadratic Whitney forms spaces at these time steps for a visual comparison.}
  \label{fig:example1_lf6_quadratic}
\end{figure}

\begin{figure}[htb]
  \centering
  \includegraphics[scale=1]{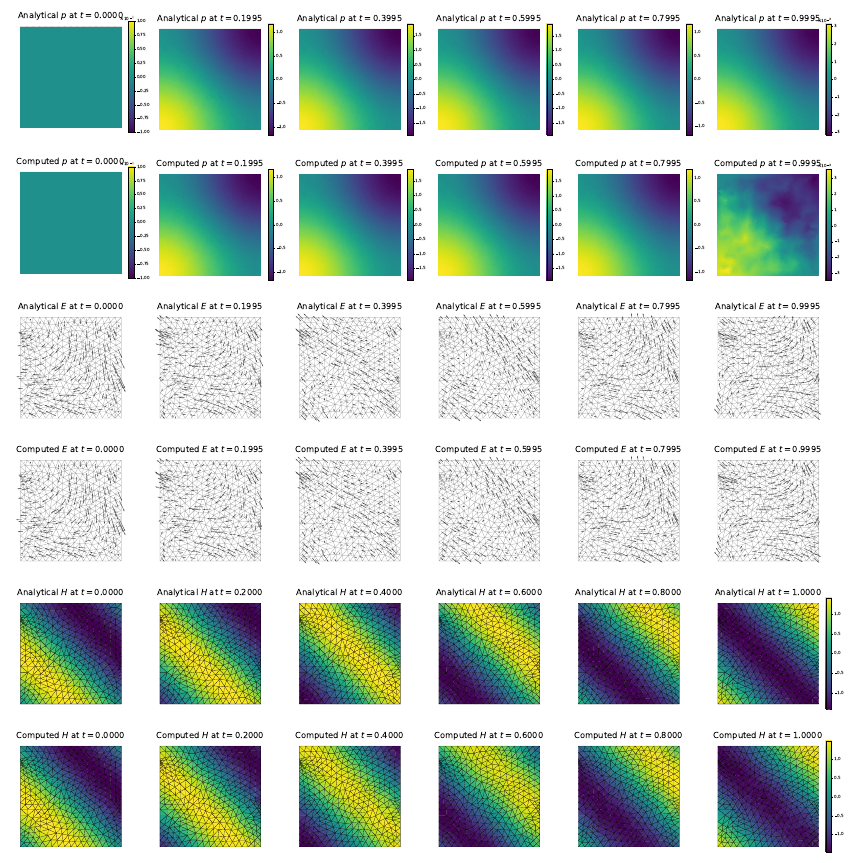}
  \caption{\textbf{Linear finite elements with LF$_6$}: Plots of solutions $(p_h^{n + 1/2}, E_h^{n - 1/2}, H_h^n)$ at different times for \textbf{Example 2} of Section~\ref{sec:numerics} using the LF$_6$ and linear Whitney forms in FEEC. We also show projections of the analytical solutions to the linear Whitney forms spaces at these time steps for a visual comparison.}
  \label{fig:example2_lf6_linear}
\end{figure}

\begin{figure}[htb]
  \centering
  \includegraphics[scale=1]{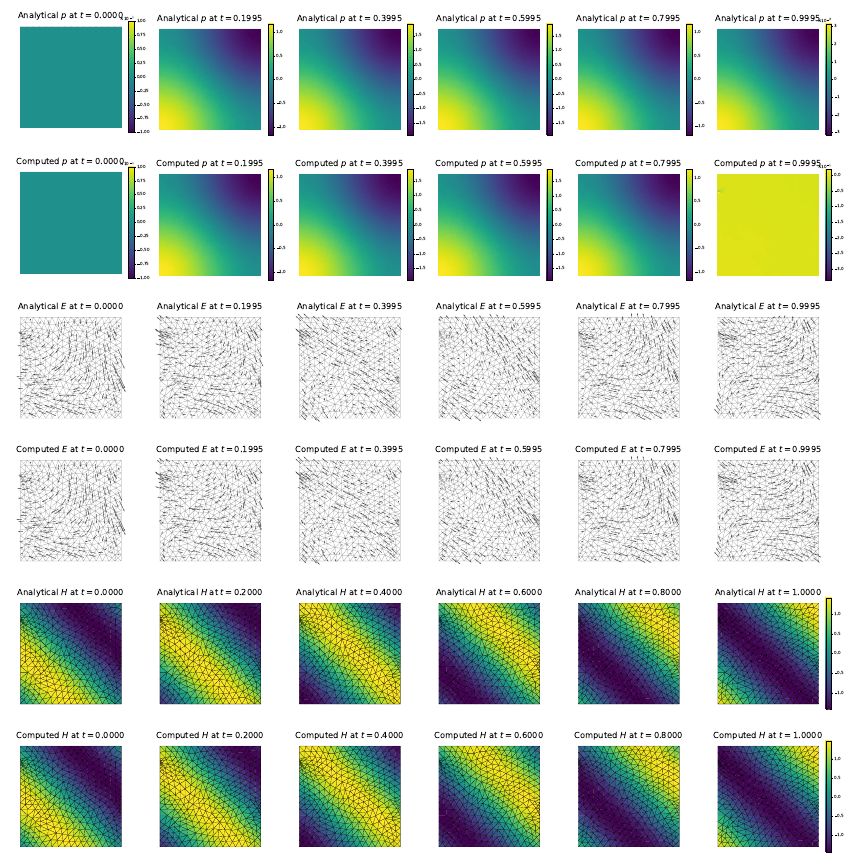}
  \caption{\textbf{Quadratic finite elements with LF$_6$}: Plots of solutions $(p_h^{n - 1/2}, E_h^{n - 1/2}, H_h^n)$ at different times for \textbf{Example 2} of Section~\ref{sec:numerics} using the LF$_6$ and quadratic Whitney forms in FEEC. We also show projections of the analytical solutions to the quadratic Whitney forms spaces at these time steps for a visual comparison.}
  \label{fig:example2_lf6_quadratic}
\end{figure}

\end{document}